\documentclass[a4paper]{amsart}
\usepackage{txfonts, amsmath,amstext,amsthm,amscd,amsopn,verbatim,amssymb, amsfonts}
\usepackage{fullpage}

\usepackage[makeroom]{cancel}

\usepackage{float}
\restylefloat{table}

\usepackage[bbgreekl]{mathbbol}
\usepackage{tikz}
\usepackage{tikz-cd}
\usetikzlibrary{matrix}
\usetikzlibrary{shapes}
\usetikzlibrary{arrows}
\usetikzlibrary{calc,3d}
\usetikzlibrary{decorations,decorations.pathmorphing,decorations.pathreplacing,decorations.markings}
\usetikzlibrary{through}

\tikzset{ext/.style={circle, draw,inner sep=1pt},int/.style={circle,draw,fill,inner sep=1.4pt},nil/.style={inner sep=1pt}}
\tikzset{cy/.style={circle,draw,fill,inner sep=2pt},scy/.style={circle,draw,inner sep=2pt},scyx/.style={draw,cross out,inner sep=2pt},scyt/.style={draw,regular polygon,regular polygon sides=3,inner sep=0.95pt}}
\tikzset{exte/.style={circle, draw,inner sep=3pt},inte/.style={circle,draw,fill,inner sep=3pt}}
\tikzset{diagram/.style={matrix of math nodes, row sep=3em, column sep=2.5em, text height=1.5ex, text depth=0.25ex}}
\tikzset{diagram2/.style={matrix of math nodes, row sep=0.5em, column sep=0.5em, text height=1.5ex, text depth=0.25ex}}

\tikzset{
  rightblue/.style={
    decoration={markings,mark=at position .8 with {\arrow[scale=1.2,blue]{latex}}},
    postaction={decorate},
    shorten >=0.4pt}}
\tikzset{
  leftblue/.style={
    decoration={markings,mark=at position .6 with {\arrowreversed[scale=1.2,blue]{latex}}},
    postaction={decorate},
    shorten >=0.4pt}}
\tikzset{
  rightred/.style={
    decoration={markings,mark=at position .4 with {\arrow[scale=1.2,red]{latex}}},
    postaction={decorate},
    shorten >=0.4pt}}
\tikzset{
  leftred/.style={
    decoration={markings,mark=at position .2 with {\arrowreversed[scale=1.2,red]{latex}}},
    postaction={decorate},
    shorten >=0.4pt}}
    
\tikzset{
  crossed/.style={
    decoration={markings,mark=at position .5 with {\arrow{|}}},
    postaction={decorate},
    shorten >=0.4pt}}

\newcommand{\notadp}{{
\begin{tikzpicture}[baseline=-.55ex,scale=.2, every loop/.style={}]
 \node[circle,draw,fill,inner sep=.5pt] (a) at (0,0) {};
 \draw (a) edge[loop] (a);
 \draw (-.2,-.2) -- (.2,.5);
\end{tikzpicture}}}

\newcommand{\Ed}{{
\begin{tikzpicture}[baseline=-.8ex,scale=.5]
\node[nil] (a) at (0,0) {};
\node[nil] (b) at (1,0) {};
\draw (a) edge[-latex] (b);
\end{tikzpicture}}}

\newcommand{\dE}{{
\begin{tikzpicture}[baseline=-.8ex,scale=.5]
\node[nil] (a) at (0,0) {};
\node[nil] (b) at (1,0) {};
\draw (a) edge[latex-] (b);
\end{tikzpicture}}}

\newcommand{\EdE}{{
\begin{tikzpicture}[baseline=-.65ex,scale=.5]
 \node[nil] (a) at (0,0) {};
 \node[int] (b) at (1,0) {};
 \node[nil] (c) at (2,0) {};
 \draw (a) edge[-latex] (b);
 \draw (b) edge[latex-] (c);
\end{tikzpicture}}}

\newcommand{\dEd}{{
\begin{tikzpicture}[baseline=-.65ex,scale=.5]
 \node[nil] (a) at (0,0) {};
 \node[int] (b) at (1,0) {};
 \node[nil] (c) at (2,0) {};
 \draw (a) edge[latex-] (b);
 \draw (b) edge[-latex] (c);
\end{tikzpicture}}}

\newcommand{\Es}{{
\begin{tikzpicture}[baseline=-.8ex,scale=.5]
\node[nil] (a) at (0,0) {};
\node[nil] (b) at (1,0) {};
\draw (a) edge[->] (b);
\end{tikzpicture}}}

\newcommand{\Ess}{{
\begin{tikzpicture}[baseline=-.65ex,scale=.5]
 \node[nil] (a) at (0,0) {};
 \node[nil] (c) at (1.4,0) {};
 \draw (a) edge[crossed,->] (c);
\end{tikzpicture}}}

\newcommand{\Esss}{{
\begin{tikzpicture}[baseline=-.65ex,scale=.5]
 \node[nil] (a) at (0,0) {};
 \node[nil] (c) at (1.8,0) {};
 \draw (a) edge[->] (c);
 \draw (.6,.15) edge (.6,-.15);
 \draw (1.2,.15) edge (1.2,-.15);
\end{tikzpicture}}}

\newcommand{\ET}{{
\begin{tikzpicture}[baseline=-.65ex,scale=.5]
 \node[int] (a) at (0,0) {};
 \node[int] (c) at (1,0) {};
 \draw (a) edge[very thick,->] (c);
\end{tikzpicture}}}

\newcommand{\mxto}[1]{{\overset{#1}{\longmapsto}}}

\usepackage{chngcntr}
\counterwithin{figure}{section}
\theoremstyle{plain}
  \newtheorem{thm}[figure]{Theorem}
  \newtheorem{defi}[figure]{Definition}
  \newtheorem{prop}[figure]{Proposition}
  
  \newtheorem{cor}[figure]{Corollary}
  
  \newtheorem{lemma}[figure]{Lemma}
\theoremstyle{definition}
  \newtheorem{ex}[figure]{Example}
  \newtheorem{rem}[figure]{Remark}

\newcommand{\K}{{\mathbb{K}}}
\newcommand{\Z}{{\mathbb{Z}}}
\newcommand{\N}{{\mathbb{N}}}

\newcommand{\GC}{\mathrm{GC}}
\newcommand{\fGC}{\mathrm{fGC}}
\newcommand{\dGC}{\mathrm{dGC}}
\newcommand{\MGC}{\mathrm{MGC}}

\newcommand{\fGCc}{\mathrm{fGCc}}

\newcommand{\mV}{\mathrm{V}}
\newcommand{\mE}{\mathrm{E}}
\newcommand{\mB}{\mathrm{B}}
\newcommand{\mD}{\mathrm{D}}
\newcommand{\mO}{\mathrm{O}}
\newcommand{\mS}{\mathrm{S}}
\newcommand{\mM}{\mathrm{M}}

\DeclareMathOperator{\sgn}{sgn}

\newcommand{\gra}{\mathrm{gra}}
\newcommand{\sgra}{\mathrm{sgra}}
\newcommand{\Gra}{\mathrm{Gra}}
\newcommand{\sGra}{\mathrm{sGra}}
\newcommand{\grac}{\mathrm{grac}}
\newcommand{\Grac}{\mathrm{Grac}}

\begin{document}
\title{Multi-directed graph complexes and quasi-isomorphisms between them II: Sourced graphs}

\author{Marko \v Zivkovi\' c}
\address{Mathematics Research Unit\\ University of Luxembourg\\ Grand Duchy of Luxembourg}
\email{marko.zivkovic@uni.lu}


\keywords{Graph Complexes}

\begin{abstract}
We prove that the inclusion from oriented graph complex into graph complex with at least one source is a quasi-isomorphism, showing that homology of the ``sourced'' graph complex is also equal to the homology of standard Kontsevich's graph complex. This result may have applications in theory of multi-vector fields $T_{\rm poly}^{\geq 1}$ of degree at least one, and to the hairy graph complex which computes the rational homotopy of the space of long knots.
The result is generalized to multi-directed graph complexes, showing that all such graph complexes are quasi-isomorphic.
These complexes play a key role in the deformation theory of multi-oriented props recently invented by Sergei Merkulov.
We also develop a theory of graph complexes with arbitrary edge types.
\end{abstract}

\maketitle

\section{Introduction}

Generally speaking, graph complexes are graded vector spaces of formal linear combinations of isomorphism classes of some kind of graphs. Each of graph complexes plays a certain role in a subfield of homological algebra or algebraic topology. They have an elementary and simple combinatorial definition, yet we know very little about what their cohomology actually is.

The broad kind of results in graph-complex theory are results of quasi-isomorphism between various graph complexes. Even if we do not know the homology of any of them, the result that their homologies are the same may be useful and may connect various fields of mathematics.

The basic graph complexes we study in this paper are the following.
\begin{itemize}
\item Kontsevich's graph complex $\GC_n$, introduced  by M.\ Kontsevich in \cite{Kont1}, \cite{Kont2}.
\item Directed graph complex $\mD\GC_n$ (in literature often denoted by $\dGC_n$) of graphs with directed edges, mentioned e.g.\ in \cite[Appendix K]{grt}.
\item Oriented graph complex $\mO\GC_n$ (in literature often denoted by $\GC_n^{or}$) of directed graphs without a loop along directed edges, studied e.g.\ in \cite{oriented}.
\item Sourced graph complex $\mS\GC_n$ of directed graphs with at least one source, i.e.\ vertex with only outgoing edges. Note that no-loop condition implies that there is a source, so $\mO\GC_n\subset\mS\GC_n$.
\item Multi-directed oriented and sourced graph complex $\mO_j\mS_k\mD_\ell\GC_n$. It is the complex spanned by graphs whose edges have $j+k+l$ independent directions, $j$ of which make the graph oriented and $k$ of which make the graph sourced.
\end{itemize}

Not yet fully known, the homology of Kontsevich's graph complex is studied the most. Willwacher provided the explicit inclusion $\GC_n\rightarrow\mD\GC_n$ that is a quasi-isomorphism in \cite[Appendix K]{grt}. In \cite{oriented} Willwacher also showed that $\GC_n$ is quasi-isomorphic to $\mO\GC_{n+1}$. The same result is the consequence of the broader theory of Merkulov and Willwacher, \cite[6.3.8.]{MW}. Therefore, homologies of Kontsevich's, directed and oriented graph complexes are essentially the same.

This paper is the second in the series about multi-directed graph complexes and quasi-isomorphisms between them. In the first paper, \cite{Multi}, we presented the third, more direct proof that $\GC_n$ is quasi-isomorphic to $\mO\GC_{n+1}$ by providing an explicit quasi-isomorphism, and we generalised it to all multi-directed complexes.

In this paper we develop the techniques to get a new result, to prove that sourced graph complex has the same homology too. Again, we generalize the result to all multi-directed versions. We do it by providing explicit quasi-isomorphisms, see part (3) of Theorem \ref{thm:main} and Corollary \ref{cor:main}. For completeness, already known results about directed (\cite[Appendix K]{grt}) and oriented (\cite{Multi}) complexes are recalled and generalized to multi-directed case in parts (1) and (2) of Theorem \ref{thm:main}. The proof for oriented case is slightly more conceptual than in \cite{Multi}.

\begin{thm}
\label{thm:main}
For every integers $j,k,\ell\geq 0$ and $n\in\Z$
\begin{enumerate}
\item there is a quasi-isomorphism $g:\mO_j\mS_k\mD_\ell\GC_n\rightarrow\mO_j\mS_k\mD_{l+1}\GC_n$,
\item there is a quasi-isomorphism $h:\mO_j\mS_k\mD_\ell\GC_n\rightarrow\mO_{j+1}\mS_k\mD_\ell\GC_{n+1}$,
\item the inclusion $\mO_{j+1}\mS_k\mD_\ell\GC_{n+1}\hookrightarrow\mO_j\mS_{k+1}\mD_\ell\GC_{n+1}$ is a quasi-isomorphism.
\end{enumerate}
\end{thm}

\begin{cor}
\label{cor:main}
For every integers $j,k,\ell\geq 0$ and $n\in\Z$ there is a quasi-isomorphism
$$
\GC_n\rightarrow\mO_j\mS_k\mD_\ell\GC_{n+j+k}.
$$
\end{cor}

Among other things, Kontsevich's graph complex $\GC_2$ acts on multi-vector fields $T_{\rm poly}$. Homology of $\GC_2$ captures all the Lie algebra homology of $T_{\rm poly}$ in the stable limit, see \cite{Poly}, c.f.\ \cite{Kont1}. The sourced graph complex similarly acts on the multi-vector fields $T_{\rm poly}^{\geq 1}$ of degree at least one, i.e.\ there is at least one vector in the multi-vector. It is expected that its homology similarly captures all the Lie algebra homology of $T_{\rm poly}^{\geq 1}$. It has not yet been studied since we did not know the homology of sourced graph complex, but our result may be a motivation to write it down.

Another use of sourced graph complexes is giving the universal algebraic structures on the hairy graph complex. These complexes compute the rational homotopy of the spaces of embeddings of disks modulo immersions, fixed at the boundary, $\overline{\mathrm{Emb}_\partial}(\mathbb{D}^m,\mathbb{D}^n)$, provided that $n-m\geq 3$, cf. \cite{FTW, AT, Tur1,DGC2}.
In this context, Willwacher conjectured our result in \cite{Will}.\footnote{This paper mentions a graph complex with at lest one sink. However, sunk and sourced graph complexes are the same, just with reversed directions.} Therefore, his Theorem 1.4 can be extended to sourced graph complex.

On the other side, multi-oriented graph complexes play a key role in the deformation theory of multi-oriented props recently invented by Sergei Merkulov in \cite{Merk1} and to appear in \cite{Merk2}.
In his papers Merkulov gives a meaning to the extra directions, and provides interesting applications and representations of multi-oriented props.

To make the arguments smoother, in Section \ref{s:gc} we develop a theory of graph complexes with more edge types modelled by arbitrary dg $S_2$-module $\Sigma$. The theory may have broader applications.

In the proof of Theorem \ref{thm:main} we immediately prove the most general result for arbitrary $j$, $k$ and $\ell$. If the reader is not interested in multi-directed case, he can easily read the paper assuming $j=k=\ell=0$ and skip some thoughts and imagination about extra directions.

\subsection*{Outline of the paper}

There are many technical results in the paper, especially concerning signs. Therefore it is useful to give briefly the ideas and strategy of proving.

In Section \ref{s:gc} we define complexes used in this paper. The definition is similar to the standard definition of graph complexes with edge contraction as a differential, but there are two specialities.

Firstly, we allow edges to have more types. \emph{Core graphs} defined in Subsection \ref{ss:cg} are used as frames to attach edge types in Subsection \ref{ss:sog}. There is another differential defined in Subsection \ref{ss:ed}, called \emph{edge differential}, that only changes the type of an edge, without changing the core graph. Simple spectral sequence on the number of vertices leaves this differential as the first differential. On this first page, since the edge differential does not change the number of edges and vertices, the homology commutes with permuting them. This facts are helpful in many proofs.

The second speciality is introducing multiple directions to the core graphs in Subsection \ref{ss:mdg}. Directions are labelled by colours, and every edge has defined direction in every colour. 
In Subsection \ref{ss:saog} we define oriented and sourced graphs. A graph is \emph{oriented} in a particular colour if there is no loop along edges in that colour, and it is \emph{sourced} if there is at least one vertex (source) with all its adjacent edges having direction in that colour away of the vertex.
The last subsection, \ref{ss:scwed}, defines the \emph{$j$-oriented $k$-sourced $\ell$-directed full graph complex} $\mO_j\mS_k\mD_\ell\fGC_n$ that has one type of edges. For $j=k=
\ell=0$ it is the well known Kontsevich's full graph complex.


Section \ref{s:Simpl} simplifies complexes to prepare for the main proof of Theorem \ref{thm:main} in Section \ref{s:qi}. Recall \cite[Proposition 3.4]{grt}: it says that homology is essentially not changed if we disallow 1-valent and 2-valent vertices in Kontsevich's graph complex. Subsections \ref{ss:al2vv} and \ref{ss:npv} give the analogous result for $\mO_j\mS_k\mD_\ell\fGC_n$, and introduce the standard versions $\mO_j\mS_k\mD_\ell\GC_n$ that are subject of Theorem \ref{thm:main}. Here, instead of just 2-valent vertices, we disallow \emph{passing vertices}, i.e.\ 2-valent vertices that have an income and outcome in every colour.

In Subsection \ref{ss:gcwsd} we introduce an equivalent view of a coloured direction. Instead of using one colour, we can model it with edge types $\Ed$ and $\dE$, that have \emph{type direction}. We will use this view on the complexes $\mO_j\mS_k\mD_{\ell+1}\GC_n$, $\mO_{j+1}\mS_k\mD_\ell\GC_{n+1}$ and $\mO_j\mS_{k+1}\mD_\ell\GC_{n+1}$ from Theorem \ref{thm:main}. Like this, all complexes mentioned in the theorem have the same number of colours, and hence the same core graphs, with the only difference being in the edge types.

Further equivalent view of graph complexes with type direction is introduced in Subsection \ref{ss:sgc}. Here we fully use the theory of arbitrary edge types. Vertices that are not passing in remaining colours (in the simplest case $j=k=\ell=0$ that are more than 2-valent vertices) are called \emph{skeleton edges}. Strings of edges and passing vertices between skeleton vertices form a \emph{skeleton edge}. A graph is equivalently seen as a graph with skeleton vertices as vertices, and skeleton edges as edges, with an edge differential well adjusted.

Subsection \ref{ss:scwbse} removes unnecessarily long skeleton edges in skeleton graphs. The result is again of the similar nature as \cite[Proposition 3.4]{grt} where skeleton edges longer than 1 are removed. We do the same in the case of directed complex and prove the first part of Theorem \ref{thm:main} already here. In the other cases we can not go down to the length 1 because some low-length skeleton edges are needed to define oriented and sourced complexes. Oriented and sourced complexes with shortest possible edges are $\mO^1\mO_j\mS_k\mD_\ell\GC_{n+1}$ and $\mS^2\mO_j\mS_k\mD_\ell\GC_{n+1}$ respectively. A part of the main idea of this paper is handling these low-length skeleton edges.

In the last, Section \ref{s:qi}, we first define the map $h:\mO_j\mS_k\mD_\ell\GC_n\rightarrow\mO^1\mO_j\mS_k\mD_\ell\GC_{n+1}$ in Subsection \ref{ss:tm} as follows:
$$
h(\Gamma):=\sum_{x\in V(\Gamma)}(v(x)-2)\sum_{\tau\in S(\Gamma)}h_{x,\tau}(\Gamma),
$$
where sums go through all vertices $x$ of $\Gamma$ and all spanning trees $\tau$ of $\Gamma$, and $v(x)$ is the valence of $x$. The graph $h_{x,\tau}(\Gamma)$ is the graph with the same core graph as $\Gamma$, edges of $\tau$ having type $\Ed$ away from the vertex $x$, and other edges having type $\EdE-\dEd$. This construction clearly ensures that the resulting graph is oriented.
Much of the subsection deals with defining correct signs such that the map is well defined on complexes with permuted edges and vertices, and such that it commutes with the differential.
A ``conference version'' of this definition that skips details on signs could be much shorter.

Finally, in Subsection \ref{ss:tp} we show that the maps
$$
h:\mO_j\mS_k\mD_\ell\GC_n\rightarrow
\mO^1\mO_j\mS_k\mD_\ell\GC_{n+1}\hookrightarrow
\mS^2\mO_j\mS_k\mD_\ell\GC_{n+1}
$$
are quasi-isomorphisms, finishing the proof of Theorem \ref{thm:main}. Already mentioned spectral sequence on the number of vertices and the fact that the edge differential commutes with permuting edges and vertices reduces the claim to
$$
h:\left\langle\{\Phi\}\right\rangle\rightarrow
\mO\Phi\hookrightarrow
\mS\Phi
$$
being quasi-isomorphisms for every core graph $\Phi$. Here, $\left\langle\{\Phi\}\right\rangle$ is one-dimensional complex spanned by core graph $\Phi$, $\mO\Phi$ is sub-complex of $\mO^1\mO_j\mS_k\mD_\ell\GC_{n+1}$ with $\Phi$ as a core graph and before all permutations (distinguishable edges, vertices and edge directions), and $\mS\Phi$ is sub-complex of $\mS^2\mO_j\mS_k\mD_\ell\GC_{n+1}$ with $\Phi$ as a core graph and before all permutations. Latter two complexes are more than one-dimensional, but we introduce auxiliary one-dimensional complexes $\mO\Phi^{v-1}$ and $A^{e-v+1}$, and maps as on the following commutative diagram.
$$
\begin{tikzcd}
\left\langle\{\Phi\}\right\rangle
\arrow{r}{h}
\arrow[swap]{dr}{}
& \mO\Phi
\arrow[hookrightarrow]{r}{}
\arrow{d}{f}
& \mS\Phi
\arrow{d}{g} \\
& \mO\Phi^{v-1}
\arrow[hookrightarrow]{r}{}
\arrow[swap]{dr}{}
& \mS\Phi^{v-1}
\arrow{d}{p} \\
& & A^{e-v+1}
\end{tikzcd}
$$
The proof is finished by proving that vertical and diagonal maps are quasi-isomorphisms, implying that all other maps are quasi-isomorphisms too. Since diagonal complexes are one-dimensional, diagonal maps are quasi-isomorphisms if they are not zero. That would be trivial to show if, again, we did not have to check the signs.

\subsection*{Acknowledgements}

I thank Sergei Merkulov and Thomas Willwacher for providing the motivation for this work and some fruitful discussions.

\section{Graph complexes}
\label{s:gc}

In this section we make the set up for general definition of graph complexes with more than one type of edges. In the last subsection we particularly define Kontsevich's graph complex and graph complex $\mO_j\mS_k\mD_\ell\fGC_n$ for $n\in\Z$, $j,k,\ell\geq 0$, called \emph{$j$-oriented $k$-sourced $\ell$-directed graph complex}.

We work over a field $\K$ of characteristic zero. All vector spaces and differential graded vector spaces are assumed to be $\K$-vector spaces.
For a set $A$, the vector space of $\K$-linear combinations of elements of $A$ is denoted by $\langle A \rangle$.

\subsection{Core graphs}
\label{ss:cg}

\begin{defi}[Core graph]
Let $v>0$ and $e\geq 0$ be integers.
Let $V:=\{1,2,\dots,v\}$ be set of vertices and $E:=\{1,2,\dots,e\}$ be set of edges.

A \emph{core graph} $\Gamma$ with $v$ vertices and $e$ edges is a map $\Gamma=(\Gamma_-,\Gamma_+):E\rightarrow V^2$.

We say that an edge $a\in E$ connects vertices $\Gamma_-(a)$ and $\Gamma_+(a)$. The direction from $\Gamma_-(a)$ to $\Gamma_+(a)$ is called the \emph{core direction} of the edge $a$.

By $\bar\mV_v\bar\mE_e\gra$ we denote the set of all core graphs with $v$ vertices and $e$ edges.
We set $\bar\mV_v\bar\mE_e\Gra:=\langle\bar\mV_v\bar\mE_e\gra\rangle$ and
\begin{equation}
\bar\mV\bar\mE\gra:=\bigcup_{v,e}\bar\mV_v\bar\mE_e\gra,
\end{equation}
\begin{equation}
\bar\mV\bar\mE\Gra:=\bigoplus_{v,e}\bar\mV_v\bar\mE_e\Gra.
\end{equation}
\end{defi}

The bar on $\mV$ and $\mE$ stands for distinguishable vertices, respectively edges. It will disappear for spaces defined after permuting edges and vertices in Subsections \ref{ss:pe} and \ref{ss:pv}.

An example of core graph is drawn in Figure \ref{fig:coregraph}.

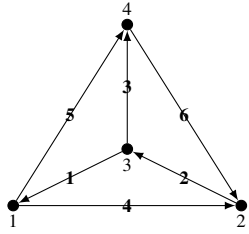
\begin{figure}[H]
\centering
\begin{tikzpicture}[baseline=1ex,scale=1.5]
 \node[int] (a) at (-1,-.5) {};
 \node[int] (b) at (1,-.5) {};
 \node[int] (c) at (0,0) {};
 \node[int] (d) at (0,1.1) {};
 \node[below] at (a) {$\scriptstyle 1$};
 \node[below] at (b) {$\scriptstyle 2$};
 \node[below] at (c) {$\scriptstyle 3$};
 \node[above] at (d) {$\scriptstyle 4$};
 \draw (a) edge[-latex] node {$\scriptstyle {\mathbf 4}$} (b);
 \draw (a) edge[latex-] node {$\scriptstyle {\mathbf 1}$} (c);
 \draw (a) edge[-latex] node {$\scriptstyle {\mathbf 5}$} (d);
 \draw (b) edge[-latex] node {$\scriptstyle {\mathbf 2}$} (c);
 \draw (b) edge[latex-] node {$\scriptstyle {\mathbf 6}$} (d);
 \draw (c) edge[-latex] node {$\scriptstyle {\mathbf 3}$} (d);
\end{tikzpicture}
\caption{\label{fig:coregraph}
An example of a core graph. The tick arrows are used to depict the direction of edges. Vertices and edges are labelled.}
\end{figure}

We often do not consider all core graphs, but only core graphs that satisfy some extra conditions. Here are some of them.

\begin{defi}
Let $\Gamma\in\bar\mV_v\bar\mE_e\gra$, $V$ its set of vertices and $E$ its set of edges.

An edge $a\in E$ is a \emph{tadpole} if $\Gamma_-(a)=\Gamma_+(a)$.

A \emph{valence} of vertex $x\in V$ is the number of edges $a\in E$ such that $\Gamma_-(a)=x$ plus the number of edges $a\in E$ such that $\Gamma_+(a)=x$.

For $a\in E$ we say that vertices $\Gamma_-(a)$ and $\Gamma_+(a)$ are connected. We extend the notion of being connected by transitivity, such that it is a relation of equivalence. Equivalence classes are called \emph{connected components}. A core graph is \emph{connected} if it has one connected component and \emph{disconnected} if it has more than one connected component.

Let
\begin{equation}
\bar\mV_v\bar\mE_e\gra^\notadp
\end{equation}
be the set of all core graphs with $v$ vertices ans $e$ edges without tadpoles.
Let
\begin{equation}
\bar\mV_v^{\geq i}\bar\mE_e\gra
\end{equation}
be the set of core graphs with $v$ vertices and $e$ edges, whose vertices are at least $i$-valent. Also, let
\begin{equation}
\bar\mV_v\bar\mE_e\grac
\end{equation}
be the set of all connected core graphs with $v$ vertices and $e$ edges.
Similarly, we define $\bar\mV_v^{\geq i}\bar\mE_e\grac$, $\bar\mV_v^{\geq i}\bar\mE_e\gra^\notadp$, etc.
\end{defi}

With the same extra notation we define corresponding spaces of core graphs and spaces we are going to define later, e.g.\ $\bar\mV_v^{\geq i}\bar\mE_e\Gra$, $\bar\mV_v\bar\mE_e\Grac$ etc.
We are also going to introduce another conditions that will be denoted by adding extra notation.

\subsection{Space of graphs}
\label{ss:sog}

For an edge $a\in E$ we define the map $i_a:\bar\mV_v\bar\mE_e\gra\rightarrow\bar\mV_v\bar\mE_e\gra$ as follows.
\begin{equation}
(i_a\Gamma)(b)=((i_a\Gamma)_-(b),(i_a\Gamma)_+)(b)
\left\{
\begin{array}{ll}
\Gamma(b)=(\Gamma_-(b),\Gamma_+(b))
\qquad&\text{for $b\neq a$,}\\
(\Gamma_+(b),\Gamma_-(b))
\qquad&\text{for $b=a$.}
\end{array}
\right.
\end{equation}
So, $i_a$ reverts the direction of the edge $a$. It is clearly an involution (inverse of itself) and defines an action of $S_2$ on $\bar\mV_v\bar\mE_e\gra$, where a non-trivial element $\chi\in S_2$ acts like $i_a$. All $e$ actions for edges in $E$ define an action of $S_2^{\times e}$ on $\bar\mV_v\bar\mE_e\gra$. The action is by linearity extended to $\bar\mV_v\bar\mE_e\Gra$.

Our goal is to define more types of edges. We model them with an $\langle S_2\rangle$ module $\Sigma$, that is a $\K$-vector space $\Sigma$ together with an involution. The tensor product $\Sigma^{\otimes e}$ has a natural action of $S_2^{\times e}$.

\begin{defi}[Space of graphs]
Let $v>0$ and $e\geq 0$ be integers and $\Sigma$ be an $\langle S_2\rangle$ module.
The \emph{space of graphs} with $v$ vertices and $e$ edges of types in $\Sigma$ is
\begin{equation}
\bar\mV_v\bar\mE_e^\Sigma\Gra:=\bar\mV_v\bar\mE_e\Gra\otimes_{S_2^{\times e}}\Sigma^{\otimes e}=
\left(\bar\mV_v\bar\mE_e\Gra\otimes\Sigma^{\otimes e}\right)_{S_2^{\times e}}.
\end{equation}
Also, let
\begin{equation}
\bar\mV\bar\mE^\Sigma\Gra:=\bigoplus_{v,e}\bar\mV_v\bar\mE_e^\Sigma\Gra.
\end{equation}
\end{defi}

The group in the subscript means taking coinvariants of its action on both sides. It means that we identify $\Gamma\otimes(\alpha_1\otimes\dots\otimes\alpha_e)$ with $i_a\Gamma\otimes(\alpha_1\otimes\dots\otimes\chi(\alpha_a)\otimes\dots\otimes\alpha_e)$.

\begin{rem}
In $\bar\mV_v\bar\mE_e^\Sigma\Gra^\notadp$ we can get rid of taking coinvariants of the action of group $S_2^{\times e}$ by choosing a particular standard directions of edges. Let $\bar\mV_v\bar\mE_e\sgra^\notadp\subset\bar\mV_v\bar\mE_e\mE_e\gra^\notadp$ be the subset of core graphs $\Gamma$ such that $\Gamma_-(a)<\Gamma_+(a)$ for every edge $a$, called \emph{standard core graphs}, and let $\bar\mV_v\bar\mE_e\sGra^\notadp=\langle\bar\mV_v\bar\mE_e\sgra^\notadp\rangle$.
Then
$$
\bar\mV_v\bar\mE_e^\Sigma\Gra^\notadp=
\bar\mV_v\bar\mE_e\Gra^\notadp\otimes_{S_2^{\times e}}\Sigma^{\otimes e}=
\bar\mV_v\bar\mE_e\sGra^\notadp\otimes\Sigma^{\otimes e}.
$$
Somewhat more complicated analogous expression for $\bar\mV_v\bar\mE_e\bar\mV\bar\mE^\Sigma\Gra$ where tadpoles are allowed is left to the reader.
\end{rem}

\begin{ex}
\label{ex:s2mod}
Here we list three simplest $\langle S_2\rangle$ modules.

There are two 1-dimensional $\langle S_2\rangle$ modules, one where non-trivial $\chi\in S_2$ acts as an identity, and the other where it multiplies by $-1$. The first one we denote by $\Sigma^+$ and the latter one by $\Sigma^-$.

The space $\bar\mV_v\bar\mE_e^{\Sigma^+}\Gra$, also denoted simpler by $\bar\mV_v\bar\mE_e^+\Gra$, is the space $\bar\mV_v\bar\mE_e\Gra$ where we identify $\Gamma$ with $i_a(\Gamma)$ for every $a\in E$. That means that a graph is identified with the graph with reverted edge, so edges are essentially not directed.
In drawings of graphs we usually depict this kind of edges with single line without an arrow such as lines in the following graph. We skip labels of vertices and edges for simplicity.
$$
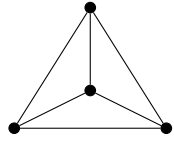
\begin{tikzpicture}[baseline=1ex]
 \node[int] (a) at (-1,-.5) {};
 \node[int] (b) at (1,-.5) {};
 \node[int] (c) at (0,0) {};
 \node[int] (d) at (0,1.1) {};
 \draw (a) edge (c);
 \draw (b) edge (c);
 \draw (c) edge (d);
 \draw (a) edge (b);
 \draw (a) edge (d);
 \draw (b) edge (d);
\end{tikzpicture}
$$

The space $\bar\mV_v\bar\mE_e^{\Sigma^-}\Gra$, also denoted simpler by $\bar\mV_v\bar\mE_e^-\Gra$, is the space $\bar\mV_v\bar\mE_e\Gra$ where we identify $\Gamma$ with $-i_a(\Gamma)$ for every $a\in E$. That means that a graph is identified with the negative of the graph with reverted edge.
In drawings of graphs we usually depict this kind of edges with a line with a simple arrow at one end, such as lines in the following graphs.
$$
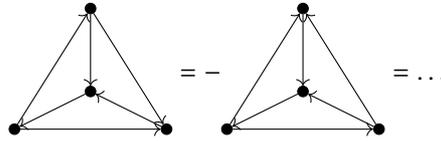
\begin{tikzpicture}[baseline=1ex]
 \node[int] (a) at (-1,-.5) {};
 \node[int] (b) at (1,-.5) {};
 \node[int] (c) at (0,0) {};
 \node[int] (d) at (0,1.1) {};
 \draw (a) edge[<-] (c);
 \draw (b) edge[->] (c);
 \draw (c) edge[<-] (d);
 \draw (a) edge[->] (b);
 \draw (a) edge[->] (d);
 \draw (b) edge[<-] (d);
\end{tikzpicture}
=
-
\begin{tikzpicture}[baseline=1ex]
 \node[int] (a) at (-1,-.5) {};
 \node[int] (b) at (1,-.5) {};
 \node[int] (c) at (0,0) {};
 \node[int] (d) at (0,1.1) {};
 \draw (a) edge[<-] (c);
 \draw (b) edge[->] (c);
 \draw (c) edge[<-] (d);
 \draw (a) edge[->] (b);
 \draw (a) edge[->] (d);
 \draw (b) edge[->] (d);
\end{tikzpicture}
=\dots
$$

A simple 2-dimensional $\langle S_2\rangle$ module is $\Sigma^{fix}=\langle\Ed,\dE\rangle$, where $\chi(\Ed)=\dE$ and vice-versa. The space $\bar\mV_v\bar\mE_e^{\Sigma^{fix}}\Gra$ is isomorphic to the space of core graphs $\bar\mV_v\bar\mE_e\Gra$, so edges are essentially directed.
In drawings of graphs we usually depict this kind of edges with a line with a thick arrow like the core graphs, such as lines in the following graph.
$$
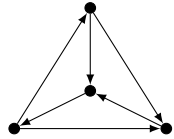
\begin{tikzpicture}[baseline=1ex]
 \node[int] (a) at (-1,-.5) {};
 \node[int] (b) at (1,-.5) {};
 \node[int] (c) at (0,0) {};
 \node[int] (d) at (0,1.1) {};
 \draw (a) edge[latex-] (c);
 \draw (b) edge[-latex] (c);
 \draw (c) edge[latex-] (d);
 \draw (a) edge[-latex] (b);
 \draw (a) edge[-latex] (d);
 \draw (b) edge[latex-] (d);
\end{tikzpicture}
$$
\end{ex}

\subsection{Base graphs}
\label{ss:bg}
A simple study of $\langle S_2\rangle$ modules imply that every such module is a direct sum of modules isomorphic to $\Sigma^+$, $\Sigma^-$ and $\Sigma^{fix}$ from Example \ref{ex:s2mod}\footnote{Indeed $\Sigma^{fix}=\Sigma^+\oplus\Sigma^-$, so every such module is a direct sum of only $\Sigma^+$ and $\Sigma^-$. But we are often interested in $\Sigma^{fix}$ as a whole.}. In the other words, every $\langle S_2\rangle$ module $\Sigma$ has a basis $\sigma$ such that for every $\alpha\in\sigma$ its reverse $\chi(\alpha)$ is another basis element, or its negative. We will always try to use this kind of basis.

Let $\sigma$ be this kind of basis of $\Sigma$. Then $\bar\mV_v\bar\mE_e\gra\times\sigma^{\times e}$ is the basis of $\bar\mV_v\bar\mE_e\Gra\otimes\Sigma^{\otimes e}$. However, the set of classes
\begin{equation}
\label{eq:gen1}
\bar\mV_v\bar\mE_e^\sigma\gra:=\{[\Gamma]|\Gamma\in\bar\mV_v\bar\mE_e\gra\times\sigma^{\times e}\}
\end{equation}
under the action of $S_2^{\times e}$ is not the basis of $\bar\mV_v\bar\mE_e^\Sigma\Gra=\bar\mV_v\bar\mE_e\Gra\otimes_{S_2^{\times e}}\Sigma^{\otimes e}$ because it can be linearly dependant, but it still generates $\bar\mV_v\bar\mE_e^\Sigma\Gra$. Let also
\begin{equation}
\label{eq:gen2}
\bar\mV\bar\mE^\sigma\gra:=\bigcup_{v,e}\bar\mV_v\bar\mE_e^\sigma\gra.
\end{equation}

\begin{defi}
Elements of $\bar\mV\bar\mE^\sigma\gra$ are called \emph{base graphs}. Elements of $\bar\mV_v\bar\mE_e\gra\times\sigma^{\times e}$ are called \emph{base graph representatives}.
\end{defi}

Since the set of base graphs $\bar\mV\bar\mE^\sigma\gra$ generates the space of graphs $\bar\mV_v\bar\mE_e^\Sigma\Gra$ we will write graphs as linear combination of base graphs. Clearly, the expression may not be unique. Figure \ref{fig:basegraph} shows an example of a base graph, and the way how to write it in more compact way.

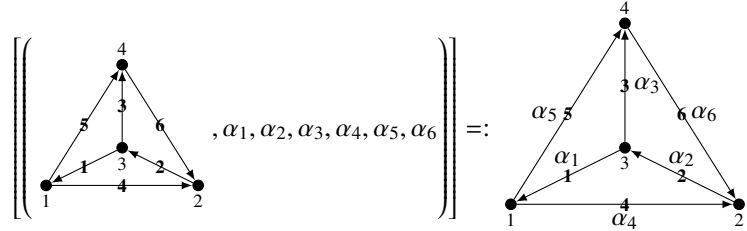
\begin{figure}[H]
\centering
$$
\left[\left(
\begin{tikzpicture}[baseline=1ex]
 \node[int] (a) at (-1,-.5) {};
 \node[int] (b) at (1,-.5) {};
 \node[int] (c) at (0,0) {};
 \node[int] (d) at (0,1.1) {};
 \node[below] at (a) {$\scriptstyle 1$};
 \node[below] at (b) {$\scriptstyle 2$};
 \node[below] at (c) {$\scriptstyle 3$};
 \node[above] at (d) {$\scriptstyle 4$};
 \draw (a) edge[-latex] node {$\scriptstyle {\mathbf 4}$} (b);
 \draw (a) edge[latex-] node {$\scriptstyle {\mathbf 1}$} (c);
 \draw (a) edge[-latex] node {$\scriptstyle {\mathbf 5}$} (d);
 \draw (b) edge[-latex] node {$\scriptstyle {\mathbf 2}$} (c);
 \draw (b) edge[latex-] node {$\scriptstyle {\mathbf 6}$} (d);
 \draw (c) edge[-latex] node {$\scriptstyle {\mathbf 3}$} (d);
\end{tikzpicture}
,\alpha_1,\alpha_2,\alpha_3,\alpha_4,\alpha_5,\alpha_6\right)\right]
=:
\begin{tikzpicture}[baseline=1ex,scale=1.5]
 \node[int] (a) at (-1,-.5) {};
 \node[int] (b) at (1,-.5) {};
 \node[int] (c) at (0,0) {};
 \node[int] (d) at (0,1.1) {};
 \node[below] at (a) {$\scriptstyle 1$};
 \node[below] at (b) {$\scriptstyle 2$};
 \node[below] at (c) {$\scriptstyle 3$};
 \node[above] at (d) {$\scriptstyle 4$};
 \draw (a) edge[-latex] node {$\scriptstyle {\mathbf 4}$} node [below] {$\alpha_4$} (b);
 \draw (a) edge[latex-] node {$\scriptstyle {\mathbf 1}$} node [above] {$\alpha_1$} (c);
 \draw (a) edge[-latex] node {$\scriptstyle {\mathbf 5}$} node [left] {$\alpha_5$} (d);
 \draw (b) edge[-latex] node {$\scriptstyle {\mathbf 2}$} node [above] {$\alpha_2$} (c);
 \draw (b) edge[latex-] node {$\scriptstyle {\mathbf 6}$} node [right] {$\alpha_6$} (d);
 \draw (c) edge[-latex] node {$\scriptstyle {\mathbf 3}$} node [right] {$\alpha_3$} (d);
\end{tikzpicture}
$$
\caption{\label{fig:basegraph}
An example of a base graph. We can write the type of an edge next to its label on the picture ans skip the brackets.}
\end{figure}

Relations on base graph representatives are
\begin{equation}
\label{eq:BaseRel}
\Gamma\otimes(\alpha_1\otimes\dots\otimes\dots\otimes\alpha_e)\sim i_a\Gamma\otimes(\alpha_1\otimes\dots\otimes\chi(\alpha_a)\otimes\dots\otimes\alpha_e)
\end{equation}
for any choice of $\Gamma$, $\alpha_i$ and $a$. One special consequence is that if $a$ is a tadpole and $\chi(\alpha_a)=-\alpha_a$ then
\begin{equation}
\Gamma\otimes(\alpha_1\otimes\dots\otimes\alpha_e)\sim 0.
\end{equation}

To simplify the drawings, elements of $\sigma$ are usually depicted by actual lines with arrows and decorations:
\begin{itemize}
\item If $\chi(\alpha)=\alpha$ we use lines with symmetric decorations such as 
\begin{tikzpicture}[baseline=-.8ex,scale=.5]
\node[nil] (a) at (0,0) {};
\node[nil] (b) at (1,0) {};
\draw (a) edge (b);
\end{tikzpicture} and
\begin{tikzpicture}[baseline=-.8ex,scale=.5]
\node[nil] (a) at (0,0) {};
\node[nil] (b) at (1,0) {};
\draw (a) edge[crossed] (b);
\end{tikzpicture}.
\item If $\chi(\alpha)=-\alpha$ we use lines with a simple arrow on one end and otherwise symmetric decorations such as
\begin{tikzpicture}[baseline=-.8ex,scale=.5]
\node[nil] (a) at (0,0) {};
\node[nil] (b) at (1,0) {};
\draw (a) edge[->] (b);
\end{tikzpicture} and
\begin{tikzpicture}[baseline=-.8ex,scale=.5]
\node[nil] (a) at (0,0) {};
\node[nil] (b) at (1,0) {};
\draw (a) edge[->,crossed] (b);
\end{tikzpicture},
where $-\alpha$ is depicted by the same line with the arrow on the other side.
\item If $\chi(\alpha)=\beta\neq\pm\alpha$, for $\alpha$ we use line with thick arrow at one end, and some other decorations such as
\begin{tikzpicture}[baseline=-.8ex,scale=.5]
\node[nil] (a) at (0,0) {};
\node[nil] (b) at (1,0) {};
\draw (a) edge[-latex] (b);
\end{tikzpicture} and
\begin{tikzpicture}[baseline=-.8ex,scale=.5]
\node[nil] (a) at (0,0) {};
\node[nil] (b) at (1,0) {};
\draw (a) edge[-latex,crossed] (b);
\end{tikzpicture},
while $\beta$ is depicted with the opposite line.
\end{itemize}
It is now easy to draw a base graph by using such lines instead of writing an edge type next to the label. A line is drawn in place of the core edge such that the core direction goes from left to right, like on Figure \ref{fig:basegraph2}.
The drawings are consistent to those in Example \ref{ex:s2mod}.

\begin{figure}[H]
\centering
$$
\left[\left(
\begin{tikzpicture}[baseline=1ex]
 \node[int] (a) at (-1,-.5) {};
 \node[int] (b) at (1,-.5) {};
 \node[int] (c) at (0,0) {};
 \node[int] (d) at (0,1.1) {};
 \node[below] at (a) {$\scriptstyle 1$};
 \node[below] at (b) {$\scriptstyle 2$};
 \node[below] at (c) {$\scriptstyle 3$};
 \node[above] at (d) {$\scriptstyle 4$};
 \draw (a) edge[-latex] node {$\scriptstyle {\mathbf 4}$} (b);
 \draw (a) edge[latex-] node {$\scriptstyle {\mathbf 1}$} (c);
 \draw (a) edge[-latex] node {$\scriptstyle {\mathbf 5}$} (d);
 \draw (b) edge[-latex] node {$\scriptstyle {\mathbf 2}$} (c);
 \draw (b) edge[latex-] node {$\scriptstyle {\mathbf 6}$} (d);
 \draw (c) edge[-latex] node {$\scriptstyle {\mathbf 3}$} (d);
\end{tikzpicture},
\begin{tikzpicture}[baseline=-.8ex,scale=.5]
\node[nil] (a) at (0,0) {};
\node[nil] (b) at (1,0) {};
\draw (a) edge (b);
\end{tikzpicture},
\begin{tikzpicture}[baseline=-.8ex,scale=.5]
\node[nil] (a) at (0,0) {};
\node[nil] (b) at (1,0) {};
\draw (a) edge[crossed] (b);
\end{tikzpicture},
\begin{tikzpicture}[baseline=-.8ex,scale=.5]
\node[nil] (a) at (0,0) {};
\node[nil] (b) at (1,0) {};
\draw (a) edge[->] (b);
\end{tikzpicture},
\begin{tikzpicture}[baseline=-.8ex,scale=.5]
\node[nil] (a) at (0,0) {};
\node[nil] (b) at (1,0) {};
\draw (a) edge[->,crossed] (b);
\end{tikzpicture},
\begin{tikzpicture}[baseline=-.8ex,scale=.5]
\node[nil] (a) at (0,0) {};
\node[nil] (b) at (1,0) {};
\draw (a) edge[latex-] (b);
\end{tikzpicture},
\begin{tikzpicture}[baseline=-.8ex,scale=.5]
\node[nil] (a) at (0,0) {};
\node[nil] (b) at (1,0) {};
\draw (a) edge[-latex,crossed] (b);
\end{tikzpicture}
\right)\right]
=:
\begin{tikzpicture}[baseline=1ex,scale=1.5]
 \node[int] (a) at (-1,-.5) {};
 \node[int] (b) at (1,-.5) {};
 \node[int] (c) at (0,0) {};
 \node[int] (d) at (0,1.1) {};
 \node[below] at (a) {$\scriptstyle 1$};
 \node[below] at (b) {$\scriptstyle 2$};
 \node[below] at (c) {$\scriptstyle 3$};
 \node[above] at (d) {$\scriptstyle 4$};
 \draw (a) edge[->,crossed] node [below] {$\scriptstyle {\mathbf 4}$} (b);
 \draw (a) edge node [above] {$\scriptstyle {\mathbf 1}$} (c);
 \draw (a) edge[latex-] node [left] {$\scriptstyle {\mathbf 5}$} (d);
 \draw (b) edge[crossed] node [above] {$\scriptstyle {\mathbf 2}$} (c);
 \draw (b) edge[latex-,crossed] node [right] {$\scriptstyle {\mathbf 6}$} (d);
 \draw (c) edge[->] node [right] {$\scriptstyle {\mathbf 3}$} (d);
\end{tikzpicture}
$$
\caption{\label{fig:basegraph2}
An example of a base graph with base edge types depicted with lines. Note that the notation is chosen such that $[\Gamma\otimes(\alpha_1\otimes\dots\otimes\alpha_e)]$ and $[i_a\Gamma\otimes(\alpha_1\otimes\dots\otimes\chi(\alpha_a)\otimes\dots\otimes\alpha_e)]$ are depicted the same, so we do not need to draw core direction.}
\end{figure}
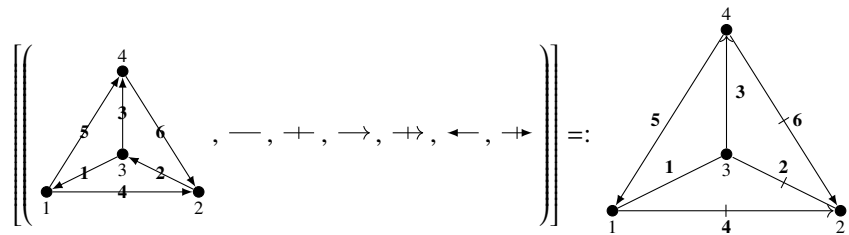

\subsection{Multi-directed graphs}
\label{ss:mdg}

\begin{defi}[Multi-directed graph]
Let $L$ be the set of ``colours'' with $|L|=\ell$. Let us pick a space of core graphs $\bar\mV_v\bar\mE_e\gra$, $V$ its set of vertices and $E$ its set of edges.

An \emph{$\ell$-directed core graph} with $v$ vertices and $e$ edges is a core graph $\Gamma\in\bar\mV_v\bar\mE_e\gra$ together with maps $o_c:E\rightarrow \{+,-\}$ for all $c\in L$.

By $\mD_\ell\bar\mV_v\bar\mE_e\gra$ we denote the set of all $\ell$-directed core graphs, with $v$ vertices and $e$ edges.
We set $\mD_\ell\bar\mV_v\bar\mE_e\Gra:=\langle\mD_\ell\bar\mV_v\bar\mE_e\gra\rangle$ and
\begin{equation}
\mD_\ell\bar\mV\bar\mE\Gra:=\bigoplus_{v,e}\mD_\ell\bar\mV_v\bar\mE_e\Gra.
\end{equation}
\end{defi}

The maps $o_c$ indicate the coloured directions. If it is $o_c(a)=+$, the coloured direction of the edge $a$ in color $c$ is the same as the core direction, and if $o_c(a)=-$ it is the opposite.

For an edge $a\in E$ the action of $i_a$ on core graphs is extended to multi-directed core graphs, where it acts like:
\begin{equation}
(i_a o_c)(b)=
\left\{
\begin{array}{ll}
o_c(a)
\qquad&\text{for $b\neq a$,}\\
-o_c(a)
\qquad&\text{for $b=a$.}
\end{array}
\right.
\end{equation}
on the maps $o_c$ for every color $c\in L$.

\begin{defi}[Space of multi-directed graphs]
Let $v>0$, $e\geq 0$ and $\ell\geq 0$ be integers and $\Sigma$ be an $\langle S_2\rangle$ module.
The \emph{space of $\ell$-directed graphs} with $v$ vertices and $e$ edges of types in $\Sigma$ is
\begin{equation}
\mD_\ell\bar\mV_v\bar\mE_e^\Sigma\Gra:=\mD_\ell\bar\mV_v\bar\mE_e\Gra\otimes_{S_2^{\times e}}\Sigma^{\otimes e}=
\left(\mD_\ell\bar\mV_v\bar\mE_e\Gra\otimes\Sigma^{\otimes e}\right)_{S_2^{\times e}}.
\end{equation}
Also, let
\begin{equation}
\mD_\ell\bar\mV\bar\mE^\Sigma\Gra:=\bigoplus_{v,e}\mD_\ell\bar\mV_v\bar\mE_e^\Sigma\Gra.
\end{equation}
\end{defi}

In the space $\mD_\ell\bar\mV\bar\mE^\Sigma\Gra$, every edge has both a type in $\Sigma$ and coloured directions for every color in $L$.

Note that the coloured orientation is always preserved under the action of $i_a$ relative to the actual vertices. So coloured directions are essentially not revertible.
In drawings we depict this directions with thick coloured arrows over the standard presentation of an edge, such as in the following graph.
$$
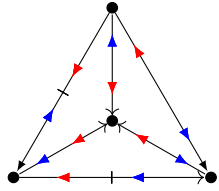
\begin{tikzpicture}[baseline=-1ex,scale=1.5]
\node[int] (a) at (0,0) {};
\node[int] (b) at (90:1) {};
\node[int] (c) at (210:1) {};
\node[int] (d) at (330:1) {};
\draw (a) edge[<-,rightblue,leftred] (b);
\draw (a) edge[<-,rightblue,rightred] (c);
\draw (a) edge[<-,rightblue,leftred] (d);
\draw (b) edge[-latex,leftblue,rightred,crossed] (c);
\draw (b) edge[-latex,rightblue,leftred] (d);
\draw (c) edge[->,leftblue,leftred,crossed] (d);
\end{tikzpicture}
$$

\begin{rem}
\label{rem:ColType}
Let $\Sigma^{fix}=\langle\Ed,\dE\rangle$ be the $\langle S_2\rangle$ module from Example \ref{ex:s2mod}.
The space of $\ell$-directed graphs with edges of types in $\Sigma$, $\mD_\ell\bar\mV\bar\mE^\Sigma\Gra$ is isomorphic to the space $\bar\mV\bar\mE^{\Sigma\otimes\left(\Sigma^{fix}\right)^{\otimes\ell}}\Gra$ of graphs with edges of types in $\Sigma\otimes\left(\Sigma^{fix}\right)^{\otimes\ell}$ where $S_2$ acts on each factor in the tensor product. Here coloured directions are represented by the choice of the element of $\Sigma^{fix}$.
This is the alternative definition of the space of multi-directed graphs.
\end{rem}

\subsection{Sourced and oriented graphs}
\label{ss:saog}

\begin{defi}
\label{defi:SO}
Let $\Gamma\in\mD_\ell\bar\mV_v\bar\mE_e\gra$, $V$ its set of vertices, $E$ its set of edges and $L$ its set of colours.

A vertex $x\in V$ is a \emph{source} in color $c\in L$ if there is no edge $a\in E$ such that $\Gamma_{o_c(a)}(a)=x$. A graph is \emph{sourced} in color $c$ if it has at least one source.

A sequence of edges $a_0,a_1,\dots, a_p=a_0$ is a \emph{cycle} in color $c$ if $\Gamma_{o_c(a_i)}(a_i)=\Gamma_{-o_c(a_{i+1})}(a_{i+1})$ for every $i=0,\dots,p-1$. A graph is \emph{oriented} in color $c$ if there is no cycle in that color.
\end{defi}

\begin{defi}[Oriented sourced directed core graph]
Let $J$, $K$ and $L$ be three disjoint sets of colours with $|J|=j$, $|K|=k$ and $|L|=\ell$. Let $v>0$ and $e\geq 0$.

An \emph{$j$-oriented $k$-sourced $\ell$-directed core graph} is a $(j+k+\ell)$-directed core graph $\Gamma$ such that it is oriented in every color $c\in J$ and sourced in every color $c\in K$.

By $\mO_j\mS_k\mD_\ell\bar\mV_v\bar\mE_e\gra$ we denote the set of all $j$-oriented $k$-sourced $\ell$-directed core graphs with $v$ vertices and $e$ edges.
We set $\mO_j\mS_k\mD_\ell\bar\mV_v\bar\mE_e\Gra:=\langle\mO_j\mS_k\mD_\ell\bar\mV_v\bar\mE_e\gra\rangle$ and
\begin{equation}
\mO_j\mS_k\mD_\ell\bar\mV\bar\mE\Gra:=\bigoplus_{v,e}\mO_j\mS_k\mD_\ell\bar\mV_v\bar\mE_e\Gra.
\end{equation}

Let $\Sigma$ be an $\langle S_2\rangle$ module.
The \emph{space of $j$-oriented $k$-sourced $\ell$-directed graphs} with $v$ vertices and $e$ edges of types in $\Sigma$ is
\begin{equation}
\mO_j\mS_k\mD_\ell\bar\mV_v\bar\mE_e^\Sigma\Gra:=\mO_j\mS_k\mD_\ell\bar\mV_v\bar\mE_e\Gra\otimes_{S_2^{\times e}}\Sigma^{\otimes e}=
\left(\mO_j\mS_k\mD_\ell\bar\mV_v\bar\mE_e\Gra\otimes\Sigma^{\otimes e}\right)_{S_2^{\times e}}.
\end{equation}
Also, let
\begin{equation}
\mO_j\mS_k\mD_\ell\bar\mV\bar\mE^\Sigma\Gra:=\bigoplus_{v,e}\mO_j\mS_k\mD_\ell\bar\mV_v\bar\mE_e^\Sigma\Gra.
\end{equation}
\end{defi}

Some examples of the graphs are drawn in Figure \ref{fig:exGraphs}. If Any of $j,k,\ell$ is 0, we can skip the whole prefix. This is consistent with the equality $\mO_0\mS_0\mD_0\bar\mV\bar\mE^\Sigma\Gra=\bar\mV\bar\mE^\Sigma\Gra$.

\begin{figure}[H]
\centering
$$
\begin{tikzpicture}[baseline=-1ex,scale=1.5]
\node[int] (a) at (0,0) {};
\node[int] (b) at (1,0) {};
\draw (a) edge[rightblue,leftred] (b);
\end{tikzpicture}
\;,\quad
\begin{tikzpicture}[baseline=-1ex,scale=1.5]
\node[int] (a) at (0,0) {};
\node[int] (b) at (90:1) {};
\node[int] (c) at (210:1) {};
\node[int] (d) at (330:1) {};
\draw (a) edge[rightblue,leftred] (b);
\draw (a) edge[rightblue,rightred] (c);
\draw (a) edge[leftblue,leftred] (d);
\draw (b) edge[leftblue,rightred] (c);
\draw (b) edge[leftblue,leftred] (d);
\draw (c) edge[leftblue,leftred] (d);
\end{tikzpicture}
\;,\quad
\begin{tikzpicture}[baseline=-1ex,scale=1.5]
\node[int] (a) at (0,0) {};
\node[int] (b) at (90:1) {};
\node[int] (c) at (210:1) {};
\node[int] (d) at (330:1) {};
\draw (a) edge[rightblue,leftred] (b);
\draw (a) edge[rightblue,rightred] (c);
\draw (a) edge[rightblue,leftred] (d);
\draw (b) edge[leftblue,rightred] (c);
\draw (b) edge[rightblue,leftred] (d);
\draw (c) edge[leftblue,leftred] (d);
\end{tikzpicture}
\;.
$$
\caption{\label{fig:exGraphs}
Examples of 1-oriented 1-sourced graphs. Sourced color is blue and oriented color is red. First two graphs are also 2-oriented.}
\end{figure}
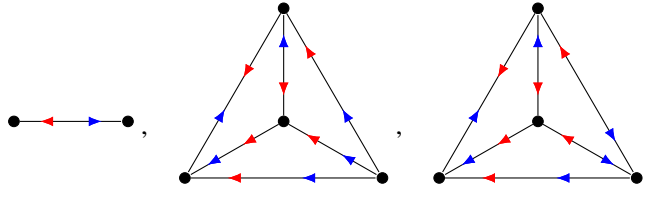

\begin{rem}
If $\Sigma=\Sigma^{fix}$ from Example \ref{ex:s2mod}, we can also talk about sourced and oriented graphs in $\bar\mV\bar\mE^{\Sigma^{fix}}\Gra$. Recall that it is isomorphic to $\bar\mV\bar\mE\Gra$, and we simply use core direction instead of coloured direction in Definition \ref{defi:SO}.
\end{rem}

\subsection{Edge differential}
\label{ss:ed}

Let the $\langle S_2\rangle$ module $\Sigma$ be differential graded (abbreviated \emph{dg}), i.e.\ a graded vector space with a map $\delta_E:\Sigma\rightarrow\Sigma$ of degree $-1$ such that $\delta_E^2=0$. The action of $\langle S_2\rangle$ is of degree 0.
Assume also that $\delta_E$ commutes  with the $\langle S_2\rangle$ module structure, i.e.\ $\delta_E\chi \alpha=\chi\delta_E \alpha$ for every $\chi\in S_2$ and $\alpha\in\Sigma$.

The complex structure is easily extended to tensor product $\Sigma^{\otimes e}$. The differential is defined on homogeneous elements by
\begin{equation}
\label{def:deltaEsgn}
\delta_E^{(a)}(\alpha_1\otimes\dots\otimes \alpha_e):=
(-1)^{k_a}\alpha_1\otimes\dots\otimes \alpha_{a-1}\otimes\delta_E \alpha_a\otimes \alpha_{a+1}\otimes\dots\otimes \alpha_e,
\end{equation}
where $k_a$ is the number of odd graded $\alpha_i$ for $i>a$, and
\begin{equation}
\label{def:deltaE}
\delta_E(\alpha_1\otimes\dots\otimes \alpha_e)=\sum_{a=1}^e\delta_E^{(a)}(\alpha_1\otimes\dots\otimes \alpha_e).
\end{equation}
By abuse of notation, we denote this differential also by $\delta_E$. The signs in the formula are necessary to ensure that $\delta_E^2=0$. It is clear that $\delta_E$ on $\Sigma^{\otimes e}$ commutes with the action of $S_2^{\times e}$.

There is no differential on $\bar\mV_v\bar\mE_e\Gra$, so one easily extends the action of $\delta_E$ to $\bar\mV_v\bar\mE_e\Gra\otimes\Sigma^{\otimes e}$. Since it commutes with the action of $S_2^{\times e}$, the differential $\delta_E$ is also well defined on $\bar\mV_v\bar\mE_e^\Sigma\Gra=\bar\mV_v\bar\mE_e\Gra\otimes_{S_2^{\times e}}\Sigma^{\otimes e}$, and hence on the whole $\bar\mV\bar\mE^\Sigma\Gra$. By further abuse of notation, we denote all those differentials by $\delta_E$.

\begin{prop}
For every dg $\langle S_2\rangle$ module $(\Sigma,\delta_E)$ it holds that
\begin{equation}
H\left(\bar\mV\bar\mE^\Sigma\Gra,\delta_E\right)=\bar\mV\bar\mE^{H(\Sigma,\delta_E)}\Gra.
\end{equation}
\end{prop}
\begin{proof}
Easy homological arguments imply that the homology commutes with the tensor product and also with taking coinvariants under the actions of groups involved, so the formula follows.
\end{proof}

Analogous arguments hold for multi-direced graphs, and other versions of graphs spaces.

\subsection{Permuting edges}
\label{ss:pe}

In this subsection we study the action of permuting edges in a graph.
There is an action of the symmetric group $S_e$ on core graphs in $\bar\mV_v\bar\mE_e\gra$ that permutes edges:
\begin{equation}
\label{eq:SeAct}
(\chi\Gamma)(a)=\Gamma(\chi^{-1}(a))
\end{equation}
for $\chi\in S_e$. The action extends to $\ell$-directed core graphs in $\mD_\ell\bar\mV_v\bar\mE_e\gra$ by
\begin{equation}
\label{eq:SeActC}
(\chi o_c)(a)=o_c(\chi^{-1}(a))
\end{equation}
for $c\in L$. By linearity, actions are extended to $\bar\mV_v\bar\mE_e\Gra$ and $\mD_\ell\bar\mV_v\bar\mE_e\Gra$.

The action of $S_e$ on $\Sigma^{\otimes e}$ seems easy, we can just permute the factors. But we want the action to commute with the differential $\delta_E$ defined in \eqref{def:deltaE}. This demand causes some difficulties with signs. Indeed we permute factors:
\begin{equation}
\label{def:SeSigma}
\chi(\alpha_1\otimes\dots\otimes\alpha_e)=\pm\alpha_{\chi^{-1}(1)}\otimes\dots\otimes\alpha_{\chi^{-1}(e)}
\end{equation}
but the sign is determined as follows. Assume all $\alpha_i$ are of homogeneous degree. Let $O\subseteq E$ be the set of all $i$ such that $\alpha_i$ is of odd degree. Permutation $\chi$ restricted to $O$ and corestricted to its image can be seen as a permutation on $O$, after ignoring other elements of $E$ and saving the order. The sign in the formula \eqref{def:SeSigma} is the sign of this restricted corestricted permutation.

\begin{prop}
The action of $S_e$ on $\Sigma^{\otimes e}$ commutes with the differential $\delta_E$, i.e.\ $\delta_E\chi\alpha=\chi\delta_E\alpha$ for every $\chi\in S_e$ and $\alpha\in\Sigma^{\times e}$.
\end{prop}
\begin{proof}
Straightforward verification of signs.
\end{proof}

Now we can extend the action of $S_e$ onto $\bar\mV_v\bar\mE_e\Gra\otimes\Sigma^{\otimes e}$. To extend it to $\bar\mV_v\bar\mE_e^\Sigma\Gra$, that is, coinvariants of the $S_2^{\times e}$ action, we need to notice that $S_e$ also acts on $S_2^{\times e}$ by permuting factors. This well defines the action of $S_e$ onto $\bar\mV_v\bar\mE_e^\Sigma\Gra$.

\begin{defi}[Space of graphs with permuted edges]
Let $v>0$ and $e\geq 0$ be integers and $\Sigma$ be an $\langle S_2\rangle$ module.

The \emph{space of graphs with permuted edges} with $v$ vertices and $e$ edges of types in $\Sigma$ is
\begin{equation}
\bar\mV_v\mE_e^\Sigma\Gra:=
\left(\bar\mV_v\bar\mE_e^\Sigma\Gra\right)_{S_e}.
\end{equation}
Also, let
\begin{equation}
\bar\mV\mE^\Sigma\Gra:=\bigoplus_{v,e}\bar\mV_v\mE_e^\Sigma\Gra.
\end{equation}
\end{defi}

Note the disappearance of the bar on $\mE$ in the notation.
Space of multi-directed, sourced or oriented graphs with permuted edges, as well as their subspaces with permuted edges, are defined in analogous way.

\begin{rem}
Since the space $\bar\mV_v\bar\mE_e^\Sigma\Gra$ is already space of coinvariants of the action of $S_2^{\times e}$, the space $\bar\mV_v\mE_e^\Sigma\Gra$ can be directly defined as a space of coinvariants
\begin{equation}
\bar\mV_v\mE_e^\Sigma\Gra=\left(\bar\mV_v\bar\mE_e\Gra\otimes\Sigma^{\otimes e}\right)_{S_e\ltimes S_2^{\times e}}
\end{equation}
under the action of semi-direct product $S_e\ltimes S_2^{\times e}$ where $S_e$ acts on $S_2^{\times e}$ in the usual way.
\end{rem}

\begin{rem}
In many cases of interest, the space of edge types $\Sigma$ will be concentrated in one degree, and $\delta_E=0$. In those cases it does still matter weather $\Sigma$ is concentrated in even or odd degree.

Let $S_e$ in this remark act on $\bar\mV_v\bar\mE_e^\Sigma\Gra$ by simple permuting edges, without sign from \eqref{def:SeSigma}. If $\Sigma$ is concentrated in even degree it holds that
$$
\bar\mV_v\mE_e^\Sigma\Gra=
\left(\bar\mV_v\bar\mE_e^\Sigma\Gra\right)_{S_e}
$$
and if it is concentrated in odd degree it holds that
$$
\bar\mV_v\mE_e^\Sigma\Gra=
\left(\bar\mV_v\bar\mE_e^\Sigma\Gra\otimes\sgn_E^-\right)_{S_e}
$$
where $\sgn_E^-$ is one-dimensional representation of $S_e$ where the odd permutation reverses the sign. In the first case edges are simply indistinguishable, and in the latter case they are indistinguishable up to the sign of the permutation.
\end{rem}

\subsection{Permuting vertices}
\label{ss:pv}

In this subsection we study the action of permuting vertices in a graph. It turns out to be easier than permuting edges.
There is an action of the symmetric group $S_v$ on core graphs in $\bar\mV_v\bar\mE_e^\Sigma\gra$ that permutes vertices:
\begin{equation}
\label{eq:SvAct}
(\chi\Gamma)(a)=(\chi(\Gamma_-(a)),\chi(\Gamma_+(a)))
\end{equation}
for $\chi\in S_v$ and $a\in E$.
The action is analogously defined on $\ell$-directed core graphs $\mD_\ell\bar\mV_v\bar\mE_e\gra$ and extends by linearity to $\bar\mV_v\bar\mE_e\Gra$ and $\mD_\ell\bar\mV_v\bar\mE_e\Gra$.
This action commutes with the actions of $S_2^{\times e}$ and $S_e$ that reverse the direction of edges and permutes edges, so we can well define the action of $S_v$ on $\bar\mV_v\bar\mE_e^\Sigma\Gra$ and $\bar\mV_v\mE_e^\Sigma\Gra$.

As the reader might expect, we will consider possible sign changes along the action of $S_v$. Let $\sgn_V^+$ be the trivial one-dimensional representation of $S_v$ and let $\sgn_V^-$ be one-dimensional representation of the same spaces where the odd permutation reverses the sign.

\begin{defi}[Space of graphs with permuted edges and vertices]
Let $v>0$ and $e\geq 0$ be integers and $\Sigma$ be an $\langle S_2\rangle$ module. Let $\mu\in\{+,-\}$.

The \emph{space of graphs with permuted edges and vertices} with $v$ vertices of sign $\mu$, and $e$ edges of types in $\Sigma$ is
\begin{equation}
\label{eq:sgnSv}
\mV_v^\mu\mE_e^\Sigma\Gra:=
\left(\bar\mV_v\mE_e^\Sigma\Gra\otimes\sgn_V^\mu\right)_{S_v}.
\end{equation}
Also, let
\begin{equation}
\mV^\mu\mE^\Sigma\Gra:=\bigoplus_{v,e}\mV_v^\mu\mE_e^\Sigma\Gra.
\end{equation}
\end{defi}

Note the disappearance of the bar on $\mV$ in the notation.
Space of multi-directed, sourced or oriented graphs with permuted edges and vertices, as well as their subspaces with permuted edges and vertices, are defined in analogous way.

\begin{rem}
We see that there are essentially two types of vertices regarding signs when permuted: even vertices and odd vertices. In the first case vertices of the space $\mV_v^+\mE_e^\Sigma\Gra$ are simply indistinguishable, and in the latter case vertices of $\mV_v^-\mE_e^\Sigma\Gra$ are indistinguishable up to the sign of the permutation.
\end{rem}

\begin{rem}
The choice of permuting edges before vertices is arbitrary because two permutations commute. We can first define the \emph{space of graph complexes with permuted vertices}
\begin{equation}
\mV_v^\mu\bar\mE_e^\Sigma\Gra:=
\left(\bar\mV_v\bar\mE_e^\Sigma\Gra\otimes\sgn_V^\mu\right)_{S_v},
\end{equation}
and then it is
\begin{equation}
\mV_v^\mu\mE_e^\Sigma\Gra=
\left(\mV_v^\mu\bar\mE_e^\Sigma\Gra\right)_{S_e}.
\end{equation}

We can also permute vertices before introducing edge types and define the \emph{space of core graphs with permuted vertices}
\begin{equation}
\mV_v^\mu\bar\mE_e\Gra:=
\left(\bar\mV_v\bar\mE_e\Gra\otimes\sgn_V^\mu\right)_{S_v},
\end{equation}
and then it is
\begin{equation}
\mV_v^\mu\bar\mE_e^\Sigma\Gra=\mV_v^\mu\bar\mE_e\Gra\otimes_{S_2^{\times e}}\Sigma^{\otimes e}.
\end{equation}

All together, it also holds that
\begin{equation}
\mV_v^\mu\mE_e^\Sigma\Gra=
\left(\bar\mV_v\bar\mE_e\Gra\otimes_{S_2^{\times e}}\Sigma^{\otimes e}\otimes\sgn_V^\mu\right)_{S_v\times S_e}=
\left(\bar\mV_v\bar\mE_e\Gra\otimes\Sigma^{\otimes e}\otimes\sgn_V^\mu\right)_{S_v\times S_e\ltimes S_2^{\times e}}.
\end{equation}
\end{rem}

From now on, we will mainly be interested in spaces with permuted edges and vertices. However, original spaces with distinguishable edges and vertices will be useful for many proofs. To emphasise the difference, we will call them \emph{distinguishable} spaces.

\subsection{Edge contraction}
\label{ss:ec}

In this subsection we define a map  called ``edge contraction'' on graph spaces. Roughly speaking, it deletes an edge and merges its end vertices, summed over all edges of certain type.

Let $\Gamma\in \bar\mV_v\bar\mE_e\gra$ be a core graph, and $a\in E$ be its edge that is not a tadpole. We define ``contracting the edge $a$'' $c_a(\Gamma)$ that puts a vertex $x$ instead of 
\begin{tikzpicture}[scale=.5]
 \node[int] (a) at (0,0) {};
 \node[int] (b) at (1,0) {};
 \draw (a) edge[->] node[above] {$\scriptstyle a$} (b);
 \node[above left] at (a) {$\scriptstyle x$};
 \node[above right] at (b) {$\scriptstyle y$};
\end{tikzpicture}
and reconnects all edges that were previous connected to old $x$ and $y$ to the new $x$. Labels of edges after $a$ and vertices after $y$ are shifted by 1. More precisely, let $\hat c_a''(\Gamma)\in \bar\mV_{v-1}\bar\mE_{e-1}\gra$ be such that
\begin{equation}
\left(\hat c_a''(\Gamma)\right)_\lambda(b):= \left\{
\begin{array}{ll}
\Gamma_\lambda(b')
\qquad&\text{if $\Gamma_\lambda(b')<\Gamma_+(a)$,}\\
\Gamma_-(a)
\qquad&\text{if $\Gamma_\lambda(b')=\Gamma_+(a)$,}\\
\Gamma_\lambda(b')-1
\qquad&\text{if $\Gamma_\lambda(b')>\Gamma_+(a)$,}\\
\end{array}
\right.
\end{equation}
for $\lambda\in\{+,-\}$, where $b'=b$ if $b<a$ and $b'=b+1$ if $b\geq a$.

We do not want to contract ``leaves'', i.e.\ edges whose one end is 1-valent, so we define
\begin{equation}
\label{def:noleaves}
c_a''(\Gamma):= \left\{
\begin{array}{ll}
\hat c_a''(\Gamma)
\qquad&\text{if neither $\Gamma_-(a)$ nor $\Gamma_+(a)$ is 1-valent,}\\
0
\qquad&\text{if exectly one out of $\Gamma_-(a)$ and $\Gamma_+(a)$ is 1-valent,}\\
-\hat c_a''(\Gamma)
\qquad&\text{if both $\Gamma_-(a)$ and $\Gamma_+(a)$ are 1-valent.}\\
\end{array}
\right.
\end{equation}

The map can be extended by linearity to $c_a'':\bar\mV_v\bar\mE_e\Gra\rightarrow\bar\mV_{v-1}\bar\mE_{e-1}\Gra$. Here we set that $c_a''=0$ if $a$ is a tadpole. But we introduce two possible sign changes: let $\mu\in\{+,-\}$ and let
\begin{equation}
\label{def:CaV}
c_a'(\Gamma):=-\mu^{v-\Gamma_+(a)}c_a''(\Gamma).
\end{equation}

\begin{rem}
Map \eqref{def:CaV} is defined such that it commutes with the action of $S_v$ that permutes vertices of sign $\mu$. So we can already well define the map $c_a':\mV_v^\mu\bar\mE_e\Gra\rightarrow\mV_{v-1}^\mu\bar\mE_{e-1}\Gra$.
\end{rem}

The analogous map can be defined on the space of $j$-oriented $k$-sourced $\ell$-directed graphs $\mO_j\mS_k\mD_\ell\mV_v^\mu\bar\mE_e\Gra$. The directions on the edge being deleted are simply forgotten. But here contracting an edge may form a cycle or destroy the last source and produce the graph that is not in $\mO_j\mS_k\mD_\ell\bar\mV_v\bar\mE_e\gra$. If so, we just put the result to be 0.

To extend the map to graphs with edge types in $\Sigma$ we need to explain what does a contraction do to an edge of certain type in $\Sigma$. It is modelled by a map $C:\Sigma\rightarrow\K$. Contracting an edge $a$, by abuse of notation named and denoted the same as the map on $\bar \mV_v\bar\mE_e\Gra$, is the map $c_a':\Sigma_2^{\otimes e}\rightarrow \Sigma_2^{\otimes (e-1)}$ mapping
\begin{equation}
c_a':\alpha_1\otimes\dots\otimes\alpha_{a-1}\otimes\alpha_a\otimes\alpha_{a+1}\otimes\dots\otimes\alpha_e\mapsto
C(\alpha_a)\alpha_1\otimes\dots\otimes\alpha_{a-1}\otimes\alpha_{a+1}\otimes\dots\otimes\alpha_e
\end{equation}

It is now straightforward to extend the map to $c_a':\bar\mV_v\bar\mE_e\Gra\otimes\Sigma^{\otimes e}\rightarrow\bar\mV_{v-1}\bar\mE_{e-1}\Gra\otimes\Sigma^{\otimes e-1}$ by $c_a'(\Gamma\otimes\alpha)\mapsto c_a'(\Gamma)\otimes c_a'(\alpha)$. It is necessarily that this map commutes with the action of ${S_2^{\times e}}$. To ensure that, we impose the condition on $C:\Sigma\rightarrow\K$:
\begin{equation}
\label{con:CSigma}
C(\chi \alpha)=\mu C(\alpha)
\end{equation}
for every $\alpha\in\Sigma$ where $\chi\in S_2$ is non-trivial element. Under this condition it is clear that $c'_a$ commutes with the action of ${S_2^{\times e}}$, so we can well define the map $c'_a$ on $\bar\mV_v\bar\mE_e^\Sigma\Gra=\bar\mV_v\bar\mE_e\Gra\otimes_{S_2^{\times e}}\Sigma^{\otimes e}$.

\begin{rem}
Condition \eqref{con:CSigma} is immediately fulfilled for $\Sigma^\mu$ from Example \ref{ex:s2mod}. For $\Sigma^{-\mu}$ the condition implies an uninteresting conclusion that $C$ is always 0.
\end{rem}

The goal now is to extend the map to the space with permuted edges. We clearly need to sum the maps over all edges, but we need to be careful with the signs. Let $\Gamma\otimes\alpha_1\otimes\dots\otimes\alpha_e$ be the representative of a graph in $\bar\mV_v^\mu\bar\mE_e^\Sigma\Gra=\bar\mV_v^\mu\bar\mE_e\Gra\otimes_{S_2^{\times e}}\Sigma^{\otimes e}$ where $\alpha_i$ are of homogeneous degree. Then
\begin{equation}
\label{def:CaE}
c_a:=(-1)^{\deg(\alpha_a)k_a}c'_a,
\end{equation}
where $k_a$ is the number of odd graded $\alpha_i$ for $i>a$, and
\begin{equation}
\delta_C\left(\Gamma\otimes\alpha_1\otimes\dots\otimes\alpha_e\right):=
\sum_{a\in E}
 c_a\left(\Gamma\otimes\alpha_1\otimes\dots\otimes\alpha_e\right).
\end{equation}

\begin{rem}
Instead of special definition of contracting edges with 1-valent ends in \eqref{def:noleaves} it is possible to continue with the original simple contraction $\hat c''_a$, define signed $\hat c'_a$ and $\hat c_a$ in the same way as $c'_a$ and $c_a$, and at the end subtract 1-valent vertices as follows:
\begin{equation}
\delta_C(\Gamma):=
\sum_{a\in E}\hat c_a(\Gamma)-\sum_{\substack{x\in V\\x\text{ 1-valent}}}\hat c_{a(x)}(\Gamma)
\end{equation}
where $a(x)$ is the only edge attached to $x$.
\end{rem}

\begin{prop}
The map $\delta_C$ commutes with the action of $S_e$ that permutes edges, and with the action of $S_v$ that permutes vertices of sign $\mu$, i.e.
\begin{itemize}
\item $\delta_C\chi\Gamma=\chi\delta_C\Gamma$ for every $\chi\in S_e$ and $\Gamma\in \bar\mV_v\bar\mE_e^\Sigma\Gra$,
\item $\delta_C\chi\Gamma=\chi\delta_C\Gamma$ for every $\chi\in S_v$ and $\Gamma\in \bar\mV_v\bar\mE_e^\Sigma\Gra$.
\end{itemize}
\end{prop}
\begin{proof}
Straightforward verification of signs.
\end{proof}

Because of the proposition we can now well define the map
\begin{equation}
\delta_C:\mV_v^\mu\mE_e^\Sigma\Gra\rightarrow\mV_{v-1}^\mu\mE_{e-1}^\Sigma\Gra.
\end{equation}

The same map is analogously defined on the space of multi-directed, sourced or oriented graphs with permuted edges and vertices, as well as their subspaces.

\begin{rem}
\label{rem:last}
The signs in the formulas \eqref{def:CaV} and \eqref{def:CaE} are set such that contracting the last edge $e$ heading towards the last vertex $v$, that is the case when no shifting of labels is necessarily after removing the edge and the vertex, has + sign.
Therefore, to calculate the proper signs of terms in $\delta_C(\Gamma)$ for $\Gamma\in\mV_v^\mu\mE_e^\Sigma\Gra$, we can, for every edge, simply chose the representative of $\Gamma$ in $\bar\mV_v\bar\mE_e^\Sigma\Gra$ where this edge is the last one, and the vertex it heads to is the last one. If it is another representative, we permute edges and vertices to get the proper representative, and the sign comes from the permutation.
\end{rem}

\begin{rem}
It is easy to see that $\delta_C$ sends connected graphs to connected graphs and saves the minimum valence condition, so it is well defined also on $\mV^\mu\mE^\Sigma\Grac$ and $\mV^{\mu,\geq i}\mE^\Sigma\Gra$.
\end{rem}

\subsection{The core differential}
\label{ss:tcd}

One gets the real theory of graph complexes after introducing the \emph{core differential}, that is the differential that changes the core graph. More of them are possible, but in this paper we consider only the core differential that comes from the edge contraction.

To do it properly, we need to introduce the grading on the space of core graphs such that the degree of the edge contraction is $-1$. Edge contraction $\delta_C$ deletes an edge, say of type  $\alpha\in\Sigma$ of homogeneous degree, and a vertex. To ensure the degree of $\delta_C$ to be homogeneously $-1$, we give degrees to vertices in the sense of the following definition. 

\begin{defi}[Graded core graph space]
Let $v>0$, $e\geq 0$ and $n$ be integers.
Let
\begin{equation}
\bar\mV_v\bar\mE_e\Gra_n:=\bar\mV_v\bar\mE_e\Gra[vn-n]
\end{equation}
and
\begin{equation}
\bar\mV\bar\mE\Gra_n:=\bigoplus_{v,e}\bar\mV_v\bar\mE_e\Gra_n.
\end{equation}
The number $n$ is called \emph{graph complex parameter}.
\end{defi}

The definition actually gives a degree $n$ to each vertex, and gives a total degree shift $-n$ to every graph, such that one-vertex graphs have degree 0.
Space of edge types $\Sigma$ is already a graded vector space, so now we can introduce the grading on the space of graphs. The notation of the graded graph spaces is given by adding a graph complex parameter in the subscript to $\Gra$, such as $\bar\mV\bar\mE^\Sigma\Gra_n$.

The contracting an edge $\delta_C$ deletes a vertex, decreasing the degree by $n$, and an edge of some type $\alpha$, decreasing the degree by $\deg(\alpha)$. The total degree has to decrease by $1$, so the degree of $\alpha$ has to be $1-n$. Therefore $C:\Sigma\rightarrow\K$ is concentrated in the degree $1-n$, i.e.\ $C(\alpha)=0$ for a homogeneous $\alpha\in\Sigma$ unless $\deg(\alpha)=1-n$.

Another condition that must be fulfilled for the differential is that $\delta_C^2=0$. This can be ensured by setting that vertices and edges have the opposite parity while permuting, as stated in the following proposition.

\begin{prop}
Let $n$ be an integer and let $\mu=(-1)^n$. Then $\delta_C^2=0$ on $\mV^\mu\mE^\Sigma\Gra_n$.
\end{prop}
\begin{proof}
The map $\delta_C^2$ contracts two edges in $\Gamma\in\mV^\mu\mE^\Sigma\Gra_n$, summed over all edges twice. We claim that contracting one edge and then contracting another one cancels with contracting the same two edges in the opposite order.

We choose representative of $\Gamma$ where the edges being deleted have last two labels $e$ and $e-1$, heading towards vertices $v$ and $v-1$ respectively. We leave to the reader the case when both edges head towards the same vertex.

According to Remark \ref{rem:last} contracting those two edges in the order $c''_{e-1}c''_e$ gives the + sign. To contract them in the opposite order we chose another representative by switching $e$ and $e-1$, as well as $v$ and $v-1$, so that the action is again $c''_{e-1}c''_e$ with + sign. But switching edges is a permutation that gives the sign $(-1)^{1-n}$ because edges, at least those that give non-zero differential, are of degree $1-n$. Switching vertices is the permutation that gives the sign $\mu=(-1)^n$, so the total sign of the contractions done in the opposite order is $(-1)^{1-n}(-1)^n=-1$, what we wanted to show.

We leave to the reader to show that conditions from Formula \eqref{def:noleaves} do not change the argument.
\end{proof}

On the spaces with $\Gra_n$ where the graph complex parameter $n$ is given, we assume
\begin{equation}
\label{def:mu_n}
\mu=(-1)^n
\end{equation}
and from now on consider actions of $S_v$ onto $\bar\mV\mE_e^\Sigma\Gra_n$ as it acts on $\bar\mV\mE_e^\Sigma\Gra_n\otimes\sgn_V^{(-1)^n}$:
\begin{equation}
\label{eq:sgnV}
\chi^{new}(\Gamma)=\sgn(\chi)^n\chi(\Gamma),
\end{equation}
such that it is
\begin{equation}
\label{eq:actSv}
\mV_v^\mu\bar\mE_e^\Sigma\Gra=
\left(\bar\mV_v\bar\mE_e^\Sigma\Gra\right)_{S_v}
\end{equation}
instead of \eqref{eq:sgnSv}. We also skip the label $\mu$, shortening e.g.\ the notation $\mV_v^\mu\mE_e^\Sigma\Gra_n$ to $\mV_v\mE_e^\Sigma\Gra_n$. Furthermore, instead of the notation $\mV\mE^\Sigma\Gra_n$ where there is no more data attached to the symbol $\mV$ we skip it and use the simple notation $\mE^\Sigma\Gra_n$.

\subsection{Combined differential}
\label{ss:cd}

\begin{prop}
Let $\mE^\Sigma\Gra_n$ be a graph complex with space of edge types dg $\langle S_2\rangle$ module $(\Sigma,\delta_E)$ with contraction $C:\Sigma\rightarrow\K$ such that $C\delta_E=0$.
Then
$$
\delta_E\delta_C=(-1)^{1-n}\delta_C\delta_E.
$$
\end{prop}
\begin{proof}
We have
$$
\delta_E\delta_C(\Gamma)=\delta_E\left(\sum_{a\in E(\Gamma)}c_a(\Gamma)\right)=
\sum_{b\in E(c_a(\Gamma))}\sum_{a\in E(\Gamma)}\delta_E^{(b)}(c_a(\Gamma)).
$$
$E(c_a(\Gamma))$ has actually all edges from $E(\Gamma)$ except $a$, with shifted labels. Careful verification of signs imply that
$$
\delta_E\delta_C(\Gamma)=
(-1)^{1-n}\sum_{a\in E(\Gamma)}\sum_{\substack{b\in E(\Gamma)\\x\neq a}}c_a\left(\delta_E^{(b)}(\Gamma)\right).
$$
On the other side
$$
\delta_C\delta_E(\Gamma)=
\sum_{a\in E(\Gamma)}\sum_{b\in E(\Gamma)}c_a\left(\delta_E^{(b)}(\Gamma)\right)=
(-1)^{1-n}\delta_E\delta_C(\Gamma)+\sum_{a\in E(\Gamma)}c_a\left(\delta_E^{(a)}(\Gamma)\right).
$$
The latter term is 0 because $C\delta_E=0$.
\end{proof}

We will always assume $C\delta_E=0$. The proposition enables us to introduce the combined differential
\begin{equation}
\delta:=\delta_C+(-1)^{n\deg}\delta_E:\Gamma\mapsto\delta_C(\Gamma)+(-1)^{n\deg(\Gamma)}\delta_E(\Gamma)
\end{equation}
on $\mE^\Sigma\Gra_n$.

\subsection{Splitting complexes and convergence of spectral sequence}
\label{ss:scacoss}

We often want that spectral sequence converges to the homology of the starting complex. In that case we say that spectral sequence converges \emph{correctly}.
For ensuring the correct convergence of a spectral sequence standard arguments are used, such as those from \cite[Appendix C]{DGC1}. E.g.\ we want spectral sequence to be bounded in each degree.

The differential will never change the \emph{loop number} $b:=e-v$, so the complexes split as the direct sum
\begin{equation}
\mE^\Sigma\Gra_n=\bigoplus_{b\in\Z}\mB_b\mE^\Sigma\Gra_n,
\end{equation}
where the prefix $\mB_b$ means
\begin{equation}
\mB_b=\sum_v\mV_v\mE_{b+v},
\end{equation}
so $\mB_b\mE^\Sigma\Gra_n=\sum_v\mV_v\mE_{b+v}^\Sigma\Gra_n$ is the sub-complex of $\mE^\Sigma\Gra_n$ with fixed $b=e-v$. All complexes in this paper can be split like this.

To show that a spectral sequence of the complex that is equal to the direct sum of simpler complexes converges correctly it is enough to show the statement for the complexes in the sum. Sub-complexes with fixed loop number mentioned above will often have bounded spectral sequences. That easily implies their correct convergence, and hence the correct convergence of the whole complex.

\subsection{Simple complexes without edge differential}
\label{ss:scwed}

In this subsection we introduce the simplest complexes with space of edge types from Example \ref{ex:s2mod}.

\begin{defi} [Kontsevich's full graph complex]
Let $n\in\Z$. \emph{Kontsevich's full graph complex} $\fGC_n$ is
$$
\fGC_n:=\mE^{\Sigma^\mu[1-n]}\Gra_n
$$
where $\mu=(-1)^n$, with the differential $\delta=\delta_C$ where $C:\Sigma^\mu\rightarrow\K$ sends a basis element $\alpha\in\Sigma$ to 1.
\end{defi}

Clearly, there is no edge differential in Kontsevich's complex. There is only one type of edges. The degree and the parity of the edge is chosen such that there is non-trivial contraction differential.

Kontsevich's graph complex $\fGC_n$ is originaly defined by M. Kontsevich in \cite{Kont1}, \cite{Kont2}, cf.\ \cite{grt} and \cite{eulerchar}. The complex defined in some of those papers is the dual of the one we define here, the differential being the vertex splitting, but it does not change the homology. Since the complex is defined as a formal vector space over the basis of graphs, the dual can be identified with the complex as the vector space. The real difference is in the differential.

We also define the multi-directed, oriented and sourced version of Kontsevich's graph complex.

\begin{defi}
\label{defi:OSDfGC}
Let $n\in\Z$, $j,k,\ell\geq 0$. The \emph{$j$-oriented $k$-sourced $\ell$-directed full graph complex} $\mO_j\mS_k\mD_\ell\fGC_n$ is
$$
\mO_j\mS_k\mD_\ell\fGC_n:=\mO_j\mS_k\mD_\ell\mE^{\Sigma^\mu[1-n]}\Gra_n
$$
where $\mu=(-1)^n$, with the differential $\delta=\delta_C$ where $C:\Sigma^\mu\rightarrow\K$ sends a basis element $\alpha\in\Sigma$ to 1.
\end{defi}

In this paper we will only consider complexes spanned by connected graphs $\mO_j\mS_k\mD_\ell\fGCc_n\subset\mO_j\mS_k\mD_\ell\fGC_n$. It is indeed $\mO_j\mS_k\mD_\ell\fGCc_n:=\mO_j\mS_k\mD_\ell\mE^{\Sigma^\mu[1-n]}\Grac_n$.

\section{Simplification of graph complexes}
\label{s:Simpl}

In this section we study some simpler versions (often sub-complexes) of the full graph complex $\mO_j\mS_k\mD_\ell\fGC_n$ needed in the paper.
The results of the first two subsections are generalization of \cite[Proposition 3.4]{grt} and \cite[Propositions 6 and 7(1)]{Multi}, and use essentially the same ideas for proving.

\subsection{At least 2-valent vertices}
\label{ss:al2vv}

Let us consider $\mO_j\mS_k\mD_\ell\fGCc_n^{\geq 2}$, the connected $j$-oriented $k$-sourced $\ell$-directed full graph complex with vertices at least 2-valent vertices.

\begin{defi}[Passing vertex]
Let $\Gamma\in\mO_j\mS_k\mD_\ell\bar\mV\bar\mE\gra$ be a $j$-oriented $k$-sourced $\ell$-directed core graph.

A 2-valent vertex $x$ is a \emph{passing vertex} if it is the head of one edge and the tail of another for every colour in $J\cup K\cup L$. I.e.\ let $(a,\mu)\neq(b,\nu)$, for $a,b\in E$ and $\mu,\nu\in\{+,-\}$, be two edge ends such that $\Gamma_\mu(a)=\Gamma_\nu(b)=x$. Then $x$ is passing if it is 2-valent and for every $c\in J\cup K\cup L$ it holds that $\mu o_c(a)\nu o_c(b)=-$.
\end{defi}

\begin{prop}
\label{prop:2}
$$
H\left(\mO_j\mS_k\mD_\ell\fGCc_n,\delta\right)=H\left(\mO_j\mS_k\mD_\ell\fGCc_n^{\geq 2},\delta\right)
$$
\end{prop}
\begin{proof}
The differential can neither create nor destroy 1-valent or isolated vertices. Therefore we have direct sum of complexes
$$
\mO_j\mS_k\mD_\ell\fGCc_n=\mO_j\mS_k\mD_\ell\fGCc_n^{\geq 2}\oplus \mO_j\mS_k\mD_\ell\fGCc_n^{1}
$$
where $\mO_j\mS_k\mD_\ell\fGCc_n^{1}$ is the sub-complex of graphs containing at least one 1-valent vertex, including the single vertex graph. It is enough to prove that $\mO_j\mS_k\mD_\ell\fGCc_n^{1}$ is acyclic.

We call \emph{the antenna} a maximal connected sub-graph consisting of 1-valent and passing vertices in a graph.

We set up a spectral sequence on $\mO_j\mS_k\mD_\ell\fGCc_n^{1}$ on the number of edges that are not in an antenna. The spectral sequence is bounded, and hence converges correctly. The first differential is the one retracting an antenna. There is a homotopy that extends an antenna (summed over all antennas) that leads to the conclusion that the first differential is acyclic, and hence is the whole differential.
\end{proof}

\subsection{No passing vertices}
\label{ss:npv}

\begin{defi}
Let $\bar\mV\bar\mE\grac^{\circ}\subset\bar\mV\bar\mE\grac^{\geq 2}$ be the set of connected core graphs with only 2-valent vertices and let $\bar\mV\bar\mE\grac^{\varnothing}\subset\bar\mV\bar\mE\grac^{\geq 2}$ be spanned by connected graphs that have at least one vertex that is at least 3-valet.
\end{defi}

Using this notations we define graded spaces $\mO_j\mS_k\mD_\ell\fGCc_n^{\circ}\subset\mO_j\mS_k\mD_\ell\fGCc_n^{\geq 2}$ and $\mO_j\mS_k\mD_\ell\fGCc_n^{\varnothing}\subset\break\mO_j\mS_k\mD_\ell\fGCc_n^{\geq 2}$.
The differential $\delta$ does not change the properties on graphs, so they form sub-complexes and
\begin{equation}
\left(\mO_j\mS_k\mD_\ell\fGCc_n^{\geq 2},\delta\right)=\left(\mO_j\mS_k\mD_\ell\fGCc_n^{\circ},\delta\right)\oplus\left(\mO_j\mS_k\mD_\ell\fGCc_n^{\varnothing},\delta\right).
\end{equation}

The part $\mO_j\mS_k\mD_\ell\fGCc_n^{\circ}$ spanned by the graphs with only 2-valent vertices consists only of loops. They are in general easy to study and they will not be of our interest in this paper. On the other side, we continue with the definition.

\begin{defi}
Let $\mO_j\mS_k\mD_\ell\bar\mV\bar\mE\grac^{\varnothing\rightarrow}\subset\mO_j\mS_k\mD_\ell\bar\mV\bar\mE\grac^{\varnothing}$ be the subset of core graphs with at least one passing vertex and let $\mO_j\mS_k\mD_\ell\bar\mV\bar\mE\grac^{\varnothing\nrightarrow}\subset\mO_j\mS_k\mD_\ell\bar\mV\bar\mE\grac^{\varnothing}$ be the subset of core graphs with no passing vertices.
\end{defi}

Using this notations we define graded spaces $\mO_j\mS_k\mD_\ell\fGCc_n^{\varnothing\rightarrow}\subset\mO_j\mS_k\mD_\ell\fGCc_n^{\varnothing}$ and $\mO_j\mS_k\mD_\ell\fGCc_n^{\varnothing\nrightarrow}\subset\break\mO_j\mS_k\mD_\ell\fGCc_n^{\varnothing}$. The differential $\delta$ can not destroy the last passing vertex, so the first one forms a sub-complex. But $\delta$ can produce a passing vertex, so $\mO_j\mS_k\mD_\ell\fGCc_n^{\varnothing\nrightarrow}$ is not the sub-complex. However, we still talk about the complex $\left(\mO_j\mS_k\mD_\ell\fGCc_n^{\varnothing\nrightarrow},\delta\right)$ where we identify graphs that have passing vertex with 0. In the other words, it is the quotient
\begin{equation}
\left(\mO_j\mS_k\mD_\ell\fGCc_n^{\varnothing\nrightarrow},\delta\right):=\left(\mO_j\mS_k\mD_\ell\fGCc_n^{\varnothing},\delta\right)/
\left(\mO_j\mS_k\mD_\ell\fGCc_n^{\varnothing\rightarrow},\delta\right).
\end{equation}

We will often use the shorter notation
\begin{equation}
\mO_j\mS_k\mD_\ell\GC_n:=\mO_j\mS_k\mD_\ell\fGCc_n^{\varnothing\nrightarrow}.
\end{equation}
called simply \emph{$j$-oriented $k$-sourced $\ell$-directed (ordinary) graph complex}.

For $j=k=\ell=0$ every 2-valent vertex is a passing vertex, so $\fGCc_n^{\varnothing\nrightarrow}\cong\fGCc_n^{\geq 3}$, abbreviated by the standard notation $\GC_n$, called simply \emph{(Kontsevich's) graph complex}.

\begin{prop}
For all $j,k,\ell\geq 0$ it holds that $H\left(\mO_j\mS_k\mD_\ell\fGCc_n^{\varnothing},\delta\right)=H\left(\mO_j\mS_k\mD_\ell\GC_n,\delta\right)$,
\end{prop}
\begin{proof}
It is enough to prove that $\mO_j\mS_k\mD_\ell\fGCc_n^{\varnothing\rightarrow}$ is acyclic.

We set up a spectral sequence on $\mO_j\mS_k\mD_\ell\fGCc_n^{\varnothing\rightarrow}$ on the number of non-passing vertices. The spectral sequence obviously converges correctly. The first differential decreases the number of passing vertices by one. There is a homotopy that extends the string of neighbouring passing vertices by one, summed over all such strings, showing the acyclicity.
\end{proof}

\subsection{Graph complexes with special direction}
\label{ss:gcwsd}

Our goal is to construct two quasi-isomorphisms $g:\mO_j\mS_k\mD_\ell\GC\break\rightarrow\mO_j\mS_k\mD_{\ell+1}\GC_n$ and $h:\mO_j\mS_k\mD_\ell\GC\rightarrow\mO_{j+1}\mS_k\mD_\ell\GC_{n+1}$. Since there is an inclusion 
$\mO_{j+1}\mS_k\mD_\ell\GC_{n+1}\hookrightarrow\break\mO_j\mS_{k+1}\mD_\ell\GC_{n+1}$, the latter map can also be seen as $h:\mO_j\mS_k\mD_\ell\GC\rightarrow\mO_j\mS_{k+1}\mD_\ell\GC_{n+1}$.

In all cases the codomain has one more colour. To make maps smoother, we will not use a new colour, but instead of the last colour we will use edge types from $\Sigma^{fix}=\langle\Ed,\dE\rangle$ from Exercise \ref{ex:s2mod}, thus partially using the reasoning from Remark \ref{rem:ColType}. In drawings, the extra direction will be drawn in black.
Let
\begin{equation}
\left(\mD\mO_j\mS_k\mD_\ell\fGC_n,\delta\right)=\left(\mO_j\mS_k\mD_\ell\mE^{\Sigma^{fix}[1-n]}\Gra_n,\delta\right),
\end{equation}
with the differential $\delta=\delta_C$ where $C:\Sigma^{fix}\rightarrow\K$ is defined by $C(\Ed)=1$, while Condition \eqref{con:CSigma}, cf.\ \eqref{def:mu_n}, implies that $C(\dE)=(-1)^n$.

\begin{prop}
\label{prop:DOSD}
There is an isomorphism
\begin{equation}
\eta:\mD\mO_j\mS_k\mD_\ell\bar\mV\bar\mE\fGC_n\rightarrow\mO_j\mS_k\mD_{\ell+1}\bar\mV\bar\mE\fGC_n
\end{equation}
that defines an isomorphism of complexes
\begin{equation}
\eta:\left(\mD\mO_j\mS_k\mD_\ell\fGC_n,\delta\right)\rightarrow\left(\mO_j\mS_k\mD_{\ell+1}\fGC_n,\delta\right).
\end{equation}
\end{prop}
\begin{proof}
Straightforward, cf. Example \ref{ex:s2mod} and Remark \ref{rem:ColType}.
The class of $\Gamma\otimes(\Ed\otimes\dots\otimes\Ed)$ is mapped to the class of $\Gamma$ with  $o_c(a)=+$ for the new colour $c$ for every edge $a\in E$.
\end{proof}

We can pull back subspaces of $\mO_j\mS_k\mD_{\ell+1}\fGC_n$ along the isomorphism $\eta$ and define:
\begin{equation}
\mS\mO_j\mS_k\mD_\ell\fGC_n:=\eta^{-1}\left(\mD\mO_j\mS_{k+1}\mD_\ell\fGC_n\right),
\end{equation}
\begin{equation}
\mO\mO_j\mS_k\mD_\ell\fGC_n:=\eta^{-1}\left(\mD\mO_{j+1}\mS_k\mD_\ell\fGC_n\right),
\end{equation}
where the new colour is the one added to sourced, respectively oriented, colours. Furthermore
\begin{equation}
\label{def:DOSDGC}
\mD\mO_j\mS_k\mD_\ell\GC_n:=\eta^{-1}\left(\mO_j\mS_k\mD_{\ell+1}\GC_n\right),
\end{equation}
\begin{equation}
\mS\mO_j\mS_k\mD_\ell\GC_n:=\eta^{-1}\left(\mO_j\mS_{k+1}\mD_\ell\GC_n\right),
\end{equation}
\begin{equation}
\mO\mO_j\mS_k\mD_\ell\GC_n:=\eta^{-1}\left(\mO_{j+1}\mS_k\mD_\ell\GC_n\right),
\end{equation}
where the new colour is the one added to directed, sourced, respectively oriented, colours.

Here one needs to be a bit careful with the notion of passing vertex.
The spaces on the right-hand side are the spaces spanned by graphs without passing vertices, e.g.
\begin{equation}
\mO_j\mS_k\mD_{\ell+1}\GC_n=\mO_j\mS_k\mD_{\ell+1}\mE^{\Sigma^{(-1)^n}[1-n]}\Grac_n^{\varnothing\nrightarrow}.
\end{equation}
Here, we disallow vertices that are passing in all $j+k+\ell+1$ colours. After pull-back with $\eta^{-1}$ we disallow vertices that are passing in all $j+k+\ell$ colours and also in the type-direction. So it is
\begin{equation}
\mD\mO_j\mS_k\mD_\ell\GC_n\neq\mO_j\mS_k\mD_\ell\mE^{\Sigma^{fix}[1-n]}\Grac_n^{\varnothing\nrightarrow}
\end{equation}
because here on the right-hand side we disallow all vertices that are passing in $j+k+\ell$ colours, even if they are not passing in type-direction. This motivates the following definition.

\begin{defi}
Let $x$ be a vertex of a graph $\Gamma\in\mD\mO_j\mS_k\mD_\ell\bar\mV\bar\mE\fGC_n$.

Vertex $x$ is \emph{weakly passing} if it is passing in the underlying core graph in $\mD\mO_j\mS_k\mD_\ell\bar\mV\bar\mE\Gra_n$.

Vertex $x$ is \emph{strongly passing} if the corresponding vertex in $\eta(\Gamma)$ is passing in the underlying core graph in $\mO_j\mS_k\mD_{\ell+1}\bar\mV\bar\mE\Gra_n$.
\end{defi}

Let us now fix $j,k,\ell$ and let the prefix $\mM$ (for \emph{multi}) abbreviate $\mO_j\mS_k\mD_\ell$.
After Proposition \ref{prop:DOSD} our goal is to construct quasi-isomorphisms $g:\MGC\rightarrow\mD\mM\GC_n$ and $h:\MGC\rightarrow\mO\mM\GC_{n+1}\hookrightarrow\mS\mM\GC_{n+1}$.

\subsection{Skeleton graph complexes}
\label{ss:sgc}
Recall that $\mD\MGC_n\subset\mM\mE^{\Sigma^{fix}[1-n]}\Gra_n$ is the complex spanned by connected graphs with all vertices at least 2-valent, at least one vertex more than 2-valent, and no strongly passing vertices. $\mS\MGC_n$ and $\mO\MGC_n$ are its sub-complexes spanned by graphs that are sourced, respectively oriented, in the extra direction.

Vertices in $\Gamma\in\mD\mM\bar\mV\bar\mE\GC_n$ that are not weakly passing are called \emph{skeleton vertices}. String of edges and weakly-passing vertices between two skeleton vertices can be seen as an edge. We are going to construct a graph space where these will indeed be edges of types that indicate the original structure of an edge, called \emph{skeleton edge}. Precisely, let
\begin{equation}
\sigma_n^\infty=\{
\Ed,\;
\dE,\;
\EdE,\;
\dEd,\;
\begin{tikzpicture}[baseline=-.65ex,scale=.5]
 \node[nil] (a) at (0,0) {};
 \node[int] (b) at (1,0) {};
 \node[int] (c) at (2,0) {};
 \node[nil] (d) at (3,0) {};
 \draw (a) edge[-latex] (b);
 \draw (b) edge[latex-] (c);
 \draw (c) edge[-latex] (d);
\end{tikzpicture},\;
\begin{tikzpicture}[baseline=-.65ex,scale=.5]
 \node[nil] (a) at (0,0) {};
 \node[int] (b) at (1,0) {};
 \node[int] (c) at (2,0) {};
 \node[nil] (d) at (3,0) {};
 \draw (a) edge[latex-] (b);
 \draw (b) edge[-latex] (c);
 \draw (c) edge[latex-] (d);
\end{tikzpicture},\;
\begin{tikzpicture}[baseline=-.65ex,scale=.5]
 \node[nil] (a) at (0,0) {};
 \node[int] (b) at (1,0) {};
 \node[int] (c) at (2,0) {};
 \node[int] (d) at (3,0) {};
 \node[nil] (e) at (4,0) {};
 \draw (a) edge[-latex] (b);
 \draw (b) edge[latex-] (c);
 \draw (c) edge[-latex] (d);
 \draw (d) edge[latex-] (e);
\end{tikzpicture},\;
\begin{tikzpicture}[baseline=-.65ex,scale=.5]
 \node[nil] (a) at (0,0) {};
 \node[int] (b) at (1,0) {};
 \node[int] (c) at (2,0) {};
 \node[int] (d) at (3,0) {};
 \node[nil] (e) at (4,0) {};
 \draw (a) edge[latex-] (b);
 \draw (b) edge[-latex] (c);
 \draw (c) edge[latex-] (d);
 \draw (d) edge[-latex] (e);
\end{tikzpicture},\;
\dots
\},
\end{equation}
\begin{equation}
\Sigma_n^{sk}:=\langle\sigma_n^\infty\rangle.
\end{equation}
The non-trivial $\chi\in S_2$ acts on $\Sigma_n^{sk}$ like in Table \ref{tbl:rev}. Since vertices in the skeleton edge are weakly passing, for every colour $c\in J\cup K\cup L$ directions are the same for all original edges in a skeleton edge, so the direction of the skeleton edge in that colour is well defined. The type-direction has to alternate because vertices can not be strongly passing.

\begin{table}[H]
\centering
\begin{tabular}{ r  c  l }
\Ed
& $\leftrightarrow$ &
\dE
\\
\\
\EdE
& $\mapsto$ & $(-1)^{1-n}
\EdE$
\\
\\
\dEd
& $\mapsto$ & $(-1)^{1-n}
\dEd$
\\
\\
$\begin{tikzpicture}[baseline=-.65ex,scale=.5]
 \node[nil] (a) at (0,0) {};
 \node[int] (b) at (1,0) {};
 \node[int] (c) at (2,0) {};
 \node[nil] (d) at (3,0) {};
 \draw (a) edge[-latex] (b);
 \draw (b) edge[latex-] (c);
 \draw (c) edge[-latex] (d);
\end{tikzpicture}$
& $\leftrightarrow$ & $-
\begin{tikzpicture}[baseline=-.65ex,scale=.5]
 \node[nil] (a) at (0,0) {};
 \node[int] (b) at (1,0) {};
 \node[int] (c) at (2,0) {};
 \node[nil] (d) at (3,0) {};
 \draw (a) edge[latex-] (b);
 \draw (b) edge[-latex] (c);
 \draw (c) edge[latex-] (d);
\end{tikzpicture}$
\\
\\
$\begin{tikzpicture}[baseline=-.65ex,scale=.5]
 \node[nil] (a) at (0,0) {};
 \node[int] (b) at (1,0) {};
 \node[int] (c) at (2,0) {};
 \node[int] (d) at (3,0) {};
 \node[nil] (e) at (4,0) {};
 \draw (a) edge[-latex] (b);
 \draw (b) edge[latex-] (c);
 \draw (c) edge[-latex] (d);
 \draw (d) edge[latex-] (e);
\end{tikzpicture}$
& $\mapsto$ & $(-1)^n
\begin{tikzpicture}[baseline=-.65ex,scale=.5]
 \node[nil] (a) at (0,0) {};
 \node[int] (b) at (1,0) {};
 \node[int] (c) at (2,0) {};
 \node[int] (d) at (3,0) {};
 \node[nil] (e) at (4,0) {};
 \draw (a) edge[-latex] (b);
 \draw (b) edge[latex-] (c);
 \draw (c) edge[-latex] (d);
 \draw (d) edge[latex-] (e);
\end{tikzpicture}$
\\
\\
$\begin{tikzpicture}[baseline=-.65ex,scale=.5]
 \node[nil] (a) at (0,0) {};
 \node[int] (b) at (1,0) {};
 \node[int] (c) at (2,0) {};
 \node[int] (d) at (3,0) {};
 \node[nil] (e) at (4,0) {};
 \draw (a) edge[latex-] (b);
 \draw (b) edge[-latex] (c);
 \draw (c) edge[latex-] (d);
 \draw (d) edge[-latex] (e);
\end{tikzpicture}$
& $\mapsto$ & $(-1)^n
\begin{tikzpicture}[baseline=-.65ex,scale=.5]
 \node[nil] (a) at (0,0) {};
 \node[int] (b) at (1,0) {};
 \node[int] (c) at (2,0) {};
 \node[int] (d) at (3,0) {};
 \node[nil] (e) at (4,0) {};
 \draw (a) edge[latex-] (b);
 \draw (b) edge[-latex] (c);
 \draw (c) edge[latex-] (d);
 \draw (d) edge[-latex] (e);
\end{tikzpicture}$
\\
\\
& $\dots$ &
\end{tabular}
\caption {\label{tbl:rev}
The action of the non-trivial $\chi\in S_2$ on $\Sigma_n^{sk}$ defined on the basis $\sigma_n^\infty$. We take the representative of the original graph where original edges and weakly-passing vertices are labelled with consecutive numbers from left to right. Reversing the direction reverses the labelling, and the sign that depends on the graph complex parameter $n$ appears as written.
}
\end{table}

$\Sigma_n^{sk}$ is easily graded by the number $t$ of original edges in it. We call that number the \emph{length} of a skeleton edge. It also has a homological grading that comes from degrees of vertices and original edges in it that is $t(1-n)+(t-1)n=t-n$. It also has a differential $\delta_E$ described in Table \ref{tbl:deltaE}.

\begin{table}[H]
\centering
\begin{tabular}{ r c l c r c l }
$\begin{tikzpicture}[baseline=-.65ex,scale=.45]
 \node[nil] (a) at (0,0) {};
 \node[int] (b) at (1,0) {};
 \node[nil] (c) at (2,0) {};
 \draw (a) edge[-latex] (b);
 \draw (b) edge[latex-] (c);
\end{tikzpicture}
$ & $\mapsto$ & $
\begin{tikzpicture}[baseline=-.65ex,scale=.45]
 \node[nil] (a) at (0,0) {};
 \node[nil] (b) at (1,0) {};
 \draw (a) edge[-latex] (b);
\end{tikzpicture}
-(-1)^n
\begin{tikzpicture}[baseline=-.65ex,scale=.45]
 \node[nil] (a) at (0,0) {};
 \node[nil] (b) at (1,0) {};
 \draw (a) edge[latex-] (b);
\end{tikzpicture}$
& &
$\begin{tikzpicture}[baseline=-.65ex,scale=.45]
 \node[nil] (a) at (0,0) {};
 \node[int] (b) at (1,0) {};
 \node[nil] (c) at (2,0) {};
 \draw (a) edge[latex-] (b);
 \draw (b) edge[-latex] (c);
\end{tikzpicture}
$ & $\mapsto$ & $-
\begin{tikzpicture}[baseline=-.65ex,scale=.45]
 \node[nil] (a) at (0,0) {};
 \node[nil] (b) at (1,0) {};
 \draw (a) edge[-latex] (b);
\end{tikzpicture}
+(-1)^n
\begin{tikzpicture}[baseline=-.65ex,scale=.45]
 \node[nil] (a) at (0,0) {};
 \node[nil] (b) at (1,0) {};
 \draw (a) edge[latex-] (b);
\end{tikzpicture}$
\\
\\
$\begin{tikzpicture}[baseline=-.65ex,scale=.45]
 \node[nil] (a) at (0,0) {};
 \node[int] (b) at (1,0) {};
 \node[int] (c) at (2,0) {};
 \node[nil] (d) at (3,0) {};
 \draw (a) edge[-latex] (b);
 \draw (b) edge[latex-] (c);
 \draw (c) edge[-latex] (d);
\end{tikzpicture}
$ & $\mapsto$ & $
\begin{tikzpicture}[baseline=-.65ex,scale=.45]
 \node[nil] (a) at (0,0) {};
 \node[int] (b) at (1,0) {};
 \node[nil] (c) at (2,0) {};
 \draw (a) edge[-latex] (b);
 \draw (b) edge[latex-] (c);
\end{tikzpicture}
+
\begin{tikzpicture}[baseline=-.65ex,scale=.45]
 \node[nil] (a) at (0,0) {};
 \node[int] (b) at (1,0) {};
 \node[nil] (c) at (2,0) {};
 \draw (a) edge[latex-] (b);
 \draw (b) edge[-latex] (c);
\end{tikzpicture}$
& &
$\begin{tikzpicture}[baseline=-.65ex,scale=.45]
 \node[nil] (a) at (0,0) {};
 \node[int] (b) at (1,0) {};
 \node[int] (c) at (2,0) {};
 \node[nil] (d) at (3,0) {};
 \draw (a) edge[latex-] (b);
 \draw (b) edge[-latex] (c);
 \draw (c) edge[latex-] (d);
\end{tikzpicture}
$ & $\mapsto$ & $ (-1)^n
\begin{tikzpicture}[baseline=-.65ex,scale=.45]
 \node[nil] (a) at (0,0) {};
 \node[int] (b) at (1,0) {};
 \node[nil] (c) at (2,0) {};
 \draw (a) edge[-latex] (b);
 \draw (b) edge[latex-] (c);
\end{tikzpicture}
+(-1)^n
\begin{tikzpicture}[baseline=-.65ex,scale=.45]
 \node[nil] (a) at (0,0) {};
 \node[int] (b) at (1,0) {};
 \node[nil] (c) at (2,0) {};
 \draw (a) edge[latex-] (b);
 \draw (b) edge[-latex] (c);
\end{tikzpicture}$
\\
\\
$\begin{tikzpicture}[baseline=-.65ex,scale=.45]
 \node[nil] (a) at (0,0) {};
 \node[int] (b) at (1,0) {};
 \node[int] (c) at (2,0) {};
 \node[int] (d) at (3,0) {};
 \node[nil] (e) at (4,0) {};
 \draw (a) edge[-latex] (b);
 \draw (b) edge[latex-] (c);
 \draw (c) edge[-latex] (d);
 \draw (d) edge[latex-] (e);
\end{tikzpicture}
$ & $\mapsto$ & $
\begin{tikzpicture}[baseline=-.65ex,scale=.45]
 \node[nil] (a) at (0,0) {};
 \node[int] (b) at (1,0) {};
 \node[int] (c) at (2,0) {};
 \node[nil] (d) at (3,0) {};
 \draw (a) edge[-latex] (b);
 \draw (b) edge[latex-] (c);
 \draw (c) edge[-latex] (d);
\end{tikzpicture}
-(-1)^n
\begin{tikzpicture}[baseline=-.65ex,scale=.45]
 \node[nil] (a) at (0,0) {};
 \node[int] (b) at (1,0) {};
 \node[int] (c) at (2,0) {};
 \node[nil] (d) at (3,0) {};
 \draw (a) edge[latex-] (b);
 \draw (b) edge[-latex] (c);
 \draw (c) edge[latex-] (d);
\end{tikzpicture}$
& &
$\begin{tikzpicture}[baseline=-.65ex,scale=.45]
 \node[nil] (a) at (0,0) {};
 \node[int] (b) at (1,0) {};
 \node[int] (c) at (2,0) {};
 \node[int] (d) at (3,0) {};
 \node[nil] (e) at (4,0) {};
 \draw (a) edge[latex-] (b);
 \draw (b) edge[-latex] (c);
 \draw (c) edge[latex-] (d);
 \draw (d) edge[-latex] (e);
\end{tikzpicture}
$ & $\mapsto$ & $ -
\begin{tikzpicture}[baseline=-.65ex,scale=.45]
 \node[nil] (a) at (0,0) {};
 \node[int] (b) at (1,0) {};
 \node[int] (c) at (2,0) {};
 \node[nil] (d) at (3,0) {};
 \draw (a) edge[-latex] (b);
 \draw (b) edge[latex-] (c);
 \draw (c) edge[-latex] (d);
\end{tikzpicture}
+(-1)^n
\begin{tikzpicture}[baseline=-.65ex,scale=.45]
 \node[nil] (a) at (0,0) {};
 \node[int] (b) at (1,0) {};
 \node[int] (c) at (2,0) {};
 \node[nil] (d) at (3,0) {};
 \draw (a) edge[latex-] (b);
 \draw (b) edge[-latex] (c);
 \draw (c) edge[latex-] (d);
\end{tikzpicture}$
\\
\\
& $\dots$ & & & & $\dots$ & 
\end{tabular}
\caption {\label{tbl:deltaE}
The differential $\delta_E$ on $\Sigma_n^{sk}$ defined on the basis $\sigma_n^\infty$.}
\end{table}

\begin{defi}[Skeleton graph complex]
The \emph{skeleton graph complex} is
\begin{equation}
\left(\mD^{sk}\MGC_n,\delta\right):=
\left(\mM\mE^{\Sigma_n^{sk}}\Grac_n^{\varnothing\nrightarrow},\delta\right),
\end{equation}
where $\delta=\delta_C+(-1)^{n\deg}\delta_E$, $\delta_E$ is the edge differential that comes from the differential on $\Sigma_n^{sk}$ and $\delta_C$ is the core differential that comes from $C:\Sigma_n^{sk}\rightarrow\K$ defined by $C(\Ed)=1$, $C(\dE)=(-1)^n$ and $C(\alpha)=0$ for other $\alpha\in\sigma^\infty$.
\end{defi}

\begin{prop}
\label{prop:SkEq}
There is an isomorphism of complexes
\begin{equation}
\kappa:\left(\mD^{sk}\MGC_n,\delta\right)\rightarrow
\left(\mD\MGC_n,\delta\right).
\end{equation}
\end{prop}
\begin{proof}
The skeleton complex $\left(\mD^{sk}\MGC_n,\delta_E+(-1)^{n\deg}\delta_C\right)$ is constructed exactly to fulfil this proposition. Details are left to the reader. 
\end{proof}

We can pull back sub-complexes of $\mD\MGC_n$ along the isomorphism $\kappa$ and define:
\begin{equation}
\mS^{sk}\MGC_n:=\kappa^{-1}\left(\mS\MGC_n\right),
\end{equation}
\begin{equation}
\mO^{sk}\MGC_n:=\kappa^{-1}\left(\mO\MGC_n\right).
\end{equation}

\subsection{A new basis for edge types}
\label{ss:anbfet}

The differential looks simpler if we use another bases $\sigma_n^0$ of $\Sigma_n^{sk}$ that consist of elements defined in Table \ref{tbl:sedg}. The action of $S_2$ is derived in Table \ref{tbl:rev2} and the differential $\delta_E$ in Table \ref{tbl:difsknew}.

\begin{table}[H]
\centering
\begin{tabular}{ c  c  c }
$
\Es
:=\frac{1}{2}\left(
\Ed
+(-1)^n
\dE
\right)$
& $\quad$ &
$
\begin{tikzpicture}[baseline=-.65ex,scale=.5]
 \node[nil] (a) at (0,0) {};
 \node[nil] (c) at (1,0) {};
 \draw (a) edge[dotted,->] (c);
\end{tikzpicture}
:=
\Ed
-(-1)^n
\dE$
\\
\\
$
\Ess
:=\frac{1}{2}\left(
\EdE
-
\dEd
\right)$
& &
$
\begin{tikzpicture}[baseline=-.65ex,scale=.5]
 \node[nil] (a) at (0,0) {};
 \node[nil] (c) at (1.4,0) {};
 \draw (a) edge[dotted,->] (c);
 \draw (.7,.15) edge (.7,-.15);
\end{tikzpicture}
:=
\EdE
+
\dEd
$
\\
\\
$
\Esss
:=\frac{1}{2}\left(
\begin{tikzpicture}[baseline=-.65ex,scale=.5]
 \node[nil] (a) at (0,0) {};
 \node[int] (b) at (1,0) {};
 \node[int] (c) at (2,0) {};
 \node[nil] (d) at (3,0) {};
 \draw (a) edge[-latex] (b);
 \draw (b) edge[latex-] (c);
 \draw (c) edge[-latex] (d);
\end{tikzpicture}
+(-1)^n
\begin{tikzpicture}[baseline=-.65ex,scale=.5]
 \node[nil] (a) at (0,0) {};
 \node[int] (b) at (1,0) {};
 \node[int] (c) at (2,0) {};
 \node[nil] (d) at (3,0) {};
 \draw (a) edge[latex-] (b);
 \draw (b) edge[-latex] (c);
 \draw (c) edge[latex-] (d);
\end{tikzpicture}
\right)$
& &
$
\begin{tikzpicture}[baseline=-.65ex,scale=.5]
 \node[nil] (a) at (0,0) {};
 \node[nil] (c) at (1.8,0) {};
 \draw (a) edge[dotted,->] (c);
 \draw (.6,.15) edge (.6,-.15);
 \draw (1.2,.15) edge (1.2,-.15);
\end{tikzpicture}
:=
\begin{tikzpicture}[baseline=-.65ex,scale=.5]
 \node[nil] (a) at (0,0) {};
 \node[int] (b) at (1,0) {};
 \node[int] (c) at (2,0) {};
 \node[nil] (d) at (3,0) {};
 \draw (a) edge[-latex] (b);
 \draw (b) edge[latex-] (c);
 \draw (c) edge[-latex] (d);
\end{tikzpicture}
-(-1)^n
\begin{tikzpicture}[baseline=-.65ex,scale=.5]
 \node[nil] (a) at (0,0) {};
 \node[int] (b) at (1,0) {};
 \node[int] (c) at (2,0) {};
 \node[nil] (d) at (3,0) {};
 \draw (a) edge[latex-] (b);
 \draw (b) edge[-latex] (c);
 \draw (c) edge[latex-] (d);
\end{tikzpicture}$
\\
\\
$
\begin{tikzpicture}[baseline=-.65ex,scale=.5]
 \node[nil] (a) at (0,0) {};
 \node[nil] (c) at (2.2,0) {};
 \draw (a) edge[->] (c);
 \draw (.55,.15) edge (.55,-.15);
 \draw (1.1,.15) edge (1.1,-.15);
 \draw (1.65,.15) edge (1.65,-.15);
\end{tikzpicture}
:=\frac{1}{2}\left(
\begin{tikzpicture}[baseline=-.65ex,scale=.5]
 \node[nil] (a) at (0,0) {};
 \node[int] (b) at (1,0) {};
 \node[int] (c) at (2,0) {};
 \node[int] (d) at (3,0) {};
 \node[nil] (e) at (4,0) {};
 \draw (a) edge[-latex] (b);
 \draw (b) edge[latex-] (c);
 \draw (c) edge[-latex] (d);
 \draw (d) edge[latex-] (e);
\end{tikzpicture}
-
\begin{tikzpicture}[baseline=-.65ex,scale=.5]
 \node[nil] (a) at (0,0) {};
 \node[int] (b) at (1,0) {};
 \node[int] (c) at (2,0) {};
 \node[int] (d) at (3,0) {};
 \node[nil] (e) at (4,0) {};
 \draw (a) edge[latex-] (b);
 \draw (b) edge[-latex] (c);
 \draw (c) edge[latex-] (d);
 \draw (d) edge[-latex] (e);
\end{tikzpicture}
\right)$
& &
$
\begin{tikzpicture}[baseline=-.65ex,scale=.5]
 \node[nil] (a) at (0,0) {};
 \node[nil] (c) at (2.2,0) {};
 \draw (a) edge[dotted,->] (c);
 \draw (.55,.15) edge (.55,-.15);
 \draw (1.1,.15) edge (1.1,-.15);
 \draw (1.65,.15) edge (1.65,-.15);
\end{tikzpicture}
:=
\begin{tikzpicture}[baseline=-.65ex,scale=.5]
 \node[nil] (a) at (0,0) {};
 \node[int] (b) at (1,0) {};
 \node[int] (c) at (2,0) {};
 \node[int] (d) at (3,0) {};
 \node[nil] (e) at (4,0) {};
 \draw (a) edge[-latex] (b);
 \draw (b) edge[latex-] (c);
 \draw (c) edge[-latex] (d);
 \draw (d) edge[latex-] (e);
\end{tikzpicture}
+
\begin{tikzpicture}[baseline=-.65ex,scale=.5]
 \node[nil] (a) at (0,0) {};
 \node[int] (b) at (1,0) {};
 \node[int] (c) at (2,0) {};
 \node[int] (d) at (3,0) {};
 \node[nil] (e) at (4,0) {};
 \draw (a) edge[latex-] (b);
 \draw (b) edge[-latex] (c);
 \draw (c) edge[latex-] (d);
 \draw (d) edge[-latex] (e);
\end{tikzpicture}
$
\\
\\
$\dots$ & & $\dots$
\end{tabular}
\caption{\label{tbl:sedg} Definition of solid and dotted skeleton edges that form another basis $\sigma_n^0$ of $\Sigma_n^{sk}$.}
\end{table}

\begin{table}[H]
\centering
\begin{tabular}{ c c c c c c c c c c c }
\Es
& $\leftrightarrow$ &
\begin{tikzpicture}[baseline=-.65ex,scale=.5]
 \node[nil] (a) at (0,0) {};
 \node[nil] (b) at (1,0) {};
 \draw (a) edge[<-] (b);
\end{tikzpicture}
& $=$ & $(-1)^n
\Es$
& $\quad$ &
\begin{tikzpicture}[baseline=-.65ex,scale=.5]
 \node[nil] (a) at (0,0) {};
 \node[nil] (c) at (1,0) {};
 \draw (a) edge[dotted,->] (c);
\end{tikzpicture}
& $\leftrightarrow$ &
\begin{tikzpicture}[baseline=-.65ex,scale=.5]
 \node[nil] (a) at (0,0) {};
 \node[nil] (c) at (1,0) {};
 \draw (a) edge[dotted,<-] (c);
\end{tikzpicture}
& $=$ & $-(-1)^n
\begin{tikzpicture}[baseline=-.65ex,scale=.5]
 \node[nil] (a) at (0,0) {};
 \node[nil] (c) at (1,0) {};
 \draw (a) edge[dotted,->] (c);
\end{tikzpicture}$
\\
\\
\Ess
& $\leftrightarrow$ &
\begin{tikzpicture}[baseline=-.65ex,scale=.5]
 \node[nil] (a) at (0,0) {};
 \node[nil] (c) at (1.4,0) {};
 \draw (a) edge[<-] (c);
 \draw (.7,.15) edge (.7,-.15);
\end{tikzpicture}
& $=$ & $-(-1)^n
\Ess$
& &
\begin{tikzpicture}[baseline=-.65ex,scale=.5]
 \node[nil] (a) at (0,0) {};
 \node[nil] (c) at (1.4,0) {};
 \draw (a) edge[dotted,->] (c);
 \draw (.7,.15) edge (.7,-.15);
\end{tikzpicture}
& $\leftrightarrow$ &
\begin{tikzpicture}[baseline=-.65ex,scale=.5]
 \node[nil] (a) at (0,0) {};
 \node[nil] (c) at (1.4,0) {};
 \draw (a) edge[dotted,<-] (c);
 \draw (.7,.15) edge (.7,-.15);
\end{tikzpicture}
& $=$ & $-(-1)^n
\begin{tikzpicture}[baseline=-.65ex,scale=.5]
 \node[nil] (a) at (0,0) {};
 \node[nil] (c) at (1.4,0) {};
 \draw (a) edge[dotted,->] (c);
 \draw (.7,.15) edge (.7,-.15);
\end{tikzpicture}$
\\
\\
\Esss
& $\leftrightarrow$ &
\begin{tikzpicture}[baseline=-.65ex,scale=.5]
 \node[nil] (a) at (0,0) {};
 \node[nil] (c) at (1.8,0) {};
 \draw (a) edge[<-] (c);
 \draw (.6,.15) edge (.6,-.15);
 \draw (1.2,.15) edge (1.2,-.15);
\end{tikzpicture}
& $=$ & $-(-1)^n
\Esss$
& &
\begin{tikzpicture}[baseline=-.65ex,scale=.5]
 \node[nil] (a) at (0,0) {};
 \node[nil] (c) at (1.8,0) {};
 \draw (a) edge[dotted,->] (c);
 \draw (.6,.15) edge (.6,-.15);
 \draw (1.2,.15) edge (1.2,-.15);
\end{tikzpicture}
& $\leftrightarrow$ &
\begin{tikzpicture}[baseline=-.65ex,scale=.5]
 \node[nil] (a) at (0,0) {};
 \node[nil] (c) at (1.8,0) {};
 \draw (a) edge[dotted,<-] (c);
 \draw (.6,.15) edge (.6,-.15);
 \draw (1.2,.15) edge (1.2,-.15);
\end{tikzpicture}
& $=$ & $(-1)^n
\begin{tikzpicture}[baseline=-.65ex,scale=.5]
 \node[nil] (a) at (0,0) {};
 \node[nil] (c) at (1.8,0) {};
 \draw (a) edge[dotted,->] (c);
 \draw (.6,.15) edge (.6,-.15);
 \draw (1.2,.15) edge (1.2,-.15);
\end{tikzpicture}$
\\
\\
\begin{tikzpicture}[baseline=-.65ex,scale=.5]
 \node[nil] (a) at (0,0) {};
 \node[nil] (c) at (2.2,0) {};
 \draw (a) edge[->] (c);
 \draw (.55,.15) edge (.55,-.15);
 \draw (1.1,.15) edge (1.1,-.15);
 \draw (1.65,.15) edge (1.65,-.15);
\end{tikzpicture}
& $\leftrightarrow$ &
\begin{tikzpicture}[baseline=-.65ex,scale=.5]
 \node[nil] (a) at (0,0) {};
 \node[nil] (c) at (2.2,0) {};
 \draw (a) edge[<-] (c);
 \draw (.55,.15) edge (.55,-.15);
 \draw (1.1,.15) edge (1.1,-.15);
 \draw (1.65,.15) edge (1.65,-.15);
\end{tikzpicture}
& $=$ & $(-1)^n
\begin{tikzpicture}[baseline=-.65ex,scale=.5]
 \node[nil] (a) at (0,0) {};
 \node[nil] (c) at (2.2,0) {};
 \draw (a) edge[->] (c);
 \draw (.55,.15) edge (.55,-.15);
 \draw (1.1,.15) edge (1.1,-.15);
 \draw (1.65,.15) edge (1.65,-.15);
\end{tikzpicture}$
& &
\begin{tikzpicture}[baseline=-.65ex,scale=.5]
 \node[nil] (a) at (0,0) {};
 \node[nil] (c) at (2.2,0) {};
 \draw (a) edge[dotted,->] (c);
 \draw (.55,.15) edge (.55,-.15);
 \draw (1.1,.15) edge (1.1,-.15);
 \draw (1.65,.15) edge (1.65,-.15);
\end{tikzpicture}
& $\leftrightarrow$ &
\begin{tikzpicture}[baseline=-.65ex,scale=.5]
 \node[nil] (a) at (0,0) {};
 \node[nil] (c) at (2.2,0) {};
 \draw (a) edge[dotted,<-] (c);
 \draw (.55,.15) edge (.55,-.15);
 \draw (1.1,.15) edge (1.1,-.15);
 \draw (1.65,.15) edge (1.65,-.15);
\end{tikzpicture}
& $=$ & $(-1)^n
\begin{tikzpicture}[baseline=-.65ex,scale=.5]
 \node[nil] (a) at (0,0) {};
 \node[nil] (c) at (2.2,0) {};
 \draw (a) edge[dotted,->] (c);
 \draw (.55,.15) edge (.55,-.15);
 \draw (1.1,.15) edge (1.1,-.15);
 \draw (1.65,.15) edge (1.65,-.15);
\end{tikzpicture}$
\\
\\
& & $\dots$ & & & & & & $\dots$
\end{tabular}
\caption {\label{tbl:rev2}
The action of the non-trivial $\chi\in S_2$ on the new basis $\sigma_n^0$. Signs are systematized in Table \ref{tbl:par}.
}
\end{table}

\begin{table}[H]
\centering
\begin{tabular}{ c | c | c }
Length &
\begin{tikzpicture}[baseline=-.65ex,scale=.5]
 \node[nil] (a) at (0,0) {};
 \node[nil] (c) at (2.2,0) {};
 \draw (a) edge[->] (c);
 \draw (.55,.15) edge (.55,-.15);
 \node at (1.3,.15) {$\dots$};
\end{tikzpicture}
&
\begin{tikzpicture}[baseline=-.65ex,scale=.5]
 \node[nil] (a) at (0,0) {};
 \node[nil] (c) at (2.2,0) {};
 \draw (a) edge[dotted, ->] (c);
 \draw (.55,.15) edge (.55,-.15);
 \node at (1.3,.15) {$\dots$};
\end{tikzpicture}
\\
\hline
$4t$ & + & +
\\
$4t+1$ & + & -
\\
$4t+2$ & - & -
\\
$4t+3$ & - & +
\end{tabular}
\caption{\label{tbl:par} Signs of the action of the non-trivial $\chi\in S_2$ on the new basis $\sigma_n^0$, for even graph complex parameter. For odd parameter parities are the opposite.}
\end{table}

\begin{table}[H]
\centering
\begin{tabular}{ c c c c c}
\Ess
& $\mapsto$ &
\begin{tikzpicture}[baseline=-.65ex,scale=.5]
 \node[nil] (a) at (0,0) {};
 \node[nil] (c) at (1,0) {};
 \draw (a) edge[dotted,->] (c);
\end{tikzpicture}
& $\mapsto$ &
0
\\
\\
\Esss
& $\mapsto$ &
\begin{tikzpicture}[baseline=-.65ex,scale=.5]
 \node[nil] (a) at (0,0) {};
 \node[nil] (c) at (1.4,0) {};
 \draw (a) edge[dotted,->] (c);
 \draw (.7,.15) edge (.7,-.15);
\end{tikzpicture}
& $\mapsto$ &
0
\\
\\
\begin{tikzpicture}[baseline=-.65ex,scale=.5]
 \node[nil] (a) at (0,0) {};
 \node[nil] (c) at (2.2,0) {};
 \draw (a) edge[->] (c);
 \draw (.55,.15) edge (.55,-.15);
 \draw (1.1,.15) edge (1.1,-.15);
 \draw (1.65,.15) edge (1.65,-.15);
\end{tikzpicture}
& $\mapsto$ &
\begin{tikzpicture}[baseline=-.65ex,scale=.5]
 \node[nil] (a) at (0,0) {};
 \node[nil] (c) at (1.8,0) {};
 \draw (a) edge[dotted,->] (c);
 \draw (.6,.15) edge (.6,-.15);
 \draw (1.2,.15) edge (1.2,-.15);
\end{tikzpicture}
& $\mapsto$ &
0
\\
\\
& & $\dots$
\end{tabular}
\caption{\label{tbl:difsknew} The differential $\delta_E$ on the new basis $\sigma_n^0$. }
\end{table}

Moreover, we can construct a basis of $\Sigma_n^{sk}$ that has some elements of $\sigma_n^\infty$ and some element of $\sigma_n^0$. Particularly,
\begin{equation}
\sigma_n^t:=\{\alpha\in\sigma_n^\infty|\text{length of $\alpha$ is }\leq t\}\cup
\{\alpha\in\sigma_n^0|\text{length of $\alpha$ is }>t\}
\end{equation}
for $t\in\N$ is also a basis of $\Sigma_n^{sk}$.

\subsection{Base skeleton graphs}
\label{ss:bsg}

The space of edge types $\Sigma_n^{sk}$ is given by its bases $\sigma_n^t$ for $t\in\N\cup\{\infty\}$. Recall from \eqref{eq:gen1} and \eqref{eq:gen2} that in this case the set
\begin{equation}
\mM\bar\mV\bar\mE^{\sigma_n^t}\grac^{\varnothing\nrightarrow}
\end{equation}
generates $\mD^{sk}\mM\bar\mV\bar\mE\GC_n$, up to a degree shift.

In all spaces of the general form $\mE^{\Sigma}\Gra_n^\bullet$ for a space of edge types $\Sigma$ and a space of core graphs $\Gra_n^\bullet$ the choice of the type of an edge does not depend on the choice of the types of other edges. It is not the case for $\mS^{sk}\MGC_n$ and $\mO^{sk}\MGC_n$ because here a certain edge type can be disallowed if it makes a loop or destroys the last source, the condition that depends on the choice of types of other edges.
However, it is possible to describe the distinguishable version $\mS^{sk}\mM\bar\mV\bar\mE\GC_n$ and $\mO^{sk}\mM\bar\mV\bar\mE\GC_n$ as subspaces of $\mD^{sk}\mM\bar\mV\bar\mE\GC_n$ generated by a certain subsets of $\mM\bar\mV\bar\mE^{\sigma_n^t}\grac^{\varnothing\nrightarrow}$ defined in the following definition.

Note that we need at lest some elements of the old basis $\sigma_n^\infty$ to be able to give the definition. This is the very reason of introducing $\sigma_n^i$.

\begin{defi}
\label{defi:SObase}
Let $t\geq 1$ or $t=\infty$. 
Let $\sigma\subseteq\sigma_n^t$ span a dg sub-module of $\Sigma_n^{sk}$ with $\Ed,\dE\in\sigma$.
A sequence of vertices $x_0,x_1,\dots, x_p=x_0$ is a \emph{type-cycle} in a base graph representative $(\Gamma,\alpha_1,\dots,\alpha_e)\in\mM\bar\mV\bar\mE\grac^{\varnothing\nrightarrow}\times\left(\sigma\right)^{\times e}$ if for every $i=0,\dots,p-1$ there exists an edge $a_i$ that is
\begin{itemize}
\item of type $\Ed$ and $\Gamma_-(a_i)=x_i$ and $\Gamma_+(a_i)=x_{i+1}$ or
\item of type $\dE$ and $\Gamma_-(a_i)=x_{i+1}$ and $\Gamma_+(a_i)=x_i$.
\end{itemize}
A base graph representative is \emph{type-oriented} if there is no type-cycle in it.
Let 
\begin{equation}
\mM\bar\mV\bar\mE\grac^{\varnothing\nrightarrow}\times_O\sigma^{\times e}\subset
\mM\bar\mV\bar\mE\grac^{\varnothing\nrightarrow}\times\sigma^{\times e}
\end{equation}
be the set of all base graph representatives that are type-oriented.

Let $t\geq 2$ or $t=\infty$.
Let $\sigma\subseteq\sigma_n^t$ span a dg sub-module of $\Sigma_n^{sk}$ with $\Ed,\dE,\dEd\in\sigma$.
A base graph representative $(\Gamma,\alpha_1,\dots,\alpha_e)\in\mM\bar\mV\bar\mE\grac^{\varnothing\nrightarrow}\times\left(\sigma\right)^{\times e}$ is \emph{type-sourced} if
\begin{itemize}
\item at least one $\alpha_i$ has length at least 3 or
\item at least one $\alpha_i$ is $\dEd$ or
\item there is at least one vertex $x$ such that there is no edge $a$ that is
\begin{itemize}
\item of type $\Ed$ and $\Gamma_+(a)=x$ or
\item of type $\dE$ and $\Gamma_-(a)=x$.
\end{itemize}
\end{itemize}
Let
\begin{equation}
\mM\bar\mV\bar\mE\grac^{\varnothing\nrightarrow}\times_S\sigma^{\times e}\subset
\mM\bar\mV\bar\mE\grac^{\varnothing\nrightarrow}\times\sigma^{\times e}
\end{equation}
be the set of all base graph representatives that are type-sourced.
\end{defi}

Clearly, reversing edges does not change the property of being type-oriented or type-sourced, so we can well defined type-oriented and type-sourced base graphs. Let
\begin{equation}
\mO\mM\bar\mV\bar\mE^{\sigma}\grac^{\varnothing\nrightarrow}\subset
\mM\bar\mV\bar\mE^{\sigma}\grac^{\varnothing\nrightarrow}
\end{equation}
be the set of all base graphs that are type-oriented and let
\begin{equation}
\mS\mM\bar\mV\bar\mE^{\sigma}\grac^{\varnothing\nrightarrow}\subset
\mM\bar\mV\bar\mE^{\sigma}\grac^{\varnothing\nrightarrow}
\end{equation}
be the set of all base graphs that are type-oriented.

\begin{prop}
\label{prop:SObase}
\begin{itemize}
\item[]
\item For $t\geq 1$ or $t=\infty$ it is $\mO^{sk}\mM\bar\mV\bar\mE_e\GC_n=
\left\langle\mM\bar\mV\bar\mE\grac^{\varnothing\nrightarrow}\times_O\left(\sigma_n^t\right)^{\times e}\right\rangle_{S_2^{\times e}}$
i.e.\ the set of base graphs $\mO\mM\bar\mV\bar\mE^{\sigma_n^t}\grac^{\varnothing\nrightarrow}$ generates $\mO^{sk}\mM\bar\mV\bar\mE\GC_n$.
\item For $t\geq 2$ or $t=\infty$ it is $\mS^{sk}\mM\bar\mV\bar\mE_e\GC_n=
\left\langle\mM\bar\mV\bar\mE\grac^{\varnothing\nrightarrow}\times_S\left(\sigma_n^t\right)^{\times e}\right\rangle_{S_2^{\times e}}$
i.e.\ the set of base graphs $\mS\mM\bar\mV\bar\mE^{\sigma_n^t}\grac^{\varnothing\nrightarrow}$ generates $\mS^{sk}\mM\bar\mV\bar\mE\GC_n$.
\end{itemize}
\end{prop}
\begin{proof}
Straightforward.
\end{proof}

\subsection{Sub-complexes with bounded skeleton edges}
\label{ss:scwbse}

The upper simplification of skeleton edges enables the following definition.

\begin{defi}
For $u\in\N$ let
\begin{equation}
\sigma_n^{0,u}:=\{\begin{tikzpicture}[baseline=-.65ex,scale=.5]
 \node[nil] (a) at (0,0) {};
 \node[nil] (c) at (2.2,0) {};
 \draw (a) edge[->] (c);
 \draw (.55,.15) edge (.55,-.15);
 \node at (1.3,.15) {$\dots$};
\end{tikzpicture}
\text{ of length }\leq u+1\}\cup
\{\begin{tikzpicture}[baseline=-.65ex,scale=.5]
 \node[nil] (a) at (0,0) {};
 \node[nil] (c) at (2.2,0) {};
 \draw (a) edge[dotted,->] (c);
 \draw (.55,.15) edge (.55,-.15);
 \node at (1.3,.15) {$\dots$};
\end{tikzpicture}
\text{ of length }\leq u\}
\subset\sigma^0
\end{equation}
and let $\Sigma_n^u=\langle\sigma_n^{0,u}\rangle\subset\Sigma^{sk}$.
Let
\begin{equation}
\mD^u\MGC_n:=
\mM\mE^{\Sigma_n^u}\Grac_n^{\varnothing\nrightarrow},
\end{equation}
for $u\geq 1$ let
\begin{equation}
\mO^u\mM\GC_n:=\mD^u\mM\GC_n\cap\mO^{sk}\mM\GC_n,
\end{equation}
and for $u\geq 2$ let
\begin{equation}
\mS^u\mM\GC_n:=\mD^u\mM\GC_n\cap\mS^{sk}\mM\GC_n.
\end{equation}
\end{defi}

It is clear that $(\Sigma_n^u,\delta_E)$ is a sub-complex of $(\Sigma_n^{sk},\delta_E)$, so all defined graph spaces are sub-complexes with the differential $\delta_C\pm\delta_E$.

We construct a basis of $\Sigma_n^u$ that has some elements of $\sigma_n^\infty$ and some element of $\sigma_n^{0,u}$. Particularly, for $t\leq u$
\begin{equation}
\sigma_n^{t,u}:=\{\alpha\in\sigma_n^\infty|\text{length of $\alpha$ is }\leq t\}\cup
\{\alpha\in\sigma_n^{0,u}|\text{length of $\alpha$ is }>t\}
\end{equation}
is a basis of $\Sigma_n^u$. The following is a variant of Proposition \ref{prop:SObase}.

\begin{prop}
\label{prop:SObase2}
\begin{itemize}
\item[]
\item For $t,u\in\N$, $1\leq t\leq u$ it is $\mO^u\mM\bar\mV\bar\mE_e\GC_n=
\left\langle\mM\bar\mV\bar\mE\grac^{\varnothing\nrightarrow}\times_O\left(\sigma_n^{t,u}\right)^{\times e}\right\rangle_{S_2^{\times e}}$
i.e.\ the set of base graphs $\mO\mM\bar\mV\bar\mE^{\sigma_n^{t,u}}\grac^{\varnothing\nrightarrow}$ generates $\mO^u\mM\bar\mV\bar\mE\GC_n$.
\item For $t,u\in\N$, $2\leq t\leq u$ it is $\mS^u\mM\bar\mV\bar\mE_e\GC_n=
\left\langle\mM\bar\mV\bar\mE\grac^{\varnothing\nrightarrow}\times_S\left(\sigma_n^{t,u}\right)^{\times e}\right\rangle_{S_2^{\times e}}$
i.e.\ the set of base graphs $\mS\mM\bar\mV\bar\mE^{\sigma_n^{t,u}}\grac^{\varnothing\nrightarrow}$ generates $\mS^u\mM\bar\mV\bar\mE\GC_n$.
\end{itemize}
\end{prop}

\begin{prop}
\label{prop:SmallSk}
\begin{enumerate}
\item[]
\item The inclusion $\mD^u\MGC_n\hookrightarrow\mD^{sk}\MGC_n$ is a quasi-isomorphism for all $u\geq 0$.
\item The inclusion $\mO^u\MGC_n\hookrightarrow\mO^{sk}\MGC_n$ is a quasi-isomorphism for all $u\geq 1$.
\item The inclusion $\mS^u\MGC_n\hookrightarrow\mS^{sk}\MGC_n$ is a quasi-isomorphism for all $u\geq 2$.
\end{enumerate}
\end{prop}
\begin{proof}
We prove all parts of the proposition simultaneously. It is enough to prove that the quotient is acyclic. It is the complex spanned by graphs with at least one solid skeleton edge longer than $u+1$ or dotted skeleton edge longer than $u$.

On the quotient complex we set up a spectral sequence on the number of vertices. Standard splitting of the complex as the product of complexes with fixed loop number implies the correct convergence. The first differential is $\delta_E$.

One can define a homotopy that sends dotted skeleton edges opposite of $\delta_E$ from Table \ref{tbl:difsknew}, summed over all dotted skeleton edges longer than $u$, with suitable sign. This implies the acyclicity.
\end{proof}

Note that degree of $\Es$ is $1-n$ and the parity of reversing it is $\mu=(-1)^n$, exactly as for the edge type in $\Sigma^\mu$ from $\MGC_n$. Therefore the new $\mD^0\MGC_n$ is exactly equal to the original $\MGC_n$. So the proposition implies the following corollary, that is the first part of Theorem \ref{thm:main}.

\begin{cor}
There is a quasi-isomorphism $g:\mO_j\mS_k\mD_\ell\GC_n\rightarrow\mO_j\mS_k\mD_{\ell+1}\GC_n$.
It sends an edge $\Es$ with extra coloured directions to $\frac{1}{2}\left(\Ed+(-1)^n\dE\right)$ where the arrow indicates the direction in the new colour while directions in other colours remain the same.
\end{cor}
\begin{proof}
The direct consequence of Propositions \ref{prop:DOSD} (c.f.\ \eqref{def:DOSDGC}), \ref{prop:SkEq} and \ref{prop:SmallSk} (1).
\end{proof}

\section{The construction of the quasi-isomorphism}
\label{s:qi}

In this section we prove the last two parts of Theorem \ref{thm:main}, that there are quasi-isomorphisms $h:\mO_j\mS_k\mD_\ell\GC_n\rightarrow\mO_{j+1}\mS_k\mD_\ell\GC_{n+1}\hookrightarrow\mO_j\mS_{k+1}\mD_\ell\GC_{n+1}$ for all $j,k,\ell\geq 0$. Because of Propositions \ref{prop:DOSD}, \ref{prop:SkEq} and \ref{prop:SmallSk} it is enough to give quasi-isomorphisms $h:\MGC_n\rightarrow\mO^1\MGC_{n+1}\hookrightarrow\mS^2\MGC_{n+1}$. We actually define $h$ on the distinguished versions, i.e.\ $h:\mM\bar\mV\bar\mE\GC_n\rightarrow\mO^1\mM\bar\mV\bar\mE\GC_{n+1}\hookrightarrow\mS^2\mM\bar\mV\bar\mE\GC_{n+1}$. It has to be checked that $h$ is compatible with permuting edges and vertices.
The final proof will use the spectral sequence on the number of vertices that leaves only the edge differential as the first differential. Since on the first page the number of edges and vertices are not changed, we can use the result for distinguishable edges and vertices.

\subsection{The map}
\label{ss:tm}

Recall that
$$
\mM\bar\mV_v\bar\mE_e\GC_n=
\left\langle\mM\bar\mV_v\bar\mE_e\grac_n^{\varnothing\nrightarrow}\right\rangle
\otimes_{S_2^{\times e}}\left(\Sigma^{(-1)^n}[1-n]\right)^{\otimes e}=
\left\langle\mM\bar\mV_v\bar\mE_e\grac_n^{\varnothing\nrightarrow}\right\rangle
\otimes_{S_2^{\times e}}\left\langle\sigma_n^{0,0}\right\rangle^{\otimes e},
$$
$$
\mO^1\mM\bar\mV_v\bar\mE_e\GC_{n+1}\subset
\mD^1\mM\bar\mV_v\bar\mE_e\GC_{n+1}=
\left\langle\mM\bar\mV_v\bar\mE_e\grac_{n+1}^{\varnothing\nrightarrow}\right\rangle
\otimes_{S_2^{\times e}}
\left\langle\sigma_{n+1}^{1,1}\right\rangle^{\otimes e},
$$
$$
\mS^2\mM\bar\mV_v\bar\mE_e\GC_{n+1}\subset
\mD^2\mM\bar\mV_v\bar\mE_e\GC_{n+1}=
\left\langle\mM\bar\mV_v\bar\mE_e\grac_{n+1}^{\varnothing\nrightarrow}\right\rangle
\otimes_{S_2^{\times e}}
\left\langle\sigma_{n+1}^{2,2}\right\rangle^{\otimes e},
$$
where
$$
\sigma_n^{0,0}=\{\Es\},\quad
\sigma_{n+1}^{1,1}=\{\Ed,\dE,\Ess\},\quad
\sigma_{n+1}^{2,2}=\{\Ed,\dE,\dEd,\EdE,
\Esss\}.
$$

Inclusion $\mO^1\mM\bar\mV_v\bar\mE_e\GC_{n+1}\hookrightarrow\mS^2\mM\bar\mV_v\bar\mE_e\GC_{n+1}$ is clear. We are going to construct the map $h:\mM\bar\mV_v\bar\mE_e\GC_n\rightarrow\mO^1\mM\bar\mV_v\bar\mE_e\GC_{n+1}$ on the base graphs through its representatives. One needs to check that it is well defined, but all relations on base graph representatives are \eqref{eq:BaseRel} that come from reverting edges, and they will be easy to check. Note that spaces of core graphs are the same in $\mM\bar\mV_v\bar\mE_e\GC_n$ and $\mO^1\mM\bar\mV_v\bar\mE_e\GC_{n+1}$, and the map will, as expected, not change the core graph.

Let $(\Gamma,\dots)\in\mM\bar\mV_v\bar\mE_e\grac_{n+1}^{\varnothing\nrightarrow}\times\langle\Es\rangle^{\times e}$ be a representative of a base graph in $\mM\bar\mV_v\bar\mE_e\grac_n^{\varnothing\nrightarrow}$. Since the space of edge types is 1-dimensional, it is equivalent to choosing a core graph $\Gamma$. We call a \emph{spanning tree} of $\Gamma$ a connected sub-graph without cycles which contains all its vertices. Note that all spanning trees have $v-1$ edges. Let $S(\Gamma)$ be the set of all spanning trees of $\Gamma$.

For a chosen vertex $x\in V(\Gamma)$ and a spanning tree $\tau\in S(\Gamma)$ we define $h'_{x,\tau}(\Gamma)\in\mO^1\mM\bar\mV_v\bar\mE_e\GC_{n+1}$ as follows.
The core graph is again $\Gamma$ with relabelled vertices and edges as being discussed later. Coloured directions (maps $o_c$) remain the same.
Edges that are in $E(\tau)$ get the type $\Ed$ or $\dE$, such that the arrow goes in the direction away from the vertex $x$. Edges that are not in $E(\tau)$ get the type $\Ess$, in the core direction. The result gets a pre-factor 
\begin{equation}
\label{eq:sgnrinv}
(-1)^{rn}
\end{equation}
where $r$ is the number of edges in $E(\tau)$ that get a type $\dE$, i.e.\ the number of edges whose type direction (the one going away from $x$) is the opposite to the core orientation. An example of $h'_{x,\tau}$ is given in Figure \ref{fig:hxtau}.

\begin{figure}[H]
$$
\begin{tikzpicture}[baseline=1ex,scale=1.5]
 \node[int] (a) at (-1,-.5) {};
 \node[int] (b) at (1,-.5) {};
 \node[int] (c) at (0,0) {};
 \node[int] (d) at (0,1.1) {};
 \node[below] at (a) {};
 \node[below] at (b) {};
 \node[below] at (c) {};
 \node[above] at (d) {$x$};
 \draw (a) edge[->] (b);
 \draw (a) edge[<-] (c);
 \draw (a) edge[->] (d);
 \draw (b) edge[->] (c);
 \draw (b) edge[<-] (d);
 \draw (c) edge[->] (d);
\end{tikzpicture}\quad
\mxto{h'_{x,
\begin{tikzpicture}[baseline=0ex,scale=.1]
 \draw (-1,-.5) edge (0,0);
 \draw (1,-.5) edge (0,0);
 \draw (0,0) edge (0,1.1);
\end{tikzpicture}}}
\quad (-1)^{2n}
\begin{tikzpicture}[baseline=1ex,scale=1.5]
 \node[int] (a) at (-1,-.5) {};
 \node[int] (b) at (1,-.5) {};
 \node[int] (c) at (0,0) {};
 \node[int] (d) at (0,1.1) {};
 \node[below] at (a) {};
 \node[below] at (b) {};
 \node[below] at (c) {};
 \node[above] at (d) {$x$};
 \draw (a) edge[latex-] (c);
 \draw (b) edge[latex-] (c);
 \draw (c) edge[latex-] (d);
 \draw (a) edge[->,crossed] (b);
 \draw (a) edge[->,crossed] (d);
 \draw (b) edge[<-,crossed] (d);
\end{tikzpicture}
$$
\caption[]{\label{fig:hxtau}
An example of the map $h'_{x,\tau}$ on a tetrahedron graph, with the chosen vertex $x$ the top one, and with the chosen tree the middle star. Edges in the star get the types $\Ed$ in the direction away from $x$, and other edges get the type $\Ess$. Pre-factor is $(-1)^{2n}$ since 2 edges from the tree take the direction opposite of the core one.}
\end{figure}
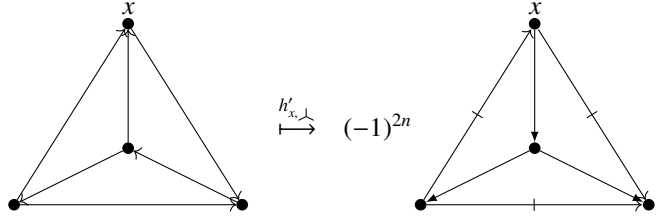

Reversing an edge in $E(\tau)$ changes the sign by $(-1)^n$ and changes the number $r$ by one, so reversed base graph is sent to the same graph. Reversing edges not in $E(\tau)$ also changes the sign by $(-1)^n$, and they are sent to $\Ess$ whose reversing changes the sign by $-(-1)^{n+1}=(-1)^n$. So the map is well defined on base graphs.

Since edge types $\Ed$ and $\dE$ that could form a type cycle are given to a tree, there can not be a type cycle in $h'_{x,\tau}(\Gamma)$, so it is indeed in $\mO^1\mM\bar\mV_v\bar\mE_e\GC_{n+1}$.

We still need to label edges and vertices. We want to do it such that the final map is compatible with permuting edges and vertices. Since the graph complex parameter has changed from $n$ to $n+1$, parities are different. Indeed vertices in the domain complex and edges of type $\Ed$ in the codomain complex have the parity $(-1)^n$, while edges in the domain complex and vertices and edges of type $\Ess$ in the codomain complex have the parity $-(-1)^n$. The map will commute with permuting if elements get labels from elements of the same parity. Therefore, vertices in $h'_{x,\tau}(\Gamma)$ will get labels from edges in $\Gamma$, particularly edges in $\tau$, edges of type $\Ess$ in $h'_{x,\tau}(\Gamma)$ will get labels from other edges in $\Gamma$, and edges of type $\Ed$ and $\dE$ in $h'_{x,\tau}(\Gamma)$ will get labels from vertices in $\Gamma$. An edge in $\tau$ corresponds to the vertex where its given type direction heads to, that is the vertex at its end away from the chosen vertex $x$.

The model for labelling is the graph $\Gamma$ where the chosen vertex is the last one with the label $v$, and all edges in $E(\tau)$ come before other edges. This labelling we call \emph{model labelling}. An example is given in Figure \ref{fig:model}.

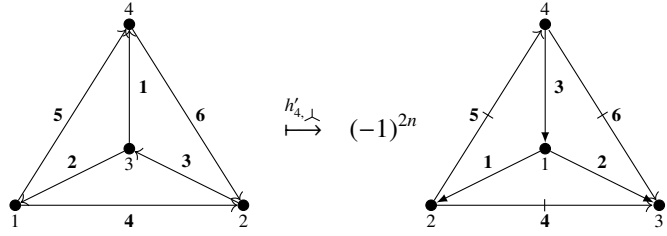
\begin{figure}[H]
$$
\begin{tikzpicture}[baseline=1ex,scale=1.5]
 \node[int] (a) at (-1,-.5) {};
 \node[int] (b) at (1,-.5) {};
 \node[int] (c) at (0,0) {};
 \node[int] (d) at (0,1.1) {};
 \node[below] at (a) {$\scriptstyle 1$};
 \node[below] at (b) {$\scriptstyle 2$};
 \node[below] at (c) {$\scriptstyle 3$};
 \node[above] at (d) {$\scriptstyle 4$};
 \draw (a) edge[->] node[below] {$\scriptstyle {\mathbf 4}$} (b);
 \draw (a) edge[<-] node[above] {$\scriptstyle {\mathbf 2}$} (c);
 \draw (a) edge[->] node[left] {$\scriptstyle {\mathbf 5}$} (d);
 \draw (b) edge[->] node[above] {$\scriptstyle {\mathbf 3}$} (c);
 \draw (b) edge[<-] node[right] {$\scriptstyle {\mathbf 6}$} (d);
 \draw (c) edge[->] node[right] {$\scriptstyle {\mathbf 1}$} (d);
\end{tikzpicture}\quad
\mxto{h'_{4,
\begin{tikzpicture}[baseline=0ex,scale=.1]
 \draw (-1,-.5) edge (0,0);
 \draw (1,-.5) edge (0,0);
 \draw (0,0) edge (0,1.1);
\end{tikzpicture}}}
\quad (-1)^{2n}
\begin{tikzpicture}[baseline=1ex,scale=1.5]
 \node[int] (a) at (-1,-.5) {};
 \node[int] (b) at (1,-.5) {};
 \node[int] (c) at (0,0) {};
 \node[int] (d) at (0,1.1) {};
 \node[below] at (a) {$\scriptstyle 2$};
 \node[below] at (b) {$\scriptstyle 3$};
 \node[below] at (c) {$\scriptstyle 1$};
 \node[above] at (d) {$\scriptstyle 4$};
 \draw (a) edge[latex-] node[above] {$\scriptstyle {\mathbf 1}$} (c);
 \draw (b) edge[latex-] node[above] {$\scriptstyle {\mathbf 2}$} (c);
 \draw (c) edge[latex-] node[right] {$\scriptstyle {\mathbf 3}$} (d);
 \draw (a) edge[->,crossed] node[below] {$\scriptstyle {\mathbf 4}$} (b);
 \draw (a) edge[->,crossed] node[left] {$\scriptstyle {\mathbf 5}$} (d);
 \draw (b) edge[<-,crossed] node[right] {$\scriptstyle {\mathbf 6}$} (d);
\end{tikzpicture}
$$
\caption[]{\label{fig:model}
An example of model labelling, i.a.\ a core graph $\Gamma$ with $v$ vertices and $e$ edges, the chosen vertex $v$ and a chosen spanning tree $\tau$ that consists of edges $\{1,\dots,v-1\}$. The image $h'_{v,\tau}(\Gamma)$ with labels is written on the right. Label of the vertex $v$ remains the same; label of edges $\{v,\dots,e\}$ that are not in $\tau$ remain the same; and labels of edges $\{1,\dots,v-1\}$ that are in $\tau$ and vertices $\{1,\dots,v-1\}$ are exchanged such that an edge corresponds to its end away from the chosen vertex $v$.
}
\end{figure}

Let $\chi\in S_v\times S_e$ permute vertices and edges such that the last vertex $v$ is fixed and sets $\{1,\dots,v-1\}$ and $\{v,\dots,e\}$ of edges are fixed. The definition is set such that $h'_{v,\tau}(\chi\Gamma)$ is sent to $\chi'h'_{v,\tau}(\Gamma)$ for some $\chi'\in S_v\times S_e$.

For an arbitrary labelled graph we first relabel it to the model labelling by permutation $\chi\in S_v\times S_e$. Note that the action of $\chi$ changes the sign as in Equations \eqref{def:SeSigma} and \eqref{eq:actSv}. The precise calculation of the sign is not essential and we leave it to the reader. We can make a convention not to change the order of vertices and edges if not necessarily. Then we apply the map $h'_{\chi(x)=v,\chi(\tau)}$. By this very construction, for every $\chi\in S_v\times S_e$, $h'_{v,\tau}(\chi\Gamma)$ is sent to $\chi'h'_{v,\tau}(\Gamma)$ for some $\chi'\in S_v\times S_e$. Therefore the quotient of $h$ to the space of graphs with permuted edges and vertices is well defined.

\begin{rem}
It is possible to label vertices and edges with the same set of $e+v$ labels without any constraints which labels are used for vertices and which for edges. Permutations of that labels give signs as in Equation \eqref{def:SeSigma}, taking into account parities of particular elements (vertices and edges of particular type).
With this convention, switching labels in the map $h$ becomes trivial.
\end{rem}

Strictly speaking, the core graphs of $\Gamma$ and $h_{x,\tau}(\Gamma)$ are not the same if the labellings are not the same. 
We indeed want them to be the same, so let
\begin{equation}
\label{eq:sgnRelab}
\hat h_{x,\tau}(\Gamma):=\chi h'_{x,\tau}(\Gamma),
\end{equation}
where $\chi\in S_v\times S_e$ relabels vertices and edges back to the labelling of $\Gamma$ and gives appropriate sign.

Now let
\begin{equation}
\label{def:h}
h_{x,\tau}(\Gamma)=(-1)^v(-1)^{(n+1)ev}\hat h_{x,\tau}(\Gamma),
\end{equation}
\begin{equation}
h(\Gamma):=\sum_{x\in V(\Gamma)}(v(x)-2)\sum_{\tau\in S(\Gamma)}h_{x,\tau}(\Gamma),
\end{equation}
where $v(x)$ is the valence of the vertex $x$ in $\Gamma$. The upper discussion implies that the quotient map
\begin{equation}
h:\mM\mV_v\mE_e\GC_n\rightarrow\mO^1\mM\mV_v\mE_e\GC_{n+1}
\end{equation}
is well defined. It is extended to $h:\mM\GC_n\rightarrow\mO^1\mM\GC_{n+1}$.
We have the following proposition.

\begin{prop}
The map $h:\mM\GC_n\rightarrow\mO^1\mM\GC_{n+1}$ is a map of complexes of degree 0, i.e.\ $\delta h(\Gamma)=h(\delta\Gamma)$ for every $\Gamma\in\mM\GC_n$.
\end{prop}
\begin{proof}
One easily calculates that degree of $h$ is 0.
For the other claim let $\Gamma\in\mM\mV_v\mE_e\GC_n$.
By abuse of notation, we will also denote by $\Gamma$ any suitable core graph that represents $\Gamma$.
It holds that
$$
h(\delta\Gamma)=
h\left(\sum_{a\in E(\Gamma)}c_a(\Gamma)\right)=
\sum_{a\in E(\Gamma)}h\left(c_a(\Gamma)\right)=
\sum_{a\in E(\Gamma)}\sum_{x\in V(c_a(\Gamma))}
\left(v_{c_a(\Gamma)}(x)-2\right)\sum_{\tau\in S(c_a(\Gamma))}h_{x,\tau}(c_a(\Gamma))
$$
where $c_a(\Gamma)$ is contracting an edge $a$ in $\Gamma$.
Spanning trees of $c_a(\Gamma)$ are in natural bijection with spanning trees of $\Gamma$ that contain $a$, $c_a(\tau)\leftrightarrow\tau$, so we can write
$$
h(\delta\Gamma)=
\sum_{a\in E(\Gamma)}
\sum_{\substack{\tau\in S(\Gamma)\\ a\in E(\tau)}}\sum_{x\in V(c_a(\Gamma))}
\left(v_{c_a(\Gamma)}(x)-2\right)h_{x,c_a(\tau)}(c_a(\Gamma))=\\
\sum_{\tau\in S(\Gamma)}\sum_{a\in E(\tau)}\sum_{x\in V(c_a(\Gamma))}
\left(v_{c_a(\Gamma)}(x)-2\right)h_{x,c_a(\tau)}(c_a(\Gamma)).
$$

\begin{lemma}
Let $\Gamma\in\mM\mV_v\mE_e\GC_n$, $\tau\in S(\Gamma)$ and $a\in E(\tau)$. Then
\begin{enumerate}
\item if $x\in V(\Gamma)$, $x\neq\Gamma_\pm(a)$ then
\begin{equation}
h_{x,c_a(\tau)}(c_a(\Gamma))=c_a(h_{x,\tau}(\Gamma)),
\end{equation}
\item if $x$ is produced by contracting $a$ then
\begin{equation}
h_{x,c_a(\tau)}(c_a(\Gamma))=c_a\left(h_{\Gamma_-(a),\tau}(\Gamma)\right)=c_a\left(h_{\Gamma_+(a),\tau}(\Gamma)\right)
\end{equation}
where $\Gamma_-(a)$ and $\Gamma_+(a)$ are ends of the edge $a$.
\end{enumerate}
\end{lemma}
\begin{proof}
Ignoring signs the equalities are clear. One only needs to check the sign. For contracting, according to Remark \ref{rem:last}, we use the representative where the last edge heading towards the last vertex is being contracted, and for $h_{x,\tau}$ we use model labelling. 
\begin{enumerate}
\item Chose a representative $\Gamma$ with model labelling, i.e.\ $x=v$ and $E(\tau)$ is labelled by $\{1,\dots,v-1\}$. Moreover, let all edges in $E(\tau)$ head away from $v$ and let $a=v-1$ and $\Gamma_+(a)=v-1$.

The map $h_{v-1,\tau}$ gives a sign $(-1)^v(-1)^{(n+1)ev}$ and switches labels between $E(\tau)$ and $V(\Gamma)$, thus $a$ is still labelled by $v-1$ and heads towards $v-1$. For contracting $a$ we switch edges $v-1$ and $e$, giving sign $(-1)^{(2-(n+1))(e-v)}=(-1)^{(n+1)(e-v)}$ because an edge of type $\Ess$ that has degree $(2-(n+1)$ jumps over $e-v$ other edges of the same type, and we switch vertices $v$ and $v-1$, giving sign $(-1)^{n+1}$. The total sign is $(-1)^v(-1)^{(n+1)(e+1)(v+1)}$. Then we contract $e$.

On the other side, we first switch labels between edges $v-1$ and $e$, and between vertices $v-1$ and $v$, giving sign $(-1)^n(-1)^{1-n}=-1$. Then we can contract $e$. Labelling of the resulting graph is model labelling, so we can act by $h$ giving the sign $(-1)^{v-1}(-1)^{(n+1)(e-1)(v-1)}$. The resulting graph has the same labelling as the result of the other side, with the total sign again $(-1)^v(-1)^{(n+1)(e+1)(v+1)}$.

\item A similar argument is left to the reader.
\end{enumerate}
\end{proof}

For $x\in V(c_a(\Gamma))$ we have two choices, either $x\in V(\Gamma)$ not adjacent to $a$, or $x$ is produced by contraction of $a$. Both cases are covered with the lemma.
Since in the first case the chosen vertex does not change valence after applying $c_a$, and it holds that $v_{c_a(\Gamma)}(x)=v_\Gamma(\Gamma_-(a))+v_\Gamma(\Gamma_+(a))-2$ in the second case, it follows that
\begin{multline*}
h(\delta\Gamma)=\sum_{\tau\in S(\Gamma)}\sum_{a\in E(\tau)}
 \left(\left(v_\Gamma(\Gamma_-(a))+v_\Gamma(\Gamma_+(a))-4\right)c_a(h_{\Gamma_-(a),\tau}(\Gamma))+\sum_{\substack{x\in V(\Gamma)\\x\notin\{\Gamma_-(a),\Gamma_+(a)\}}}(v_{\Gamma}(x)-2)c_a(h_{x,\tau}(\Gamma))\right)=\\
 =\sum_{\tau\in S(\Gamma)}\sum_{a\in E(\tau)}\sum_{x\in V(\Gamma)}(v_{\Gamma}(x)-2)c_a(h_{x,\tau}(\Gamma))=
 \sum_{x\in V(\Gamma)}(v_{\Gamma}(x)-2)\sum_{\tau\in S(\Gamma)}\sum_{a\in E(\tau)}c_a(h_{x,\tau}(\Gamma)).
\end{multline*}

On the other side the differential is $\delta=\delta_C\pm\delta_E$. The sign before $\delta_E$ is not important as will be clear soon. In $h(\Gamma)$ we can contract only edges in $E(\tau)$ so
\begin{multline*}
\delta_C h(\Gamma)=
\delta_C\left(\sum_{x\in V(\Gamma)}(v_\Gamma(x)-2)\sum_{\tau\in S(\Gamma)}h_{x,\tau}(\Gamma)\right)=
\sum_{x\in V(\Gamma)}(v_\Gamma(x)-2)\sum_{\tau\in S(\Gamma)}\delta_C\left(h_{x,\tau}(\Gamma)\right)=\\
\sum_{x\in V(\Gamma)}(v_\Gamma(x)-2)\sum_{\tau\in S(\Gamma)}\sum_{a\in E(\tau)}c_a\left(h_{x,\tau}(\Gamma)\right)=
h(\delta(\Gamma)).
\end{multline*}
Therefore, to finish the proof it is enough to show that $\delta_E(h(\Gamma))=0$. We have
$$
\delta_E(h(\Gamma))=
\delta_E\left(\sum_{x\in V(\Gamma)}(v_\Gamma(x)-2)\sum_{\tau\in S(\Gamma)}h_{x,\tau}(\Gamma)\right)=
\sum_{x\in V(\Gamma)}(v_\Gamma(x)-2)\sum_{\tau\in S(\Gamma)}\delta_E\left(h_{x,\tau}(\Gamma)\right)=0,
$$
by the following lemma.
\end{proof}

\begin{lemma}
Let $\Gamma\in\mM\mV_v\mE_e\GC_n$ and $x\in V(\Gamma)$. Then
\begin{equation}
\sum_{\tau\in S(\Gamma)}\delta_E\left(h_{x,\tau}(\Gamma)\right)=0.
\end{equation}
\end{lemma}
\begin{proof}
Edge differential acts on edges of type $\Ess$, that are those labelled by $\{v,\dots,e\}$ in $h_{x,\tau}(\Gamma)$. So
$$
N(\Gamma,x):=
\sum_{\tau\in S(\Gamma)}\delta_E\left(h_{x,\tau}(\Gamma)\right)=
\sum_{\tau\in S(\Gamma)}\sum_{a=v}^e\delta_E^{(a)}\left(h_{x,\tau}(\Gamma)\right),
$$
where $\delta_E^{(a)}$ maps $a$-th edge $\Ess\mapsto
\begin{tikzpicture}[baseline=-.65ex,scale=.5]
 \node[nil] (a) at (0,0) {};
 \node[nil] (c) at (1,0) {};
 \draw (a) edge[dotted,->] (c);
\end{tikzpicture}
=\Ed-(-1)^{n+1}\dE$ with a proper sign from \eqref{def:deltaEsgn}.

Terms in the above relation can be summed in another order. Let $CT(\Gamma)$ be the set of all connected sub-graphs $\rho$ of $\Gamma$ which contain all vertices and which has one more edge than a spanning tree (has one cycle), and let $C(\rho)$ be the set of edges in the cycle of $\rho$. Clearly, $\rho\setminus a$ for $a\in C(\rho)$ is a spanning tree of $\Gamma$ and sets $\{(\tau,a)|\tau\in S(\Gamma),a\in E(\Gamma)\\ E(\tau)\}$ and $\{(\rho,a)|\chi\in CT(\Gamma),a\in C(\rho)\}$ are bijective, so
$$
N(\Gamma,x)=\sum_{\rho\in CT(\Gamma)}\sum_{a\in C(\rho)}\delta_E^{(a)}\left(h_{x,\tau}(\Gamma)\right).
$$

It is now enough to show that
$$
\sum_{a\in C(\rho)}\delta_E^{(a)}\left(h_{x,\tau}(\Gamma)\right)=0
$$
for every $x\in V(\Gamma)$ and for every $\rho\in CT(\Gamma)$.
Let $y\in V(\Gamma)$ be the vertex in the cycle of $\rho$ closest to vertex $x$ (along $\rho$). After choosing $a\in C(\rho)$, that cycle in $h_{x,\chi\setminus a}(\Gamma)$ has one edge of type $\Ess$, and other edges of type $\Ed$ or $\dE$ with direction from $y$ to that the edge of type $\Ess$, such as in the following diagram.
$$
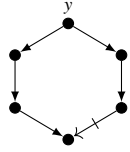
\begin{tikzpicture}[baseline=0ex,scale=.7]
 \node[int] (a) at (0,1.1) {};
 \node[int] (b) at (1,.5) {};
 \node[int] (c) at (1,-.5) {};
 \node[int] (d) at (0,-1.1) {};
 \node[int] (e) at (-1,-.5) {};
 \node[int] (f) at (-1,.5) {};
 \node[above] at (a) {$\scriptstyle y$};
 \draw (a) edge[-latex] (b);
 \draw (b) edge[-latex] (c);
 \draw (a) edge[-latex] (f);
 \draw (f) edge[-latex] (e);
 \draw (e) edge[-latex] (d);
 \draw (c) edge[->,crossed] (d);
\end{tikzpicture}
$$
After acting by $\delta_E^{(a)}$ the $\Ess$ is replaced by $\Ed+(-1)^n\dE$, like in the following diagram.
$$
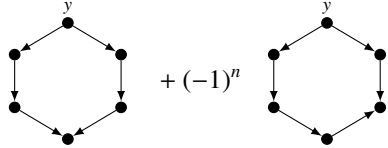
\begin{tikzpicture}[baseline=-.6ex,scale=.7]
 \node[int] (a) at (0,1.1) {};
 \node[int] (b) at (1,.5) {};
 \node[int] (c) at (1,-.5) {};
 \node[int] (d) at (0,-1.1) {};
 \node[int] (e) at (-1,-.5) {};
 \node[int] (f) at (-1,.5) {};
 \node[above] at (a) {$\scriptstyle y$};
 \draw (a) edge[-latex] (b);
 \draw (b) edge[-latex] (c);
 \draw (a) edge[-latex] (f);
 \draw (f) edge[-latex] (e);
 \draw (e) edge[-latex] (d);
 \draw (c) edge[-latex] (d);
\end{tikzpicture}
\quad+(-1)^n\quad
\begin{tikzpicture}[baseline=-.6ex,scale=.7]
 \node[int] (a) at (0,1.1) {};
 \node[int] (b) at (1,.5) {};
 \node[int] (c) at (1,-.5) {};
 \node[int] (d) at (0,-1.1) {};
 \node[int] (e) at (-1,-.5) {};
 \node[int] (f) at (-1,.5) {};
 \node[above] at (a) {$\scriptstyle y$};
 \draw (a) edge[-latex] (b);
 \draw (b) edge[-latex] (c);
 \draw (a) edge[-latex] (f);
 \draw (f) edge[-latex] (e);
 \draw (e) edge[-latex] (d);
 \draw (c) edge[latex-] (d);
\end{tikzpicture}
$$
Careful calculation of the sign shows that those two terms are cancelled with terms given from choosing neighbouring edges in $C(\rho)$, and two last terms which does not have corresponding neighbour are indeed $0$ because they have a type cycle.
This concludes the proof that $N(\Gamma,x)=0$.
\end{proof}

\subsection{The proof}
\label{ss:tp}

\begin{prop}
\label{prop:main}
The maps $h:(\MGC_n,\delta)\rightarrow\left(\mO^1\MGC_{n+1},\delta\right)\hookrightarrow\left(\mS^2\MGC_{n+1},\delta\right)$ are quasi-isomorphisms.
\end{prop}
\begin{proof}
We prove simultaneously that both maps are quasi-isomorphisms. On the mapping cone of them we set up the spectral sequence on the number of vertices. Standard splitting of complexes as the product of complexes with fixed loop number implies the correct convergence.

The edge differential does not change the number of vertices, while the core differential does. Therefore, on the first page of the spectral sequences there are mapping cones of the maps
$$
h:(\MGC_n,0)
\rightarrow
\left(\mO^1\MGC_{n+1},\pm\delta_E\right)
\hookrightarrow
\left(\mS^2\MGC_{n+1},\pm\delta_E\right).
$$

Since the edge differential does not change the number of vertices and edges, the homology commutes with permuting them. To show that this mapping cone is acyclic, it is enough to show the same for distinguishable vertices and edges. Also, complexes are direct sums of those with fixed number of vertices and edges, so it is enough to show that
$$
h:(\mM\bar\mV_v\bar\mE_e\GC_n,0)
\rightarrow
\left(\mO^1\mM\bar\mV_v\bar\mE_e\GC_{n+1},\pm\delta_E\right)
\hookrightarrow
\left(\mS^2\mM\bar\mV_v\bar\mE_e\GC_{n+1},\pm\delta_E\right)
$$
are quasi-isomorphisms for all $v$ and $e$.

Recall that all involved complexes are sub-complexes of
$\mM\bar\mV_v\bar\mE_e^\Sigma\Grac_n^{\varnothing\rightarrow}=
\mM\bar\mV_v\bar\mE_e\Grac_n^{\varnothing\nrightarrow}
\otimes_{S_2^{\times e}}\Sigma^{\otimes e}$
for different space of edge types $\Sigma$. Using Proposition \ref{prop:SObase2} we can write them as
$$
\mM\bar\mV_v\bar\mE_e\GC_n=
\left\langle\mM\bar\mV\bar\mE\grac^{\varnothing\nrightarrow}\times\left(\sigma_n^{0,0}\right)^{\times e}\right\rangle_{S_2^{\times e}},
$$
$$
\mO^1\mM\bar\mV_v\bar\mE_e\GC_{n+1}=
\left\langle\mM\bar\mV\bar\mE\grac^{\varnothing\nrightarrow}\times_O\left(\sigma_n^{1,1}\right)^{\times e}\right\rangle_{S_2^{\times e}},
$$
$$
\mS^2\mM\bar\mV_v\bar\mE_e\GC_{n+1}=
\left\langle\mM\bar\mV\bar\mE\grac^{\varnothing\nrightarrow}\times_S\left(\sigma_n^{2,2}\right)^{\times e}\right\rangle_{S_2^{\times e}}.
$$

The homology of the edge differential commutes with the action of $S_2^{\times e}$ so it is enough to show that
$$
h:\left(\left\langle\mM\bar\mV\bar\mE\grac^{\varnothing\nrightarrow}\times\left(\sigma_n^{0,0}\right)^{\times e}\right\rangle,0\right)
\rightarrow
\left(\left\langle\mM\bar\mV\bar\mE\grac^{\varnothing\nrightarrow}\times_O\left(\sigma_{n+1}^{1,1}\right)^{\times e}\right\rangle,\pm\delta_E\right)
\hookrightarrow
\left(\left\langle\mM\bar\mV\bar\mE\grac^{\varnothing\nrightarrow}\times_S\left(\sigma_{n+1}^{2,2}\right)^{\times e}\right\rangle,\pm\delta_E\right)
$$
are quasi-isomorphisms.

Neither differential $\delta_E$ nor map $h$ changes core graph, so the mapping cones of above maps split as a direct sum of mapping cones for fixed core graph. Let us fix a core graph $\Phi\in\mM\bar\mV\bar\mE\grac^{\varnothing\nrightarrow}$ with $v$ vertices and $e$ edges. Recall, for no reason, that $\Phi$ is $j$-oriented $k$-sourced $\ell$-directed core graph with chosen directions in all colours.

It is now enough to show that
\begin{equation}
\label{eq:hreduced}
h:\left(\left\langle\{\Phi\}\times\left(\sigma_n^{0,0}\right)^{\times e}\right\rangle,0\right)
\rightarrow
\left(\left\langle\{\Phi\}\times_O\left(\sigma_{n+1}^{1,1}\right)^{\times e}\right\rangle,\pm\delta_E\right)
\hookrightarrow
\left(\left\langle\{\Phi\}\times_S\left(\sigma_{n+1}^{2,2}\right)^{\times e}\right\rangle,\pm\delta_E\right)
\end{equation}
are quasi-isomorphisms. Here $\{\Phi\}\times\sigma$, $\{\Phi\}\times_O\sigma$ and $\{\Phi\}\times_S\sigma$ are subsets of $\bar\mV\bar\mE\grac^{\varnothing\nrightarrow}\times\sigma$, respectively $\bar\mV\bar\mE\grac^{\varnothing\nrightarrow}\times_O\sigma$, respectively $\bar\mV\bar\mE\grac^{\varnothing\nrightarrow}\times_S\sigma$, with the first term $\Phi$. Let us abbreviate:
\begin{equation}
\{\Phi\}\cong\{\Phi\}\times\left(\sigma_n^{0,0}\right)^{\times e},\qquad
\mO\Phi:=\left\langle\{\Phi\}\times_O\left(\sigma_{n+1}^{1,1}\right)^{\times e}\right\rangle,\qquad
\mS\Phi:=\left\langle\{\Phi\}\times_S\left(\sigma_{n+1}^{2,2}\right)^{\times e}\right\rangle.
\end{equation}

To prove \eqref{eq:hreduced} we introduce another complexes and maps as in Figure \ref{fig:cd}, and show that all mentioned maps are quasi-isomorphisms.

\begin{figure}[H]
$$
\begin{tikzcd}
\left\langle\{\Phi\}\right\rangle
\arrow{r}{h}
\arrow[swap]{dr}{f\circ h}
& \mO\Phi
\arrow[hookrightarrow]{r}{\iota_0}
\arrow{d}{f}
& \mS\Phi
\arrow{d}{g} \\
& \mO\Phi^{v-1}
\arrow[hookrightarrow]{r}{\iota_{v-1}}
\arrow[swap]{dr}{p\circ\iota_{v-1}}
& \mS\Phi^{v-1}
\arrow{d}{p} \\
& & A^{e-v+1}
\end{tikzcd}
$$
\caption{\label{fig:cd}
A commutative diagram of complexes. In what follows we introduce complexes $\mO\Phi^i$, $\mS\Phi^i$ and $A^j$, and prove that all vertical and diagonal maps in the diagram are quasi isomorphism, implying that all horizontal maps are quasi-isomorphisms too, finishing the proof.}
\end{figure}

In $\Phi$ we choose $v-1$ edges, say $a_1,\dots,a_{v-1}$, such that for every $i=1,\dots,v-1$ the edges $\{a_1,\dots, a_i\}$ form a sub-graph of $\Phi$ that is a tree and contains the last vertex $v$. Clearly, $\{a_1,\dots, a_{v-1}\}$ forms a spanning tree.
For every $i=0,\dots,v-1$ we form two graph complexes $\mO\Phi^i$ and $\mS\Phi^i$ as follows.
\begin{itemize}
\item Let $\sigma_O=\sigma_{n+1}^{1,1}\cup\{\ET\}$. The structure of $\sigma_{n+1}^{1,1}$ is as usual, $\ET$ is of degree $-n$ (the same as $\Ed$ and $\dE$), $\delta_E(\ET)=0$ and non-trivial $\chi\in S_2$ acts as $\chi\ET=(-1)^{n+1}\ET$.
Edges in $\Phi$ have types from $\sigma_O$ such that $a_1,\dots, a_i$ have type $\ET$, and other edges have types in $\sigma_n^{1,1}$.

A sequence of vertices $x_0,x_1,\dots, x_p=x_0$ in $\Phi$ is a \emph{cycle} in $(\Phi,\alpha_1,\dots,\alpha_e)$ if for every $j=0,\dots,p-1$ there exists an edge $a$ that is
\begin{itemize}
\item of type $\Ed$ or $\ET$ and $\Phi_-(a)=x_j$ and $\Phi_+(a)=x_{j+1}$ or
\item of type $\dE$ or $\ET$ and $\Phi_-(a)=x_{j+1}$ and $\Phi_+(a^)=x_j$.
\end{itemize}
Roughly speaking, thick edges $\ET$ are considered to go in both directions for the matter of cycles.
$(\Phi,\alpha_1,\dots,\alpha_e)$ is \emph{oriented} if there is no cycle in it. $(\mO\Phi^i,\delta_E)$ is the complex spanned by all oriented $(\Phi,\alpha_1,\dots,\alpha_e)$.
An example of an element of $(\mO\Phi^i,\delta_E)$ is shown in Figure \ref{fig:complexSigma}.

\item Let $\sigma_S=\sigma_{n+1}^{2,2}\cup\{\ET\}$.
The structure of $\sigma_{n+1}^{2,2}$ is as usual, and $\ET$ is as above.
Edges in $\Phi$ have types from $\sigma_S$ such that $a_1,\dots, a_i$ have type $\ET$, and other edges have types in $\sigma_n^{2,2}$.

$(\Phi,\alpha_1,\dots,\alpha_e)$ is \emph{sourced} if
\begin{itemize}
\item at least one $\alpha_i$ is of type $\Esss$, or
\item at least one $\alpha_i$ is of type $\dEd$, or
\item there is at least one vertex $x$ such that there is no edge $a$ that is
\begin{itemize}
\item of type $\Ed$ or $\ET$ and $\Phi_+(a)=x$ or
\item of type $\dE$ or $\ET$ and $\Phi_-(a)=x$, or
\end{itemize}
\item for every vertex $x$ for which there exist an edge $b$ of type $\ET$ such that $\Phi_+(b)=x$ or $\Phi_-(b)=x$ there is no edge $a$ that is
\begin{itemize}
\item of type $\Ed$ and $\Phi_+(a)=x$ or
\item of type $\dE$ and $\Phi_-(a)=x$.
\end{itemize}
\end{itemize}
Roughly speaking, the whole tree of vertices connected by thick edges is considered as one vertex for the matter of being source.
$(\mS\Phi^i,\delta_E)$ is the complex spanned by of all sourced $(\Phi,\alpha_1,\dots,\alpha_e)$.
\end{itemize}

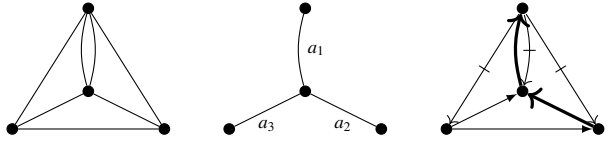
\begin{figure}[H]
$$
\begin{tikzpicture}[baseline=1ex]
 \node[int] (a) at (-1,-.5) {};
 \node[int] (b) at (1,-.5) {};
 \node[int] (c) at (0,0) {};
 \node[int] (d) at (0,1.1) {};
 \draw (a) edge (b);
 \draw (a) edge (c);
 \draw (a) edge (d);
 \draw (b) edge (c);
 \draw (b) edge (d);
 \draw (c) edge[bend left=15] (d);
 \draw (c) edge[bend right=15] (d);
\end{tikzpicture}\quad\quad
\begin{tikzpicture}[baseline=1ex]
 \node[int] (a) at (-1,-.5) {};
 \node[int] (b) at (1,-.5) {};
 \node[int] (c) at (0,0) {};
 \node[int] (d) at (0,1.1) {};
 \draw (a) edge node[below] {$\scriptstyle a_3$} (c);
 \draw (b) edge node[below] {$\scriptstyle a_2$} (c);
 \draw (c) edge[bend left=15] node[right] {$\scriptstyle a_1$} (d);
\end{tikzpicture}\quad\quad
\begin{tikzpicture}[baseline=1ex]
 \node[int] (a) at (-1,-.5) {};
 \node[int] (b) at (1,-.5) {};
 \node[int] (c) at (0,0) {};
 \node[int] (d) at (0,1.1) {};
 \draw (a) edge[-latex] (b);
 \draw (a) edge[-latex] (c);
 \draw (a) edge[crossed,<-] (d);
 \draw (b) edge[very thick,->] (c);
 \draw (b) edge[crossed,<-] (d);
 \draw (c) edge[very thick,->,bend left=15] (d);
 \draw (c) edge[bend right=15,crossed,<-] (d);
\end{tikzpicture}
$$
\caption{\label{fig:complexSigma}
For a core graph $\Phi$ on the left with chosen edges in the middle, an example of a graph in $\mO\Phi^2$ is drawn on the right.}
\end{figure}

It is clear that
\begin{equation}
\left(\mO\Phi^0,\delta_E\right)=
\left(\mO\Phi,\delta_E\right),
\qquad
\left(\mS\Phi^0,\delta_E\right)=
\left(\mS\Phi,\delta_E\right).
\end{equation}
There is a natural inclusion of complexes $\iota_i:\mO\Phi^i\hookrightarrow\mS\Phi^i$ for every $i=0,\dots,v-1$. It extends the map defined on the current bases as follows.
\begin{equation}
\Ess\mapsto \tfrac{1}{2}(\EdE-\dEd),\quad
\Ed\mapsto\Ed,\quad
\dE\mapsto\dE.
\end{equation}

\begin{lemma}
\label{lem:QIf}
For every $i=1,\dots,v-1$ there is a quasi-isomorphism $f^i:\mO\Phi^{i-1}\rightarrow\mO\Phi^i$.
\end{lemma}
\begin{proof}
The essential difference between $\mO\Phi^{i-1}$ and $\mO\Phi^i$ is in the edge $a_i$, it has to be of type $\ET$ in $\mO\Phi^i$, and it is of another type in $\mO\Phi^{i-1}$. Let $f^i:\mO\Phi^{i-1}\rightarrow\mO\Phi^i$ not change other edges, and let $a_i$ be mapped as follows.
\begin{equation}
\label{eq:fi}
\Ess\mapsto 0,\quad
\Ed\mapsto\ET,\quad
\dE\mapsto(-1)^{n+1}\ET.
\end{equation}

The map splits as a direct sum of maps between complexes with fixed types of other edges $f^i:\mO\mE^{fix}\Phi^{i-1}\rightarrow\mO\mE^{fix}\Phi^i$, where superscript $fix$ replaces particular choice of those types. So it is enough to show that  $f^i:\mO\mE^{fix}\Phi^{i-1}\rightarrow\mO\mE^{fix}\Phi^i$ is a quasi-isomorphism. Here, depending on other edge types, condition of being oriented can disallow some possibilities for edge $a_i$ both in $\mO\mE^{fix}\Phi^{i-1}$ and $\mO\mE^{fix}\Phi^i$. We list all cases, showing that the map is quasi-isomorphism in all of them.
\begin{enumerate}
\item There is a cycle that does not include edge $a_i$. In that case both $\mO\mE^{fix}\Phi^{i-1}=\mO\mE^{fix}\Phi^i=0$, so the statement is clear.

\item It is not the case from (1) and there is a path from $\Phi_-(a_i)$ to $\Phi_+(a_i)$ along fixed edge types. Precisely, there is a sequence of vertices $x_0=\Phi_-(a_i),x_1,\dots, x_p=\Phi_+(a_i)$ such that for every $j=0,\dots,p-1$ there exists an edge $a\neq a_i$ that is
\begin{itemize}
\item of type $\Ed$ or $\ET$ and $\Phi_-(a)=x_i$ and $\Phi_+(a)=x_{i+1}$ or
\item of type $\dE$ or $\ET$ and $\Phi_-(a)=x_{i+1}$ and $\Phi_+(a)=x_i$.
\end{itemize}
In that case $\mO\mE^{fix}\Phi^i=0$ because the thick edge at $a_i$ closes the cycle.
In $\mO\mE^{fix}\Phi^{i-1}$ the edge type $\dE$ of $a_i$ is not allowed since it closes the cycle, but $\Ed$ and $\Ess$ are allowed. The differential in $\mO\mE^{fix}\Phi^{i-1}$ is therefore only $\Ess\mapsto\Ed$, making the complex acyclic, so $f^i:\mO\mE^{fix}\Phi^{i-1}\rightarrow 0$ must be a quasi-isomorphism.

\item The same as (2) with $\Phi_-(a_i)$ and $\Phi_+(a_i)$ exchanged gives the analogous result.

\item Neither of the above. In that case all types of edges are allowed in both complexes. It can be easily checked that the map defined with \eqref{eq:fi} is a quasi-isomorphism.
\end{enumerate}
\end{proof}

\begin{lemma}
\label{lem:QIg}
For every $i=1,\dots,v-1$ there is a quasi-isomorphism $g^i:\mS\Phi^{i-1}\rightarrow\mS\Phi^i$.
\end{lemma}
\begin{proof}
Similarly as in the previous lemma, let $g^i:\mS\Phi^{i-1}\rightarrow\mS\Phi^i$ not change edges other than $a_i$, and let $a_i$ be mapped as follows.
\begin{equation}
\label{eq:gi}
\Esss,\dEd,\EdE\mapsto 0,\quad
\Ed\mapsto\ET,\quad
\dE\mapsto(-1)^{n+1}\ET.
\end{equation}

Again, the map splits as a direct sum of maps between complexes with fixed types of other edges $g^i:\mS\mE^{fix}\Phi^{i-1}\rightarrow\mS\mE^{fix}\Phi^i$, where superscript $fix$ replaces particular choice of those types. It is enough to show that  $g^i:\mS\mE^{fix}\Phi^{i-1}\rightarrow\mS\mE^{fix}\Phi^i$ is a quasi-isomorphism. Here, depending on other edge types, condition of being sourced can disallow some possibilities for edge $a_i$ both in $\mS\mE^{fix}\Phi^{i-1}$ and $\mS\mE^{fix}\Phi^i$. We list all cases, showing that the map is quasi-isomorphism in all of them.
\begin{enumerate}
\item There is a source independent on the edge $a_i$. Precisely,  there is a fixed edge of type \Esss or \dEd, or there is a vertex $x\neq\Phi_-(a_i),\Phi_+(a_i)$ such that there is no edge $a$ that is
\begin{itemize}
\item of type $\Ed$ or $\ET$ and $\Phi_+(a)=x$ or
\item of type $\dE$ or $\ET$ and $\Phi_-(a)=x$.
\end{itemize}
In that case all types of edges are allowed in both complexes. It can be easily checked that the map defined with \eqref{eq:gi} is a quasi-isomorphism.

\item It is not the case from (1) and both ends of $a_i$ are sources if $a_i$ is ignored, either as single vertices or as part of the tree of vertices connected by thick edges. Precisely, for $x\in\{\Phi_-(a_i),\Phi_+(a_i)\}$ there is no edge $a\neq a_i$ that is
\begin{itemize}
\item of type $\Ed$ or $\ET$ and $\Phi_+(a)=x$ or
\item of type $\dE$ or $\ET$ and $\Phi_-(a)=x$, or
\end{itemize}
there is an edge $b$ of type $\ET$ such that $\Phi_+(b)=x$ or $\Phi_-(b)=x$ and
for every vertex $y$ for which there is an edge $b$ of type $\ET$ such that $\Phi_+(b)=y$ or $\Phi_-(b)=y$ there is no edge $a$ that is
\begin{itemize}
\item of type $\Ed$ and $\Phi_+(a)=y$ or
\item of type $\dE$ and $\Phi_-(a)=y$.
\end{itemize}
In that case all types of edges are again allowed in both complexes, so the result is the same as in (1).

\item It is not the case from (1) or (2), but $x=\Phi_-(a_i)$ is a source if $a_i$ is ignored, i.e.\ fulfils the condition from (2).
In that case $\Esss$, $\dEd$, $\EdE$ and $\Ed$ are allowed in $\mS\mE^{fix}\Phi^{i-1}$, but $\dE$ is not allowed, making the complex $\mS\mE^{fix}\Phi^{i-1}$ acyclic. It is $\mS\mE^{fix}\Phi^i=0$, so $g^i:\mS\mE^{fix}\Phi^{i-1}\rightarrow 0$ is a quasi-isomorphism.

\item The same as (3) but $x=\Phi_+(a_i)$ gives the analogous result.

\item Neither of the above. In that case in $\mS\mE^{fix}\Phi^{i-1}$ $\Esss$ and $\dEd$
are allowed, while $\EdE$, $\Ed$ and $\dE$ are not, making $\mS\mE^{fix}\Phi^{i-1}$ acyclic. Again $\mS\mE^{fix}\Phi^i=0$, implying the same result.
\end{enumerate}
\end{proof}

The lemmas imply that
\begin{equation}
f:=f^{v-1}\circ\dots\circ f^1:
\mO\Phi
\rightarrow\mO\Phi^{v-1}
\end{equation}
and
\begin{equation}
g:=g^{v-1}\circ\dots\circ g^1:
\mS\Phi
\rightarrow\mS\Phi^{v-1}
\end{equation}
are quasi-isomorphisms.
One easily checks that above maps commute with inclusions, i.e.\
\begin{equation}
\iota_{v-1}\circ f=g\circ\iota_0:\mO\Phi\rightarrow\mS\Phi^{v-1}.
\end{equation}

Let us recall the structure of $\mS\Phi^{v-1}$. It is spanned by base graph representatives with core graph $\Phi$, edges $a_1,\dots a_{v-1}$ have type $\ET$ and
\begin{itemize}
\item at least one edge is of type $\Esss$, or
\item at least one edge is of type $\dEd$, or
\item all edges $\notin\{a_q,\dots,a_{v-1}\}$ are of type $\EdE$.
This is because in $\mS\Phi^{v-1}$ thick edges connect all vertices in a spanning tree, and every edge of type $\Ed$ or $\dE$ will make the tree not source.
\end{itemize}

Note that the shape of the graph $\Phi$ does not matter at all, $\mS\Phi^{v-1}$ can be seen as spanned by the subset of $\left(\sigma_{n+1}^{2,2}\right)^{\times (e-v+1)}$ whose elements fulfil the upper condition.

Now we introduce the sequence of complexes $A^j$ spanned by elements of $\left(\sigma_{n+1}^{1,1}\right)^{\times j}\times\left(\sigma_{n+1}^{2,2}\right)^{\times (e-v+1-j)}$ which satisfy the following condition:
\begin{itemize}
\item at least one term in $\left(\sigma_{n+1}^{2,2}\right)^{\times (e-v+1-j)}$ is $\Esss$, or
\item at least one term in $\left(\sigma_{n+1}^{2,2}\right)^{\times (e-v+1-j)}$ is $\dEd$, or
\item all terms in $\left(\sigma_{n+1}^{2,2}\right)^{\times (e-v+1-j)}$ are $\EdE$
and all terms in $\left(\sigma_{n+1}^{1,1}\right)^{\times j}$ are $\Ess$, i.e.\ it is the element \break $(\Ess,\dots,\Ess,\EdE,\dots,\EdE)$.
\end{itemize}
Clearly, $A^0=\mO\Phi^{v-1}$.

\begin{lemma}
\label{lem:QIp}
For every $j=1,\dots,e-v+1$ there is a quasi-isomorphism $p^j:A^{j-1}\rightarrow A^j$.
\end{lemma}
\begin{proof}
There is a natural projection $\Sigma_{n+1}^2=\left\langle\sigma_{n+1}^{2,2}\right\rangle\rightarrow\Sigma_{n+1}^1=\left\langle\sigma_{n+1}^{1,1}\right\rangle$ that acts
\begin{equation}
\label{eq:projS21}
\Esss\mapsto 0,\quad
\dEd\mapsto -\Ess,\quad
\EdE\mapsto \Ess,\quad
\Ed\mapsto\Ed,\quad
\dE\mapsto\dE.
\end{equation}
We construct a map $p^j:A^{j-1}\rightarrow A^j$ that do not change terms in the first $\left(\sigma_{n+1}^{1,1}\right)^{\times (j-1)}$ nor in the last $\left(\sigma_{n+1}^{2,2}\right)^{\times (e-v+1-j)}$, and on the $j$-th term acts like in \eqref{eq:projS21}.
The map splits as a direct sum of maps between complexes with fixed terms other that $j$-th.
Depending on the choice of other terms, the condition can disallow some possibilities for $j$-th term.
We list all cases, showing that the map is quasi-isomorphism in all of them.
\begin{enumerate}
\item There is a fixed term $\Esss$ or $\dEd$. Than all arrows are allowed in the j-th term, and it is easy to check that the projection \eqref{eq:projS21} is a quasi-isomorphism.
\item It is not the case of (1) and there is a fixed term $\Ed$ or $\dE$. In that case both complexes are $0$ and $0\rightarrow 0$ is trivially a quasi-isomorphism.
\item Neither of the above, i.e.\ all fixed terms are $\EdE$ and $\Ess$. Then the map on $j$-th term must be $\langle\EdE\rangle\rightarrow\langle\Ess\rangle$, being a quasi-isomorphism.
\end{enumerate}
\end{proof}

The lemma implies that
\begin{equation}
p:=p^{e-v+1}\circ\dots\circ p^1:\mS\Phi^{v-1}\rightarrow A^{e-v+1}
\end{equation}
is a quasi-isomorphism.

With this we have defined all complexes on Figure \ref{fig:cd} and proven that all vertical maps are quasi-isomorphisms. To finish the proof of the proposition we will prove that diagonal maps are quasi-isomorphisms too. Check that all three complexes on the diagonal are one-dimensional, and hence there homologies are also one-dimensional.
We need to prove that a representative of the class, that is any non-zero element, is sent to a representative of the class, that is any non-zero element. I.e.\ we just need to prove that diagonal maps are not zero maps.

\begin{lemma}
\label{lem:QIfh}
The map $f\circ h:\left\langle\{\Phi\}\right\rangle\rightarrow\mO\Phi^{v-1}$ is a quasi-isomorphism.
\end{lemma}
\begin{proof}
Both complexes are 1-dimensional, so we need to prove that $f\circ h\neq 0$.
The left-hand side complex has a generator $\Phi$. It is
$$
f\circ h(\Phi)=
f\left(\sum_{x\in V(\Phi)}(v(x)-2)\sum_{\tau\in S(\Phi)}h_{x,\tau}(\Phi)\right).
$$
The map $h_{x,\tau}$ gives edges in $E(\tau)$ type $\Ed$ or $\dE$, and to the other edges type $\Ess$. After that, the map $f=f_1\circ\dots\circ f_{v-1}$ kills all graphs if any of edges $e_1,\dots,a_{v-1}$ has a type $\Ess$. Therefore, $f\circ h_{x,\tau}$ is non-zero only if the tree $\tau$ consist exactly of edges $e_1,\dots,a_{v-1}$. Let us call this tree $T$. So
\begin{equation}
f\circ h(\Phi)=
\sum_{x\in V(\Phi)}(v(x)-2)f\left(h_{x,T}(\Phi)\right).
\end{equation}
From the construction it is clear that $f\left(h_{x,T}(\Phi)\right)$ is $\pm$ the generator of $\mO\Phi^{v-1}$, the graph that has $a_1,\dots,a_{v-1}$ of type $\ET$ and other edges of type $\Ess$. Factor $(v(x)-2)$ is always non-negative and at least once positive, so it is enough to show that the sign $\pm$ in front of the generator is always the same, independent on the choice of vertex $x$. Indeed, we will show that $f\left(h_{x,T}(\Phi)\right)=f\left(h_{v,T}(\Phi)\right)$ for every $x\in V(\Gamma)$.

We can assume that $\Phi$ has a model labelling.
Let there be $r$ edges in the tree $T$ from $v$ to $x$. The graphs $h_{x,T}(\Phi)$ and $h_{v,T}(\Phi)\in\{\Phi\}\times_O\left(\sigma_{n+1}^{1,1}\right)^{\times e}$ are similar, the only differences are:
\begin{enumerate}
\item $r$ edges of type $\Ed$ or $\dE$ between $x$ and $v$ have the opposite directions,
\item there is a sign difference $(-1)^{rn}$ from \eqref{eq:sgnrinv} because $r$ edges in the tree get different directions,
\item the sign difference from getting the graph into model labelling before $h_{x,T}(\Phi)$ is $(-1)^n$ because we only need to switch two vertices, and
\item vertices and edges on the path from $x$ to $v$ along $T$ are labelled differently, giving different relabelling sign from \eqref{eq:sgnRelab}.
\end{enumerate}
After the action of $f$ both graphs are sent to $\pm$ the same graph, while the difference (1) gives a sign difference $(-1)^{r(n+1)}$. So, sign difference from (1) and (2) is always $(-1)^r$, along with the sign from (3) and the relabelling sign difference from (4).
Note that signs from \eqref{def:h} are the same in both cases.
To get the relabelling sign difference it is enough to calculate the sign of relabelling that turns $h'_{x,T}(\Phi)$ to $h'_{v,T}(\Phi)$. We do it separately for even and odd parameter $n$ on Figures \ref{fig:signheven} and \ref{fig:signhodd}.

\begin{figure}[H]
$$
\begin{tikzpicture}[baseline=-3ex]
 \node[int] (a) at (0,0) {};
 \node[int] (b) at (1,0) {};
 \node[int] (c) at (1.5,-.85) {};
 \node at (2,-.85) {$\dots$};
 \node[int] (d) at (2.5,-.85) {};
 \node[int] (e) at (3.5,-.85) {};
 \draw (a) edge node[below] {$e_1$} (b);
 \draw (b) edge node[right] {$e_2$} (c);
 \draw (d) edge node[below] {$e_r$} (e);
 \draw (a) edge (-.5,.85) {};
 \draw (a) edge (-.5,-.85) {};
 \draw (b) edge (1.5,.85) {};
 \draw (c) edge (1,-1.7) {};
 \draw (e) edge (4,0) {};
 \draw (e) edge (4,-1.7) {};
 \node[left] at (a) {$v$};
 \node[right] at (e) {$x$};
\end{tikzpicture}
\qquad\mxto{{h'_{v,T}}}\qquad
\begin{tikzpicture}[baseline=-3ex]
 \node[int] (a) at (0,0) {};
 \node[int] (b) at (1,0) {};
 \node[int] (c) at (1.5,-.85) {};
 \node at (2,-.85) {$\dots$};
 \node[int] (d) at (2.5,-.85) {};
 \node[int] (e) at (3.5,-.85) {};
 \draw (a) edge[-latex] (b);
 \draw (b) edge[-latex] (c);
 \draw (d) edge[-latex] (e);
 \draw (a) edge[-latex] (-.5,.85) {};
 \draw (a) edge[-latex] (-.5,-.85) {};
 \draw (b) edge[-latex] (1.5,.85) {};
 \draw (c) edge[-latex] (1,-1.7) {};
 \draw (e) edge[-latex] (4,0) {};
 \draw (e) edge[-latex] (4,-1.7) {};
 \node[left] at (a) {$v$};
 \node[right] at (b) {$e_1$};
 \node[left] at (c) {$e_2$};
 \node[above] at (d) {$e_{r-1}$};
 \node[right] at (e) {$e_r$};
\end{tikzpicture}
$$
$$
\begin{tikzpicture}[baseline=-3ex]
 \node[int] (a) at (0,0) {};
 \node[int] (b) at (1,0) {};
 \node[int] (c) at (1.5,-.85) {};
 \node at (2,-.85) {$\dots$};
 \node[int] (d) at (2.5,-.85) {};
 \node[int] (e) at (3.5,-.85) {};
 \draw (a) edge node[below] {$e_1$} (b);
 \draw (b) edge node[right] {$e_2$} (c);
 \draw (d) edge node[below] {$e_r$} (e);
 \draw (a) edge (-.5,.85) {};
 \draw (a) edge (-.5,-.85) {};
 \draw (b) edge (1.5,.85) {};
 \draw (c) edge (1,-1.7) {};
 \draw (e) edge (4,0) {};
 \draw (e) edge (4,-1.7) {};
 \node[left] at (a) {$v$};
 \node[right] at (e) {$x$};
\end{tikzpicture}
\qquad\mxto{h'_{x,T}}\qquad
\begin{tikzpicture}[baseline=-3ex]
 \node[int] (a) at (0,0) {};
 \node[int] (b) at (1,0) {};
 \node[int] (c) at (1.5,-.85) {};
 \node at (2,-.85) {$\dots$};
 \node[int] (d) at (2.5,-.85) {};
 \node[int] (e) at (3.5,-.85) {};
 \draw (a) edge[latex-] (b);
 \draw (b) edge[latex-] (c);
 \draw (d) edge[latex-] (e);
 \draw (a) edge[-latex] (-.5,.85) {};
 \draw (a) edge[-latex] (-.5,-.85) {};
 \draw (b) edge[-latex] (1.5,.85) {};
 \draw (c) edge[-latex] (1,-1.7) {};
 \draw (e) edge[-latex] (4,0) {};
 \draw (e) edge[-latex] (4,-1.7) {};
 \node[left] at (a) {$e_1$};
 \node[right] at (b) {$e_2$};
 \node[left] at (c) {$e_3$};
 \node[above] at (d) {$e_r$};
 \node[right] at (e) {$v$};
\end{tikzpicture}
$$
\caption{\label{fig:signheven}
The action of $h_{v,T}$ and $h_{x,T}$ for even $n$, with $r$ edges $e_1,\dots,e_r$ between $v$ and $x$ along $T$.
Relabelling vertices from the label of $h'_{x,T}(\Phi)$ to the label of $h'_{v,T}(\Phi)$ gives the sign $(-1)^r$ because vertices are odd after the action.
Edges are even, so we do not need to consider them.}
\end{figure}
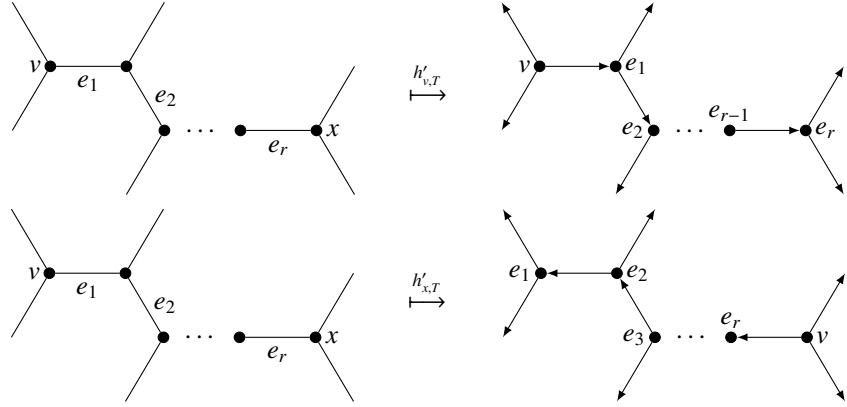

\begin{figure}[H]
$$
\begin{tikzpicture}[baseline=-3ex]
 \node[int] (a) at (0,0) {};
 \node[int] (b) at (1,0) {};
 \node[int] (c) at (1.5,-.85) {};
 \node at (2,-.85) {$\dots$};
 \node[int] (d) at (2.5,-.85) {};
 \node[int] (e) at (3.5,-.85) {};
 \node[right] at (b) {$v_1$};
 \node[left] at (c) {$v_2$};
 \node[above] at (d) {$v_r$};
 \draw (a) edge (b);
 \draw (b) edge (c);
 \draw (d) edge (e);
 \draw (a) edge (-.5,.85) {};
 \draw (a) edge (-.5,-.85) {};
 \draw (b) edge (1.5,.85) {};
 \draw (c) edge (1,-1.7) {};
 \draw (e) edge (4,0) {};
 \draw (e) edge (4,-1.7) {};
 \node[left] at (a) {$v$};
 \node[right] at (e) {$x$};
\end{tikzpicture}
\qquad\mxto{{h'_{v,T}}}\qquad
\begin{tikzpicture}[baseline=-3ex]
 \node[int] (a) at (0,0) {};
 \node[int] (b) at (1,0) {};
 \node[int] (c) at (1.5,-.85) {};
 \node at (2,-.85) {$\dots$};
 \node[int] (d) at (2.5,-.85) {};
 \node[int] (e) at (3.5,-.85) {};
 \draw (a) edge[-latex] node[below] {$v_1$} (b);
 \draw (b) edge[-latex] node[right] {$v_2$} (c);
 \draw (d) edge[-latex] node[below] {$x$} (e);
 \draw (a) edge[-latex] (-.5,.85) {};
 \draw (a) edge[-latex] (-.5,-.85) {};
 \draw (b) edge[-latex] (1.5,.85) {};
 \draw (c) edge[-latex] (1,-1.7) {};
 \draw (e) edge[-latex] (4,0) {};
 \draw (e) edge[-latex] (4,-1.7) {};
\end{tikzpicture}
$$
$$
\begin{tikzpicture}[baseline=-3ex]
 \node[int] (a) at (0,0) {};
 \node[int] (b) at (1,0) {};
 \node[int] (c) at (1.5,-.85) {};
 \node at (2,-.85) {$\dots$};
 \node[int] (d) at (2.5,-.85) {};
 \node[int] (e) at (3.5,-.85) {};
 \node[right] at (b) {$v_1$};
 \node[left] at (c) {$v_2$};
 \node[above] at (d) {$v_{r-1}$};
 \draw (a) edge (b);
 \draw (b) edge (c);
 \draw (d) edge (e);
 \draw (a) edge (-.5,.85) {};
 \draw (a) edge (-.5,-.85) {};
 \draw (b) edge (1.5,.85) {};
 \draw (c) edge (1,-1.7) {};
 \draw (e) edge (4,0) {};
 \draw (e) edge (4,-1.7) {};
 \node[left] at (a) {$v$};
 \node[right] at (e) {$x$};
\end{tikzpicture}
\qquad\mxto{h'_{x,T}}\qquad(-1)^{r+1}
\begin{tikzpicture}[baseline=-3ex]
 \node[int] (a) at (0,0) {};
 \node[int] (b) at (1,0) {};
 \node[int] (c) at (1.5,-.85) {};
 \node at (2,-.85) {$\dots$};
 \node[int] (d) at (2.5,-.85) {};
 \node[int] (e) at (3.5,-.85) {};
 \draw (a) edge[latex-] node[below] {$x$} (b);
 \draw (b) edge[latex-] node[right] {$v_1$} (c);
 \draw (d) edge[latex-] node[below] {$v_{r-1}$} (e);
 \draw (a) edge[-latex] (-.5,.85) {};
 \draw (a) edge[-latex] (-.5,-.85) {};
 \draw (b) edge[-latex] (1.5,.85) {};
 \draw (c) edge[-latex] (1,-1.7) {};
 \draw (e) edge[-latex] (4,0) {};
 \draw (e) edge[-latex] (4,-1.7) {};
\end{tikzpicture}
$$
\caption{\label{fig:signhodd}
The action of $h_{v,T}$ and $h_{x,T}$ for odd $n$, with $r$ edges  and vertices $v_1,\dots,v_{r-1}$ between $v$ and $x$ along $T$.
Relabelling edges from the label of $h'_{x,T}(\Phi)$ to the label of $h'_{v,T}(\Phi)$ gives the sign $(-1)^{r-1}$ because edges are odd after the action.
Vertices are even, so we do not need to consider them.
The sign $(-1)^{r+1}$ written in the second equation is the sign difference already calculated for (2) and (3).}
\end{figure}
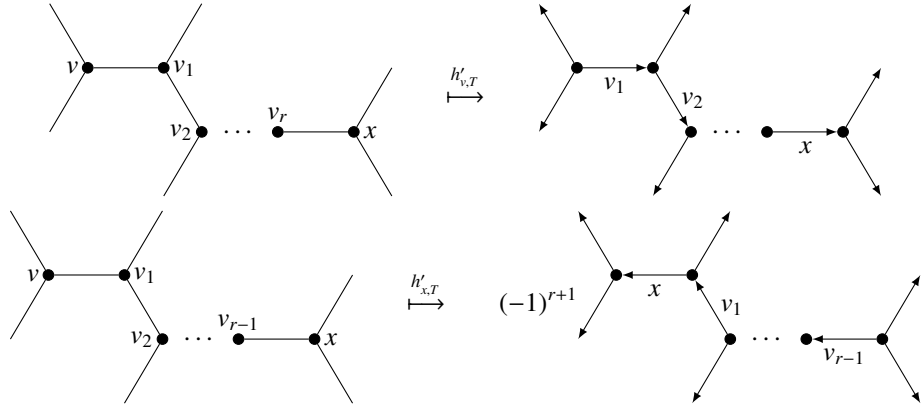

Summed together, the total sign difference between $f\left(h_{x,T}(\Phi)\right)$ and $f\left(h_{v,T}(\Phi)\right)$ for even $n$ is $(-1)^r(-1)^n(-1)^r=1$, and for odd $n$ it is $(-1)^r(-1)^n(-1)^{r-1}=1$. So there is no sign difference and it is always $f\left(h_{x,T}(\Phi)\right)=f\left(h_{v,T}(\Phi)\right)$, and
$$
f\circ h(\Phi)=
f\left(h_{v,T}(\Phi)\right)\sum_{x\in V(\Phi)}(v(x)-2)\neq 0,
$$
as claimed.
\end{proof}

\begin{lemma}
\label{lem:QIpi}
The map $p\circ \iota_{v-1}:\mO\Phi^{v-1}\rightarrow A^{e-v+1}$ is a quasi-isomorphism.
\end{lemma}
\begin{proof}
Both complexes are 1-dimensional, so we need to prove that $p\circ \iota_{v-1}\neq 0$. It is clear from the construction of the maps.
\end{proof}

Here on Figure $\ref{fig:cd2}$ we copy the commutative diagram from Figure \ref{fig:cd} with references to lemmas that show that particular maps are quasi-isomorphisms. Diagram is commutative and listed quasi-isomorphisms connect all depicted complexes, so it implies that all depicted maps are quasi-isomorphisms. Particularly, the first row consist of quasi-isomorphisms, what was to be shown. 

\begin{figure}[H]
$$
\begin{tikzcd}
\left\langle\{\Phi\}\right\rangle
\arrow{r}{h}
\arrow[swap]{dr}{\ref{lem:QIfh}}
& \mO\Phi
\arrow[hookrightarrow]{r}{}
\arrow{d}{\ref{lem:QIf}}
& \mS\Phi
\arrow{d}{\ref{lem:QIg}} \\
& \mO\Phi^{v-1}
\arrow[hookrightarrow]{r}{}
\arrow[swap]{dr}{\ref{lem:QIpi}}
& \mS\Phi^{v-1}
\arrow{d}{\ref{lem:QIp}} \\
& & A^{e-v+1}
\end{tikzcd}
$$
\caption{\label{fig:cd2}
A commutative diagram of complexes. Arrows with reference numbers are quasi-isomorphisms, and they refer to lemmas where it is shown.}
\end{figure}
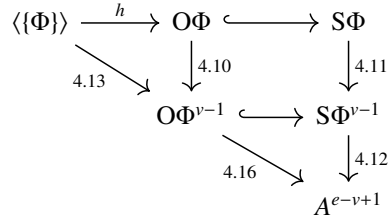
\end{proof}

Two last parts of Theorem \ref{thm:main} now follow directly from Propositions \ref{prop:DOSD} , \ref{prop:SkEq}, \ref{prop:SmallSk} and \ref{prop:main}.


\begin{thebibliography}{10}


\bibitem{AT}
G.~{Arone} and V.~{Turchin}.
\newblock {Graph-complexes computing the rational homotopy of high dimensional analogues of spaces of long knots}.
\newblock {\em Ann. Inst. Fourier} 65(1):1--62, 2015.

\bibitem{FTW}
Benoit Fresse, Victor Turchin and Thomas Willwacher.
\newblock Mapping spaces of the $E_n$ operads.
\newblock In preparation, 2015.

\bibitem{DGC1}
Anton Khoroshkin, Thomas Willwacher and Marko \v Zivkovi\'c.
\newblock Differentials on graph complexes.
\newblock \emph{Adv. Math.} 307:1184--1214, 2017.

\bibitem{DGC2}
Anton Khoroshkin, Thomas Willwacher and Marko \v Zivkovi\'c.
\newblock Differentials on graph complexes II: Hairy graphs.
\newblock \emph{Lett. Math. Phys.} 107: 1781--1797, 2017.

\bibitem{Kont1}
Maxim Kontsevich.
\newblock Formal (non)commutative symplectic geometry.
\newblock \emph{In Proceedings of the I. M. Gelfand seminar 1990--1992}, 173–-188. Birkhauser, 1993.

\bibitem{Kont2}
Maxim Kontsevich.
\newblock Formality Conjecture.
\newblock \emph{Deformation Theory and Symplectic Geometry} 139--156, 1997.
\newblock D.\ Sternheimer et al.\ (eds.).

\bibitem{Merk1}
Sergei Merkulov.
\newblock Multi-oriented props and homotopy algebras with branes.
\newblock arXiv:1712.09268.

\bibitem{Merk2}
Sergei Merkulov.
\newblock Deformation quantization of homotopy algebras with branes.
\newblock (To appear.)


\bibitem{MW}
S.\ Merkulov and T.\ Willwacher.
\newblock Props of ribbon graphs, involutive Lie bialgebras and moduli spaces of curves.
\newblock arXiv:1511.07808.

\bibitem{Tur1}
Victor Turchin.
\newblock Hodge-type decomposition in the homology of long knots.
\newblock {\em J. Topol.}, 3(3):487--534, 2010.

\bibitem{Will}
Thomas Willwacher.
\newblock Deformation quantization and the Gerstenhaber structure on the homology of knot spaces.
\newblock arXiv:1506.07078.

\bibitem{grt}
Thomas Willwacher.
\newblock {M. Kontsevich's graph complex and the Grothendieck-Teichm\"uller Lie
  algebra}.
\newblock \emph{Invent. Math.} 200(3):671--760, 2015.

\bibitem{Poly}
Thomas Willwacher.
\newblock Stable cohomology of polyvector fields.
\newblock arXiv:1110.3762.

\bibitem{oriented}
Thomas Willwacher.
\newblock {The oriented graph complexes}.
\newblock \emph{Commun. Math. Phys.} 334: 1649, 2015.

\bibitem{eulerchar}
T.\ Willwacher and M.\ \v{Z}ivkovi\'c.
\newblock {Multiple edges in M. Kontsevich's graph complexes and computations of the dimensions and Euler characteristics}.
\newblock \emph{Adv. Math.} 272:553--578, 2015.

\bibitem{Multi}
Marko \v{Z}ivkovi\'c.
\newblock Multi-directed graph complexes and quasi-isomorphisms between them I: oriented graphs.
\newblock arXiv:1703.09605.

\end{thebibliography}
\end{document}